\documentclass{amsart}

\usepackage{amsmath,amsthm, amscd, amssymb, amsfonts, mathrsfs}
\usepackage{stmaryrd}
\usepackage{enumitem}

\usepackage[all]{xy}
\usepackage{color}
\usepackage[T1]{fontenc}
\usepackage{mathtools}

 \usepackage{tikz}
\usetikzlibrary{patterns}
\usetikzlibrary{arrows}
\usetikzlibrary{positioning}
\usetikzlibrary{decorations.pathmorphing}
\usetikzlibrary{cd}

\usepackage{multicol}
\usepackage{enumitem}

\usepackage{cite}

\usepackage[plainpages=false,pdfpagelabels,
hypertexnames=false]{hyperref}

\allowdisplaybreaks

\newcommand{\ku}{\Bbbk}

\newcommand\fD{\mathsf{D}}

\newcommand\fP{\mathsf{P}}

\newcommand\fR{\mathsf{R}}

\newcommand{\Z}{{\mathbb Z}}

\newcommand{\D}{{\mathbb D}}

\newcommand{\End}{\operatorname{End}}

\newcommand{\Ext}{\operatorname{Ext}}

\newcommand\Hom{\operatorname{Hom}}

\newcommand\id{\operatorname{id}}

\numberwithin{equation}{section}
\newtheorem{lema}{Lemma}[section]
\newtheorem{theorem}[lema]{Theorem}
\newtheorem{cor}[lema]{Corollary}

\newtheorem{prop}[lema]{Proposition}

\newtheorem{question-app}{Question}

\theoremstyle{definition}
\newtheorem{definition}[lema]{Definition}

\theoremstyle{remark}
\newtheorem{example}[lema]{Example}

\newtheorem{rmk}[lema]{Remark}

\newcommand{\uv}{{\underline{v}}}
\newcommand{\uw}{{\underline{w}}}
\newcommand{\ux}{{\underline{x}}}
\newcommand{\uy}{{\underline{y}}}
\newcommand{\uz}{{\underline{z}}}

\newcommand{\uH}{\underline{H}}

\newcommand{\ustar}{\,\underline{\star}\,}

\newcommand{\oDelta}{\overline{\Delta}}
\newcommand{\onabla}{\overline{\nabla}}

\newcommand{\h}{\mathfrak{h}}

\newcommand{\rD}{\mathscr{D}}
\newcommand{\rI}{\mathscr{I}}

\newcommand{\rF}{\mathscr{F}}
\newcommand{\rG}{\mathscr{G}}

\newcommand{\rT}{\mathscr{T}}
\newcommand{\rP}{\mathscr{P}}
\newcommand{\rL}{\mathscr{L}}

\newcommand{\qq}{\mathbin{\!/\mkern-5mu/\!}}

\newcommand{\LL}{\mathbb{LL}}

\newcommand\BE{\mathsf{BE}}
\newcommand\RE{\mathsf{RE}}
\newcommand\Tilt{\mathsf{Tilt}}

\newcommand{\be}{\mathbf{e}}

\newcommand{\f}{\mathbf{f}}

\newcommand{\kF}{\mathfrak{F}}

\newcommand{\For}{\mathsf{For}}

\newcommand{\bk}{\Bbbk}
\newcommand{\fgen}{\mathrm{fg}}
\newcommand{\Mod}{\mathrm{Mod}}
\newcommand{\Free}{\mathrm{Free}}

\newcommand{\Kb}{K^{\mathrm{b}}}
\newcommand{\Db}{D^{\mathrm{b}}}

\newcommand{\DiagBS}{\mathscr{D}_{\mathrm{BS}}}
\newcommand{\oDiagBS}{\overline{\mathscr{D}}_{\mathrm{BS}}}
\newcommand{\Diag}{\mathscr{D}}
\newcommand{\oDiag}{\overline{\mathscr{D}}}
\newcommand{\BS}{\mathrm{BS}}
\newcommand{\oB}{\overline{B}}
\newcommand{\oD}{\overline{\Delta}}
\newcommand{\oN}{\overline{\nabla}}
\newcommand{\coH}{\mathsf{H}}

\makeatletter
\def\Tenint{\@ifnextchar_{\@Tenintsub}{\@Tenintnosub}}
\def\@Tenintsub_#1{\mathchoice{\mathbin{\underline{\mathop{\otimes}}}_{#1}}%
  {\underline{\otimes}_{#1}}{\underline{\otimes}_{#1}}{\underline{\otimes}^L_{#1}}}
\def\@Tenintnosub{\mathbin{\underline{\mathop{\otimes}}}}
\makeatother

\makeatletter
\def\lotimes{\@ifnextchar_{\@lotimessub}{\@lotimesnosub}}
\def\@lotimessub_#1{\mathchoice{\mathbin{\mathop{\otimes}^L}_{#1}}%
  {\otimes^L_{#1}}{\otimes^L_{#1}}{\otimes^L_{#1}}}
\def\@lotimesnosub{\mathbin{\mathop{\otimes}^L}}
\makeatother

\title{Mixed perverse sheaves on flag varieties of Coxeter groups}

\author{Pramod N. Achar}
\address{Department of Mathematics\\
  Louisiana State University\\
  Baton Rouge, LA 70803\\
  U.S.A.}
\email{pramod@math.lsu.edu}

\author{Simon Riche}
\address{Universit\'e Clermont Auvergne, CNRS, LMBP, F-63000 Clermont-Ferrand, Fran\-ce.
}
\email{simon.riche@uca.fr}

\author{Cristian Vay}
\address{Universidad Nacional de C\'ordoba, Facultad de Matem\'atica, Astronom\'ia, F\'isica y Computaci\'on, CIEM -- CONICET, C\'ordoba, Argentina.}
\email{vay@famaf.unc.edu.ar}

\thanks{P.A. was partially supported by NSF Grant No. DMS-1500890. S.R. was partially supported by ANR Grant No.~ANR-13-BS01-0001-01. This project has received funding from the European Research Council (ERC) under the European Union's Horizon 2020 research and innovation programme (grant agreement No. 677147). The work of C.V. was done during a research stay at the Universit\'e Clermont Auvergne supported by CONICET. He also was partially supported by Secyt (UNC), FONCyT PICT 2016-3957 and MathAmSud project GR2HOPF}


\begin{document}

\begin{abstract}
In this paper we construct an abelian category of ``mixed perverse sheaves'' attached to any realization of a Coxeter group, in terms of the associated Elias--Williamson diagrammatic category. This construction extends previous work of the first two authors, where we worked with parity complexes instead of diagrams, and we extend most of the properties known in this case to the general setting. As an application we prove that the split Grothendieck group of the Elias--Williamson diagrammatic category is isomorphic to the corresponding Hecke algebra, for any choice of realization.
\end{abstract}

\maketitle

\section{Introduction}

\subsection{Categorifications of Hecke algebras}
\label{ss:intro-categorifications}

Let $(W,S)$ be a Coxeter system, and let $\mathcal{H}_W$ be the associated Hecke algebra. When $W$ is crystallographic, i.e.~is the Weyl group of a Kac--Moody group $G$, a fact of fundamental importance, going back to 1980, is the existence of a remarkable geometric categorification of $\mathcal{H}_W$: it can be realized as the split Grothendieck group of the additive monoidal category of $B$-equivariant semisimple complexes (with complex coefficients) on the flag variety $G/B$ of $G$, where $B \subset G$ is the Borel subgroup; see~\cite{kl2,springer}.\footnote{The papers~\cite{kl2} and~\cite{springer} only mention finite or affine Kac--Moody groups. However, thanks to the subsequent development of the general theory of Kac--Moody groups~\cite{kumar,tits}, their methods now apply in this generality.} The main point of this categorification is that it realizes the Kazhdan--Lusztig basis of $\mathcal{H}_W$ as the classes of simple perverse sheaves.

In the 2000s, this categorification was generalized in two different directions:
\begin{enumerate}
 \item
 In~\cite{soergel-bim}, Soergel showed that semisimple complexes of flag varieties can be replaced by \emph{Soergel bimodules}, thereby providing a categorification of $\mathcal{H}_W$ for \emph{any} Coxeter group.
 \item
 In~\cite{jmw}, Juteau--Mautner--Williamson introduced the \emph{parity complexes}, which provide the appropriate replacement for semisimple complexes when we take coefficients in an arbitrary field (this leads naturally to the notion of \emph{$p$-canonical bases}).
\end{enumerate}

These two generalizations were recently united by the introduction of the \emph{Elias--Williamson diagrammatic category}~\cite{ew}, a certain monoidal category attached to any Coxeter group equipped with a ``realization.''  For certain realizations (coming from ``reflection faithful representations''), this category is equivalent to the category of Soergel bimodules.  On the other hand, for realizations constructed from Kac--Moody root data, one recovers the corresponding category of parity complexes on the associated flag variety. (This result was suggested in~\cite{ew} and~\cite{jmw}, and formally proved in~\cite[Part~3]{rw}.) Let us note that the ``reflection faithful'' requirement is rather restrictive, justifying the interest of a construction avoiding this condition.

\subsection{Triangulated categories}

The categorifications considered above take us away from the very comfortable world of perverse sheaves. The main goal of the present paper is to explain how perverse sheaves can be reintroduced into the picture. 
This paper draws inspiration from~\cite{modrap2}, which (in the setting of parity complexes on flag varieties) introduced the notions of ``mixed derived category'' and ``mixed perverse sheaves.''  These notions have since found important applications in modular representation theory; see in particular~\cite{modrap2, arider, MaR, prinblock, mkdkm}.

The first step is to embed the diagrammatic category in a suitable triangulated category.  This was done by Makisumi, Williamson, and the first two authors in~\cite{amrw1}.  That paper defines the ``biequivariant'' derived category $\BE(\h,W)$ attached to a Coxeter group $W$ and a realization $\h$ as the bounded homotopy category of the Elias--Williamson category. The same paper also defines the ``right-equivariant'' derived category $\RE(\h,W)$, which plays the role of the \emph{$B$-constructible} derived category of $G/B$ in the usual picture.

\subsection{Perverse sheaves}
\label{ss:intro-perverse}

In the present paper, we build on this approach and construct the ``perverse t-structure'' on $\BE(\h,W)$ and $\RE(\h,W)$.  One would like to follow the model of~\cite{modrap2}, but that paper exploits the fact that parity complexes are already defined in terms of sheaves on some topological space, where it makes sense to restrict to or push forward from a locally closed subspace.  Thus, one key step in the present paper is to understand the correct analogue of ``locally closed subspace'' in the diagrammatic setting. The solution, explained in Section~\ref{sec:categories-subsets} and inspired by~\cite{arider1}, is to work with certain ``naive'' subquotients of the diagrammatic category.

The proof that these subquotients have the appropriate behaviour relies on some properties of the \emph{double leaves basis} for morphism spaces, introduced by Libedinsky~\cite{libedinsky} for Soergel bimodules, and studied in the context of the diagrammatic category in~\cite[Sections~6--7]{ew}. From our point of view, this study provides another illustration of the power of these methods.

Once we have made the correct definitions, we will construct a ``recollement'' formalism for these categories (following essentially the same ideas as in~\cite{modrap2}), and use it to define the desired t-structure.

\subsection{Standard and costandard objects}
\label{ss:intro-standard-costandard}

An important property of $B$-equiva\-riant perverse sheaves on $G/B$ is that the \emph{standard} and \emph{costandard} objects (the $*$- and $!$-extensions of the constant perverse sheaves on Bruhat strata) are perverse sheaves. In the usual topological context, this property follows from the fact that the embeddings of these strata are affine morphisms. For the mixed derived categories of~\cite{modrap2}, a different argument was needed.  The proof given there is based on the study of push-forward to and pullback from partial flag varieties.

In the context of the present paper, we have no analogue of sheaves on partial flag varieties,\footnote{A definition of such a category would require a diagrammatic version of the ``singular Soergel bimodules'' of~\cite{williamson-sing}; no definition of such objects is available at the moment.} so some new ideas are needed (which, once again, rely to some extent on the properties of the light leaves basis).  The proof that standard and costandard objects are perverse in the diagrammatic setting appears in Section~\ref{sec:BE t structure}.

\subsection{Some other properties}

Let $\fP^{\BE}(\h,W)$ and $\fP^{\RE}(\h,W)$ denote the hearts of the perverse t-structures on $\BE(\h,W)$ and $\RE(\h,W)$, respectively.  These categories share many properties with their traditional counterparts. In particular, we prove that:
\begin{enumerate}
\item
In the case of field coefficients, the simple objects in the abelian categories $\fP^{\BE}(\h,W)$ and $\fP^{\RE}(\h,W)$ can be described in terms of $!*$-extensions; see~\S\ref{ss:simple-perverse} and~\S\ref{ss:RE-field}.
 \item The forgetful functor $\fP^{\BE}(\h,W) \to \fP^{\RE}(\h,W)$ is fully faithful; see Proposition~\ref{prop:perv-equiv-const}.
 \item If $\bk$ is a field, the category $\fP^{\RE}(\h,W)$ has a natural structure of highest weight category; see Theorem~\ref{thm:hw}.
 \item
 If $\bk$ is a field and $W$ is finite, one can construct a ``Ringel duality'' exchanging projective and tilting objects in $\fP^{\RE}(\h,W)$; moreover, the indecomposable tilting object associated with the longest element in $W$ is both projective and injective; see Section~\ref{sec:Ringel}.
\end{enumerate}

\subsection{Applications}
\label{ss:application}

One classical motivation for studying mixed perverse sheaves on flag varieties (with complex coefficients) is that they provide a ``mixed version'' of (a regular block of) the Bernstein--Gelfand--Gelfand category $\mathcal{O}$ associated with a semisimple complex Lie algebra; see~\cite{bgs, soergel-ICM}. In this spirit, $\fP^\RE(\h,W)$ may be thought of as a ``generalized mixed category $\mathcal{O}$'' attached to $W$ and $\h$. 

As a more concrete application of our results, we prove that for any realization of $W$, the split Grothendieck group of the Elias--Williamson diagrammatic category is isomorphic to the Hecke algebra $\mathcal{H}_W$. (Note that in~\cite{ew} this result was proved only in the case the base ring $\bk$ is a field or a complete local ring.) This application illustrates the interest of our formalism in the study of the diagrammatic category, beyond the simple game of defining perverse objects.

One of the main results of~\cite{mkdkm} is that when the realization $\h$ comes from a Kac--Moody root datum, there is an equivalence of triangulated categories between $\RE(\h,W)$ and $\RE(\h^*,W)$, known as \emph{Koszul duality}.  The main reason for the restriction to the Kac--Moody setting is that some of the arguments make use of the perverse t-structure from~\cite{modrap2}.  We expect that the methods developed in the present paper will allow one to drop this restriction.

\subsection{Relation with previous work}

As mentioned already, the idea of using the recollement formalism in this kind of setting comes from~\cite{modrap2}.
In~\cite{makisumi}, Makisumi has shown how to adapt the constructions of~\cite{modrap2} to the setting of \emph{sheaves on moment graphs}. In general, this notion, which arose as a kind of combinatorial model for torus-equivariant geometry, takes us away from the world of Coxeter groups, but it overlaps with the results of the present paper in the following special situation: for a Coxeter group equipped with a reflection faithful representation, the category of Soergel bimodules (and hence the Elias--Williamson diagrammatic category) is equivalent to the category of sheaves on the Bruhat moment graph. In this setting, Makisumi's constructions and ours are equivalent. Under the further assumption that Soergel's conjecture holds for the representation under consideration, this t-structure can also be defined purely in terms of Soergel bimodules; see~\cite[Remark~5.7]{makisumi}.
Because moment graphs are closer to geometry (both in spirit and because of the existence of moment graphs modeling partial flag varieties), the arguments in~\cite{makisumi} avoid some of the difficulties mentioned in~\S\S\ref{ss:intro-perverse}--\ref{ss:intro-standard-costandard}.

Separately, a different approach to defining a ``category $\mathcal{O}$'' for a general Coxeter group was proposed by Fiebig in~\cite{fiebig} (in terms of sheaves on the Bruhat moment graph), and studied further by Abe in~\cite{abe}. Compared to their point of view, ours is ``Koszul dual'': in their picture the indecomposable Soergel bimodules correspond to projective objects, whereas for us they correspond to ``parity objects'' (i.e.~semisimple complexes when $\bk$ is a field of characteristic $0$).

\subsection{Contents}

Section~\ref{sec:prelim} contains notation and conventions related to graded modules, and Section~\ref{sec:EW categories} contains background on the Elias--Williamson diagrammatic category, and on the categories $\BE(\h,W)$ and $\RE(\h,W)$.  In Section~\ref{sec:categories-subsets}, we study the diagrammatic analogues of parity complexes on locally closed subsets of the flag variety.  This is needed in order to formulate the recollement theorem, which is proved in Section~\ref{sec:BE}.  Next, Section~\ref{sec:standard-costandard} is devoted to the study of standard and costandard objects.  This section also contains the proof of the categorification result mentioned in~\S\ref{ss:application}.

The definition and some basic properties of the perverse t-structure on $\BE(\h,W)$ appear in Section~\ref{sec:BE t structure}.  In Section~\ref{sec:field}, we specialize to the case of field coefficients.  Much of the work in this section is aimed at understanding the composition factors of standard and costandard perverse objects.

In Section~\ref{sec:RE-category}, we turn our attention to $\RE(\h,W)$.  Many statements carry over from $\BE(\h,W)$, but there are two new results here: one about the full faithfulness of the forgetful functor, and another about the highest weight structure on $\fP^\RE(\h,W)$ (for field coefficients).  One may then ask what the Ringel dual of this highest weight category is.  We conclude the paper in Section~\ref{sec:Ringel} with a proof that $\fP^\RE(\h,W)$ is self-Ringel-dual.

\subsection{Acknowledgements}

We thank Geordie Williamson for useful discussions, and in particular for suggesting the application mentioned in~\S\ref{ss:application}. We also thank Shotaro Makisumi for helpful comments on a draft of this paper.

\section{Preliminaries}\label{sec:prelim}

\subsection{Graded categories}
\label{ss:graded-cat}

Let $\bk$ be a commutative ring, and let $\mathscr{A}$ be a (small) $\bk$-linear category which is enriched over $\Z$-graded $\bk$-modules. Recall that this means that for any $X,Y$ in $\mathscr{A}$ the set of morphisms from $X$ to $Y$ in $\mathscr{A}$ is a graded $\bk$-module
\[
\Hom^\bullet_{\mathscr{A}}(X,Y) = \bigoplus_{n \in \Z} \Hom^n_{\mathscr{A}}(X,Y),
\]
and that composition is defined by morphisms of graded $\bk$-modules (which implies that identity morphisms have degree $0$). To such a category one can attach a category $\mathscr{A}^\circ$ whose objects are symbols $X(n)$ where $X$ is an object of $\mathscr{A}$ and $n \in \Z$, and whose morphisms are defined by
\[
\Hom_{\mathscr{A}^{\circ}} \bigl( X(n), Y(m) \bigr) := \Hom_{\mathscr{A}}^{m-n}(X,Y)
\]
(with the composition defined in the obvious way). This category admits a natural autoequivalence $(1)$ sending the object $X(n)$ to $X(n+1)$; we will denote its $j$-th power by $(j)$ for $j \in \Z$. Moreover, each orbit of the group $\{(j) : j \in \Z\}$ on the set of objects of $\mathscr{A}^\circ$ admits a ``distinguished'' representative $X(0)$.

On the other hand, let $\mathscr{B}$ be a (small) $\bk$-linear category endowed with an auto\-equivalence $(1)$ (whose $j$-th power will be denoted $(j)$) and a set of representatives of the orbits of $\{(j) : j \in \Z\}$ on the set of objects of $\mathscr{B}$. Then
one can define a category $\mathscr{B}^{\mathbb{Z}}$ enriched over graded $\bk$-modules as follows. The objects of $\mathscr{B}^{\Z}$ are the representatives considered above, and the morphisms are defined by
\[
\Hom^\bullet_{\mathscr{B}^\Z}(X,Y) := \bigoplus_{n \in \Z} \Hom_{\mathscr{B}}(X,Y(n)).
\]

It is not difficult to check that the assignments $\mathscr{A} \mapsto \mathscr{A}^\circ$ and $\mathscr{B} \mapsto \mathscr{B}^{\Z}$ are inverse to each other, in the sense that there exist canonical equivalences
\[
(\mathscr{A}^\circ)^\Z \cong \mathscr{A} \quad \text{and} \quad (\mathscr{B}^{\Z})^\circ \cong \mathscr{B}
\]
of categories enriched over graded $\bk$-modules and of $\bk$-linear categories, respectively. For this reason, in the body of the paper we will sometimes not be very careful about the distinction between these points of view, and write e.g.~$\Hom^\bullet_{\mathscr{B}}(M,N)$ for $\bigoplus_{n \in \Z} \Hom_{\mathscr{B}}(M,N(n))$.

\subsection{Tensor products with \texorpdfstring{$R$}{R}-modules}
\label{ss:tensor-product}

Let $\bk$ and $\mathscr{B}$ be as in~\S\ref{ss:graded-cat}.
We let also $R$ be a commutative $\Z$-graded $\bk$-algebra, and assume that $R$ ``acts'' on the objects of the category $\mathscr{B}^\Z$ in the sense that for any $M$ in $\mathscr{B}$ and $r \in R^n$ we have morphisms
\[
\varrho^M_r : M \to M(n), \quad \lambda^M_r : M \to M(n)
\]
such that
\[
\varrho^{M(n)}_{r'} \circ \varrho^M_r = \varrho^M_{rr'} \quad \text{and} \quad \lambda^{M(n)}_{r'} \circ \lambda^M_r = \lambda^M_{r'r}
\]
for $r \in R^n$ and $r' \in R^m$, and
which satisfy
\[
\varrho^{M(1)}_r=\varrho^M_r(1), \quad \lambda^{M(1)}_r=\lambda^M_r(1)
\]
for any $M$ in $\mathscr{B}$ and $r \in R^n$,
\[
 \varrho^{N}_r \circ f = \bigl( f(n) \bigr) \circ \varrho^M_r, \quad \lambda^{N}_r \circ f = \bigl( f(n) \bigr) \circ \lambda^M_r
\]
for any $M,N$ in $\mathscr{B}$, $r \in R^n$ and $f \in \Hom_{\mathscr{B}}(M,N)$, and finally that if $r$ is in the image of $\bk$ in $R^0$ and $M$ is in $\mathscr{B}$, then
$\varrho^M_r = \lambda^M_r$ is the action given by the $\bk$-linear structure on $\mathscr{B}$.

If $X$ is in $\mathscr{B}$ and if $M$ is a $\Z$-graded left $R$-module which is free of finite rank, then we define $X \Tenint_R M$ as the object representing the functor
\[
Y \mapsto \bigl( \Hom^\bullet_{\mathscr{B}}(Y,X) \otimes_R M \bigr)^0,
\]
where the superscript ``$0$'' means the degree-$0$ part, and where the right action of $R$ on $\Hom^\bullet_{\mathscr{B}}(Y,X)$ is defined by 
$f \cdot r = \bigl(f(n)\bigr) \circ \lambda^Y_r = \lambda^{X(m)}_r \circ f$
for $f \in \Hom^m_{\mathscr{B}}(Y,X)$ and $r \in R^n$.
Then we have a natural isomorphism
\begin{equation}
\label{eqn:Hom-tensor}
\Hom^\bullet_{\mathscr{B}}(Y, X \Tenint_R M) \cong \Hom^\bullet_{\mathscr{B}}(Y,X) \otimes_R M.
\end{equation}

In practice, any choice of a graded basis $(e_i)_{i \in I}$ of $M$ as a left $R$-module defines an identification $X \Tenint_R M \cong \bigoplus_{i \in I} X(-\deg(e_i))$. Moreover, if 
$(f_j)_{j \in J}$ is another graded basis of $M$ then there exist unique homogeneous coefficients $a_{i,j} \in R$ such that $e_i = \sum_j a_{i,j} \cdot f_j$ for any $i,j$, and the matrix $\left( \lambda_{a_{ij}}^{X(-\deg f_j)} \right)_{\substack{j \in J \\ i \in I}}$ gives an isomorphism
\[
\bigoplus_{i \in I} X(-\deg(e_i))\xrightarrow{\sim}\bigoplus_{j \in J} X(-\deg(f_j)).
\]
The morphisms $\lambda_r^{X \Tenint_R M}$ and $\varrho_r^{X \Tenint_R M}$ are induced in the natural way by $\lambda_r^X$ and $\varrho_r^X$ respectively.

Now, let $X,Y$ be in $\mathscr{B}$. We consider the left $R$-action on $\Hom^\bullet_{\mathscr{B}}(X,Y)$ given by $r \cdot f = \lambda^{Y(m)}_r \circ f$ for $f \in \Hom^m_{\mathscr{B}}(X,Y)$ and $r \in R^n$. (In other words we consider the same action as before, but now considered as a \emph{left} action.) 
We assume that this action makes $\Hom^\bullet_{\mathscr{B}}(X,Y)$ a graded free left $R$-module. We claim that in this situation there exists a canonical morphism
\begin{equation}
\label{eqn:canonical-morph}
X \Tenint_R \Hom^\bullet_{\mathscr{B}}(X,Y) \to Y.
\end{equation}
In fact, if $(\varphi_i)_{i \in I}$ is a graded basis of the left $R$-module $\Hom^\bullet_{\mathscr{B}}(X,Y)$, then this choice identifies the left-hand side with $\bigoplus_{i \in I} X 
(-\deg(\varphi_i))$, and~\eqref{eqn:canonical-morph} can be defined as $\bigoplus_{i \in I} \varphi_i(-\deg(\varphi_i))$. It can be easily checked that this morphism does not depend on the choice of basis, and hence is indeed canonical. For any $Z$ in $\mathscr{B}$, the induced morphism
\[
 \Hom^{\bullet}_{\mathscr{B}} \bigl( Z,X \Tenint_R \Hom^\bullet_{\mathscr{B}}(X,Y) \bigr) \to \Hom^{\bullet}_{\mathscr{B}}(Z,Y)
\]
identifies, via~\eqref{eqn:Hom-tensor}, with the morphism induced by composition in $\mathscr{B}$.

\subsection{Derived category and free modules}
\label{ss:rings}

For some results in this paper we will impose
the following assumptions on our (commutative) ``base ring'' $\bk$:
\begin{enumerate}
\item
\label{it:ass-1}
$\bk$ is an integral domain;
\item
\label{it:ass-2}
$\bk$ is Noetherian and of finite global dimension;
\item
\label{it:ass-3}
every projective finitely generated $\bk$-module is free.
\end{enumerate}
Here Assumption~\eqref{it:ass-1} is needed in order to apply the results of~\cite{ew}.\footnote{This assumption is not explicit in~\cite{ew} but, as noted in particular in~\cite[Footnote on p.~10]{amrw1}, it is in fact needed.} Assumption~\eqref{it:ass-2} is standard, and ensures that the bounded derived category of finitely generated $\bk$-module has favorable behavior (and similarly for graded modules, and for rings of polynomials with coefficients in $\bk$). Finally, assumption~\eqref{it:ass-3} allows us to describe an appropriate derived category in terms of free modules; see Lemma~\ref{lem:DbMod-KbFree} below. Of courses, these assumptions are satisfied if $\bk$ is a field, or the ring of integers in a finite extension of $\mathbb{Q}_p$, or a finite localization of $\Z$. (These are the typical examples the reader can keep in mind.) Assumption~\eqref{it:ass-3} is also known to hold when $\bk$ is local, see~\cite[Theorem~2.5]{matsumura}. (Note that here we only need the ``trivial'' special case of Kaplansky's theorem when the module is of finite type.)

So, in this subsection we assume that $\ku$ satisfies properties~\eqref{it:ass-2}--\eqref{it:ass-3} above.
We let $V$ be a left graded $\bk$-module which is free of finite rank, and concentrated in positive degrees. Then we denote by $R$ the symmetric algebra of $V$, which we consider as a graded $\bk$-algebra. We will denote by $\Mod^{\fgen,\Z}(R)$ the abelian category of finitely generated graded left $R$-modules, and by $\Free^{\fgen,\Z}(R)$ the full subcategory whose objects are the graded free finitely generated graded left $R$-modules.

\begin{lema}
\label{lem:DbMod-KbFree}
The natural functor
\[
\Kb \Free^{\fgen,\Z}(R) \to \Db \Mod^{\fgen,\Z}(R)
\]
is an equivalence of triangulated categories.
\end{lema}

\begin{proof}
Since $\bk$ has finite global dimension, the same property holds for $R$. Hence any bounded complex of graded $R$-modules is quasi-isomorphic to a bounded complex of projective graded $R$-modules, and to conclude it suffices to prove that any finitely generated projective graded $R$-module is in fact graded free. However, if $M$ is a finitely generated projective graded $R$-module, then $\bk \otimes_R M$ is a finitely generated projective graded $\bk$-module (where $\bk$ is concentrated in degree $0$, and $R$ acts on $\bk$ via the quotient $R/V \cdot R =\bk$), and hence is graded free by Assumption~\eqref{it:ass-3}. Then we deduce that $M$ is graded free by the graded Nakayama lemma.
\end{proof}

\subsection{Terminology}

In the body of the paper we will use the following terminology. If $(X,\preceq)$ is a poset, we will say that a subset $Y \subset X$ is \emph{closed} if for any $x,x' \in X$ with $x' \in Y$ and $x \preceq x'$ we have $x \in Y$. A subset $Z \subset X$ will be called \emph{open} if $X \smallsetminus Z$ is closed. Finally we will say that $Y \subset X$ is \emph{locally closed} if it is the intersection of an open and a closed subset, or equivalently if $Y$ is open in $\overline{Y}=\{x \in X \mid \exists y \in Y \text{ such that } x \preceq y\}$, or equivalently if $Y$ is closed in $X \smallsetminus (\overline{Y} \smallsetminus Y)$. A basic observation that we will use repeatedly is that if $x \in X$, then $x$ is minimal for $\preceq$ if and only if $\{x\}$ is a closed subset of $X$, and $x$ is maximal for $\preceq$ if and only if $\{x\}$ is an open subset of $X$.

In this context, if $x \in X$ we will write $\{\preceq x\}$, resp.~$\{\prec x\}$ for $\{z \in X \mid z \preceq x\}$, resp.~$\{z \in X \mid z \prec x\}$; these subsets are closed in $X$.

\section{The Elias--Williamson diagrammatic category}
\label{sec:EW categories}

From now on we let $\bk$ be an integral domain.

\subsection{Notation and terminology regarding Coxeter systems}

For the rest of this paper we fix a Coxeter system $(W, S)$ with $S$ finite. Then $W$ is equipped with the Bruhat order $\leq$ and the length function $\ell$.

A word $\uw$ in $S$ will be called an {\it expression}. The length $\ell(\uw)$ of an expression $\uw$ is the number of letters in this word. We will by denote $\pi(\uw)$ the corresponding element in $W$ (obtained as the product in $W$ of the letters of $\uw$); then we will say that $\uw$ is \emph{an expression for} (or that $\uw$ \emph{expresses}) $\pi(\uw)\in W$. Recall also that an expression $\uw$ is said to be \emph{reduced} if $\ell(\uw)=\ell(\pi(\uw))$. 

If $\uw=(s_{1},\ldots,s_n)$ is an expression, a \emph{subexpression} of $\uw$ is defined to be a sequence $\mathbf{e}=(e_1,\ldots, e_n)$ of $0$'s and $1$'s. Such a datum determines an expression $\uv=(s_{i_1},\ldots, s_{i_m})$ where $1\leq i_1<\cdots <i_m\leq n$ are the indices such that $\{i_1, \ldots, i_m\} = \{i \in \{1, \ldots, n \} \mid e_i=1\}$.
In a minor abuse of language, we will also say that $\mathbf{e}$ expresses $\pi(\uv)$.
If $x \in W$, we will denote by $M(\uw,x)$ the set of subexpressions of $\uw$ expressing $x$.

With this terminology, the Bruhat order on $W$ can be described as follows: if $\uw$ is a reduced expression for $w \in W$, then $v\le w$ if and only if 
$v$ is expressed by some subexpression of $\uw$.

\subsection{The Elias--Williamson category}
\label{ss:EW-category}

Let $\h=(V, \{\alpha_s^\vee:s\in S\}, \{\alpha_s:s\in S\})$ be a balanced realization of $(W,S)$ over $\ku$ which satisfies Demazure surjectivity in the sense of~\cite{ew} (see also~\cite[\S 2.1]{amrw1}). In particular, $V$ is a free $\bk$-module of finite rank, $\{\alpha_s^\vee:s\in S\}$ is a subset of $V$, and $\{\alpha_s:s\in S\}$ is a subset of $V^*:=\Hom_\bk(V,\bk)$.

We let $R$ be the
symmetric algebra of $V^*$, considered as a graded ring with $V^*$ in degree 2. Following Elias--Williamson~\cite{ew} (see also~\cite[\S 2.2]{amrw1}), to $(W,S)$ and $\h$ we associate a $\bk$-linear monoidal category as follows.
First, one defines a $\bk$-linear monoidal category $\widetilde{\mathscr{D}}_{\mathrm{BS}}(\h,W)$ enriched over graded $\bk$-modules with:
\begin{itemize}
\item objects the symbols $B_{\uw}$ for $\uw$ an expression, with the monoidal product defined by $B_{\uv} \star B_{\uw} = B_{\uv\uw}$;
\item morphisms generated (under composition, monoidal product and $\bk$-linear combinations) by the following ``elementary'' morphisms:
\begin{enumerate}
\item
for any homogeneous $f \in R$, a morphism
\[
    \begin{tikzpicture}[thick,scale=0.07,baseline]
      \node at (0,0) {$f$};
      \draw[dotted] (-5,-5) rectangle (5,5);
    \end{tikzpicture}
\]
from $B_\varnothing$ to itself, of degree $\deg(f)$;
\item
for any $s \in S$, ``dot'' morphisms
\[
       \begin{tikzpicture}[thick,scale=0.07,baseline]
      \draw (0,-5) to (0,0);
      \node at (0,0) {$\bullet$};
      \node at (0,-6.7) {\tiny $s$};
    \end{tikzpicture}
    \qquad \text{and} \qquad
      \begin{tikzpicture}[thick,baseline,xscale=0.07,yscale=-0.07]
      \draw (0,-5) to (0,0);
      \node at (0,0) {$\bullet$};
      \node at (0,-6.7) {\tiny $s$};
    \end{tikzpicture}
\]
from $B_s$ to $B_\varnothing$ and from $B_\varnothing$ to $B_s$, respectively, of degree $1$;
\item
for any $s \in S$, trivalent morphisms
\[
        \begin{tikzpicture}[thick,baseline,scale=0.07]
      \draw (-4,5) to (0,0) to (4,5);
      \draw (0,-5) to (0,0);
      \node at (0,-6.7) {\tiny $s$};
      \node at (-4,6.7) {\tiny $s$};
      \node at (4,6.7) {\tiny $s$};      
    \end{tikzpicture}
    \qquad \text{and} \qquad
        \begin{tikzpicture}[thick,baseline,scale=-0.07]
      \draw (-4,5) to (0,0) to (4,5);
      \draw (0,-5) to (0,0);
      \node at (0,-6.7) {\tiny $s$};
      \node at (-4,6.7) {\tiny $s$};
      \node at (4,6.7) {\tiny $s$};    
    \end{tikzpicture}
\]
from $B_s$ to $B_{(s,s)}$ and from $B_{(s,s)}$ to $B_s$, respectively, of degree $-1$;
\item
for any pair $(s,t)$ of distinct simple reflections such that $st$ has finite order $m_{st}$ in $W$, a morphism
\[
    \begin{tikzpicture}[yscale=0.5,xscale=0.3,baseline,thick]
\draw (-2.5,-1) to (0,0) to (-1.5,1);
\draw (-0.5,-1) to (0,0);
\draw[red] (-1.5,-1) to (0,0) to (-2.5,1);
\draw[red] (0,0) to (-0.5,1);
\node at (-2.5,-1.3) {\tiny $s$\vphantom{$t$}};
\node at (-1.5,1.3) {\tiny $s$\vphantom{$t$}};
\node at (-0.5,-1.3) {\tiny $s$\vphantom{$t$}};
\node at (-1.5,-1.3) {\tiny $t$};
\node at (-2.5,1.3) {\tiny $t$};
\node at (-0.5,1.3) {\tiny $t$};
\node at (0.6,-0.7) {\small $\cdots$};
\node at (0.6,0.7) {\small $\cdots$};
\draw (2.5,-1) -- (0,0);
\draw[red] (2.5,1) -- (0,0);
\node at (2.5,-1.3) {\tiny $s$\vphantom{$t$}};
\node at (2.5,1.3) {\tiny $t$};
\end{tikzpicture} \text{ if $m_{st}$ is odd or}
    \begin{tikzpicture}[yscale=0.5,xscale=0.3,baseline,thick]
\draw (-2.5,-1) to (0,0) to (-1.5,1);
\draw (-0.5,-1) to (0,0);
\draw[red] (-1.5,-1) to (0,0) to (-2.5,1);
\draw[red] (0,0) to (-0.5,1);
\node at (-2.5,-1.3) {\tiny $s$\vphantom{$t$}};
\node at (-1.5,1.3) {\tiny $s$\vphantom{$t$}};
\node at (-0.5,-1.3) {\tiny $s$\vphantom{$t$}};
\node at (-1.5,-1.3) {\tiny $t$};
\node at (-2.5,1.3) {\tiny $t$};
\node at (-0.5,1.3) {\tiny $t$};
\node at (0.6,-0.7) {\small $\cdots$};
\node at (0.6,0.7) {\small $\cdots$};
\draw[red] (2.5,-1) -- (0,0);
\draw (2.5,1) -- (0,0);
\node at (2.5,-1.3) {\tiny $t$};
\node at (2.5,1.3) {\tiny $s$\vphantom{$t$}};
\end{tikzpicture} \text{ if $m_{st}$ is even}
\]
from $B_{(s,t,\cdots)}$ to $B_{(t,s,\cdots)}$ (where each expression has length $m_{st}$, and colors alternate), of degree $0$,
\end{enumerate}
subject to a number of relations for which we refer to~\cite{ew} or~\cite[\S 2.2]{amrw1}.
\end{itemize}
Then we set $\DiagBS(\h,W):=\bigl( \widetilde{\mathscr{D}}_{\mathrm{BS}}(\h,W) \bigr)^\circ$ (where we use the notation from~\S\ref{ss:graded-cat}). We will also denote by $\DiagBS^\oplus(\h,W)$ the additive hull of $\DiagBS(\h,W)$.

Typically, a morphism in $\DiagBS(\h,W)$ or in $\widetilde{\mathscr{D}}_{\mathrm{BS}}(\h,W)$ will be written as a linear combination of (equivalence classes of) diagrams where horizontal concatenation corresponds to the monoidal product, and vertical concatenation corresponds to composition. Such diagrams are to be read from bottom to top. We will sometimes omit the labels ``$s$'' or ``$t$'' in the diagrams for morphisms when they do not play any role.

Note that for $X,Y$ in $\DiagBS^\oplus(\h,W)$ the graded $\bk$-module
\begin{equation}
\label{eqn:def-Hombullet}
\Hom^\bullet_{\DiagBS^\oplus(\h,W)}(X,Y) := \bigoplus_{n \in \Z} \Hom_{\DiagBS^\oplus(\h,W)}(X,Y(n))
\end{equation}
has a natural structure of graded $R$-bimodule, where the left, resp.~right, action of $f \in R^n$ is induced by adding a box labelled by $f$ to the left, resp.~right, of a diagram.

We set:
\[
\begin{array}{c}
\begin{tikzpicture}[scale=0.3,thick]
\draw (-1,-1) to[out=90,in=180] (0,1) to[out=0,in=90] (1,-1);
\end{tikzpicture}
\end{array}
:=
\begin{array}{c}
\begin{tikzpicture}[scale=0.3,thick]
\draw (-1,-1) to (0,0) to (1,-1);
\draw (0,0) to (0,1);
\node at (0,1) {$\bullet$};
\end{tikzpicture}
\end{array},
\qquad
\begin{array}{c}
\begin{tikzpicture}[scale=-0.3,thick]
\draw (-1,-1) to[out=90,in=180] (0,1) to[out=0,in=90] (1,-1);
\end{tikzpicture}
\end{array}
:=
\begin{array}{c}
\begin{tikzpicture}[scale=-0.3,thick]
\draw (-1,-1) to (0,0) to (1,-1);
\draw (0,0) to (0,1);
\node at (0,1) {$\bullet$};
\end{tikzpicture}
\end{array}.
\]
These morphisms induce morphisms of functors
\[
\bigl( B_s \star (-) \bigr) \circ \bigl( B_s \star (-) \bigr) \to \id \quad \text{and} \quad \id \to \bigl( B_s \star (-) \bigr) \circ \bigl( B_s \star (-) \bigr)
\]
which make $\bigl( B_s \star (-), B_s \star (-) \bigr)$ an adjoint pair. Similarly, $\bigl( (-) \star B_s, (-) \star B_s \bigr)$ is an adjoint pair in a natural way. Let us recall also the isomorphism
 \begin{equation}
 \label{eqn:BsBs}
 B_s \star B_s \cong B_s(1) \oplus B_s(-1),
 \end{equation}
 see~\cite[(5.14)]{ew}.

This category has another symmetry which turns out to be very useful.
We denote by
\[
\D:\DiagBS^\oplus(\h,W)\to\DiagBS^\oplus(\h,W)^{\mathrm{op}}
\]
the antiinvolution which fixes each $B_\uw$ and flips diagrams upside-down; see~\cite[Definition 6.22]{ew}. 
Notice that 
\begin{align*}
\D\circ(n)\simeq(-n)\circ\D. 
\end{align*}

\subsection{The double leaves basis}

One of the main results of~\cite{ew} states that for any two expressions $\uv,\uw$, the graded $R$-bimodule $\Hom^\bullet_{\DiagBS(\h,W)}(B_{\uv}, B_{\uw})$ is graded free of finite rank as a left $R$-module and as a right $R$-module (see~\cite[Corollary~6.14]{ew}). In fact, following an idea of Libedinsky, Elias and Williamson provide a way to produce a set $\LL_{\uv,\uw}$ of homogeneous morphisms, called ``double leaves morphisms,'' which constitutes of graded basis of $\Hom^\bullet_{\DiagBS(\h,W)}(B_{\uv}, B_{\uw})$, both as a left $R$-module and as a right $R$-module. This construction is algorithmic in nature, and depends on many choices. We will not repeat the construction here, but we will recall certain properties that we will need below.

The set $\LL_{\uv,\uw}$ is in natural bijection with the set 
\begin{align*}
\bigcup_{x\in W}M(\uw,x)\times M(\uv,x).
\end{align*}
In fact, if $\mathbf{e}$ and $\mathbf{f}$ are subexpressions of $\uv$ and $\uw$ respectively expressing the same element $x \in W$, then the procedure of~\cite[\S 6.1]{ew} produces homogeneous elements $\mathrm{LL}_{\uv,\mathbf{e}} \in \Hom^\bullet_{\DiagBS(\h,W)}(B_{\uv}, B_{\ux})$ and $\mathrm{LL}_{\uw,\mathbf{f}} \in \Hom^\bullet_{\DiagBS(\h,W)}(B_{\uw}, B_{\ux})$ for a certain reduced expression $\ux$ for $x$ (which can be chosen arbitrarily), and then one defines
\[
\LL^{\uv,\uw}_{x,\mathbf{f},\mathbf{e}} := \D(\mathrm{LL}_{\uw,\mathbf{f}}) \circ \mathrm{LL}_{\uv,\mathbf{e}}
\]
and sets
\[
\LL_{\uv,\uw} = \left\{ \LL^{\uv,\uw}_{x,\mathbf{f},\mathbf{e}} : (\mathbf{f},\mathbf{e}) \in \bigcup_{x\in W}M(\uw,x)\times M(\uv,x) \right\}.
\]

Note in particular that if $\uv$ and $\uw$ are reduced expressions, then the element $x$ above must satisfy $x \leq \pi(\uv)$ and $x \leq \pi(\uw)$.

\begin{example}\label{example:hom s empty s}
Let $s\in S$. The (left or right) $R$-modules $\Hom_{\DiagBS(\h,W)}^\bullet(B_{s},B_{\varnothing})$ and $\Hom_{\DiagBS(\h,W)}^\bullet(B_{\varnothing},B_{s})$ are of rank $1$, with generators
\[
        \begin{tikzpicture}[thick,scale=0.07,baseline]
      \draw (0,-5) to (0,0);
      \node at (0,0) {$\bullet$};
      \node at (0,-6.7) {\tiny $s$};
    \end{tikzpicture}
    \qquad \text{and} \qquad
      \begin{tikzpicture}[thick,baseline,xscale=0.07,yscale=-0.07]
      \draw (0,-5) to (0,0);
      \node at (0,0) {$\bullet$};
      \node at (0,-6.7) {\tiny $s$};
    \end{tikzpicture}
\]

respectively.
\end{example}

\subsection{The biequivariant and the right equivariant categories}
\label{ss:BE-RE}

In~\cite{amrw1}, Makisumi, Williamson and the first two authors of the present paper study various triangulated categories constructed out of $\DiagBS(\h,W)$. The two cases that we will mainly consider in this paper are:
\begin{itemize}
\item
the biequivariant\footnote{The motivation for our terminology comes from geometry; see~\cite{amrw1} for details.} category $\BE(\h,W)$, which can be defined as
\[
\BE(\h,W) := \Kb \DiagBS^\oplus(\h,W);
\]
\item
the right-equivariant category $\RE(\h,W)$, which can be defined as
\[
\RE(\h,W) := \Kb \oDiagBS^\oplus(\h,W).
\]
\end{itemize}
Here, $\oDiagBS^\oplus(\h,W)$ is the additive hull of the category $\oDiagBS(\h,W)$ obtained by the procedure $(-)^\circ$ of~\S\ref{ss:graded-cat} out of the category obtained from $\widetilde{\mathscr{D}}_{\mathrm{BS}}(\h,W)$ by applying $\bk \otimes_R (-)$ to morphism spaces (where again $\bk$ is in degree $0$, and $R$ acts via the quotient $R/V \cdot R=\bk$). For $\uw$ an expression, we will denote by $\oB_{\uw}$ the image of $B_\uw$ in $\oDiagBS^\oplus(\h,W)$.

The category $\BE(\h,W)$ has a natural monoidal structure, which ``extends'' the product $\star$ on $\Diag_\BS^\oplus(\h,W)$, and whose product will be denoted $\ustar$; see~\cite[\S 4.2]{amrw1} for details. (This construction involves some rather delicate sign conventions, which will not be recalled in detail here.) As in~\S\ref{ss:EW-category}, the pairs of functors $\bigl( B_s \ustar (-), B_s \ustar (-) \bigr)$ and $\bigl( (-) \ustar B_s, (-) \ustar B_s \bigr)$ form adjoint pairs in a natural way. The unit for this product is $B_\varnothing$. The category $\RE(\h,W)$ is in a natural way a right module category over $\BE(\h,W)$; this operation is also denoted $\ustar$. There  also exists a natural ``forgetful'' functor
\[
\For^\BE_\RE : \BE(\h,W) \to \RE(\h,W)
\]
induced by tensoring morphism spaces with $\bk$ (over $R$); this functor satisfies
\begin{equation}
\label{eqn:For-convolution}
\For^\BE_\RE (\mathscr{F} \ustar \mathscr{G}) = \For^\BE_\RE(\mathscr{F}) \ustar \mathscr{G}
\end{equation}
for $\mathscr{F},\mathscr{G}$ in $\BE(\h,W)$. 

The ``cohomological shift'' functors on the triangulated categories $\BE(\h,W)$ and $\RE(\h,W)$ will be denoted $[1]$. These categories possess two other ``shift'' autoequivalences denoted $\langle 1 \rangle$ and $(1)$. Here $(1)$ extends the operation on $\DiagBS^\oplus(\h,W)$ denoted similarly in the following way: it sends a complex $(\mathscr{F}^n, d^n)_{n \in \Z}$ to the complex $(\mathscr{F}^n(1), -d^n)_{n \in \Z}$, and we have $(1)=\langle -1 \rangle [1]$. The $m$-th power of $[1]$, resp.~$\langle 1 \rangle$, resp.~$(1)$, is denoted $[m]$, resp.~$\langle m \rangle$, resp.~$(m)$.

\section{Diagrammatic categories associated with locally closed subsets of \texorpdfstring{$W$}{W}}
\label{sec:categories-subsets}

We continue with the setting of Section~\ref{sec:EW categories}. In particular, $\bk$ is only required to be an integral domain.

\subsection{The diagrammatic category attached to a closed subset}
\label{ss:Diag-closed}
Let $I \subset W$ be a closed subset. We define the category
\[
\Diag_{\BS,I}(\h,W)
\]
as the full subcategory of $\DiagBS(\h,W)$ 
whose objects are of the form $B_{\uw}(n)$ for $n \in \Z$ and $\uw$ a reduced expression for an element in $I$. We will also denote by $\Diag_{\BS,I}^\oplus(\h,W)$ the additive hull of $\Diag_{\BS,I}(\h,W)$; this category identifies in a natural way with the full subcategory of $\DiagBS^\oplus(\h,W)$ whose objects are the direct sums of objects of the form $B_\uw(n)$ with $\uw$ a reduced expression for an element in $I$.

\begin{rmk}
We warn the reader that, in the case $I=W$, it is not clear (and most probably false) that the category $\Diag^\oplus_{\BS,W}(\h,W)$ is equivalent to $\DiagBS^\oplus(\h,W)$, since the latter contains objects $B_{\uw}$ where $\uw$ is not a reduced expression.  Nevertheless, we will see later on that their homotopy categories are equivalent; see Remark~\ref{rmk:generate}\eqref{it:generate}.
\end{rmk}

Note that the antiinvolution $\D$ stabilizes the subcategory $\Diag^\oplus_{\BS,I}(\h,W)$; its restriction will be denoted $\D_I$. As in $\DiagBS^\oplus(\h,W)$, for $B,B'$ in $\Diag^\oplus_{\BS,I}(\h,W)$ we set
\[
\Hom^\bullet_{\Diag^\oplus_{\BS,I}(\h,W)}(B,B') = \bigoplus_{n \in \Z} \Hom_{\Diag^\oplus_{\BS,I}(\h,W)}(B,B'(n)).
\]

If $B,B'$ are objects of $\DiagBS^{\oplus}(\h,W)$, we will denote by
\[
\kF_I(B,B') \subset \Hom^\bullet_{\DiagBS^{\oplus}(\h,W)}(B,B')
\]
the submodule of morphisms which factor
through $\rD_{\BS,I}^{\oplus}(\h,W)$.

\begin{lema}\label{le:kF I closed}
If $\uv$ and $\uw$ are expressions, then $\kF_{I}(B_{\uv},B_{\uw})$ is the $R$-span (under either the left or right action) of the double leaves morphisms $\LL^{\uv,\uw}_{x,\f,\be}$ with $x\in I$.
\end{lema}

\begin{proof}
To fix notation we consider the left action of $R$.

It is clear from the definition that if $x \in I$ then $\LL^{\uv,\uw}_{x,\f,\be}\in\kF_{I}(B_{\uv},B_{\uy})$. In particular, the $R$-span under consideration is contained in $\kF_{I}(B_{\uv},B_{\uw})$.

For the opposite containment, we will prove 
that for any reduced expression $\uy$ for an element of $I$, any morphism which factors through a shift of $B_\uy$ belongs to the $R$-span of the light leaves morphisms $\LL^{\uv,\uw}_{x,\f,\be}$ with $x\in I$. 
Let $\uy$ be as above, and let $f \in \Hom^\bullet_{\DiagBS^{\oplus}(\h,W)}(B_\uv,B_\uw)$ be a morphism which factors through a shift of $B_\uy$. Since the light leaves morphisms form an $R$-basis of $\Hom^\bullet_{\DiagBS^{\oplus}(\h,W)}(B_\uv,B_{\uy})$, we can assume that $f=g\LL_{x,\mathbf{f},\mathbf{e}}^{\uv,\uy}$ for some subexpressions $\mathbf{e},\mathbf{f}$ of $\uv$ and $\uy$ respectively expressing some element $x$. Here $x \leq \pi(\uy)$, and hence $x \in I$. By~\cite[Claim~6.21]{ew}, $f$ is then an $R$-linear combination of light leaves morphisms $\LL_{x,\mathbf{f}',\mathbf{e}}^{\uv,\uw}$ for some subexpressions $\mathbf{f}'$ of $\uw$ expressing $x$, and some light leaves morphisms corresponding to subexpressions expressing certain elements $x' < x$. Here again $x' \in I$, so the result follows.
\end{proof}

\subsection{The diagrammatic category attached to a locally closed subset}
\label{ss:Diag-loc-closed}

Let $I_0 \subset W$ be a closed subset, and let $I_1 \subset I_0$ be closed. Then $I_1$ is also closed in $W$, so that we can consider the categories $\Diag^\oplus_{\BS,I_0}(\h,W)$ and $\Diag^\oplus_{\BS,I_1}(\h,W)$. We set
\[
\Diag_{\BS,I_0,I_1}^\oplus(\h,W) := \Diag_{\BS,I_0}^{\oplus}(\h,W)\qq\Diag_{\BS,I_1}^{\oplus}(\h,W),
\]
where the ``naive'' quotient on the right-hand side is defined as follows: its objects are the same as those of $\Diag_{\BS,I_0}^{\oplus}(\h,W)$, and its morphisms are defined by
\[
\Hom_{\Diag_{\BS,I_0,I_1}^\oplus(\h,W)}(B,B') = \left( \Hom^\bullet_{\Diag_{\BS,I_0}^\oplus(\h,W)}(B,B') / \kF_{I_1}(B,B') \right)^0
\]
for $B,B'$ in $\Diag_{\BS,I_0}^\oplus(\h,W)$, where the superscript ``$0$'' means the degree-$0$ part. Note that the objects $B_\uw$ with $\uw$ a reduced expression for an element in $I_1$ have trivial image in $\Diag_{\BS,I_0,I_1}^\oplus(\h,W)$. In particular, every object of $\Diag_{\BS,I_0,I_1}^\oplus(\h,W)$ is a direct sum of (images of) objects of the form $B_{\uw}(m)$ where $m \in \Z$ and $\uw$ is a reduced expression for an element in $I_0 \smallsetminus I_1$.

Of course the ``shift'' equivalence $(1)$ induces an autoequivalence of the category $\Diag_{\BS,I_0,I_1}^\oplus(\h,W)$, which will be denoted similarly. If 
$B,B'$ are in $\Diag_{\BS,I_0}^\oplus(\h,W)$,
the left and right actions of $R$ on $\Hom_{\Diag_{\BS,I_0}^\oplus(\h,W)}^\bullet(B,B')$ descend to actions on
\[
\Hom_{\Diag_{\BS,I_0,I_1}^\oplus(\h,W)}^\bullet(B,B') := \bigoplus_{n \in \Z} \Hom_{\Diag_{\BS,I_0,I_1}^\oplus(\h,W)}(B,B'(n)).
\]
Moreover, if $B=B_\uv$ and $B'=B_\uw$ where $\uv,\uw$ are reduced expressions for elements of $I_0$, then it follows from Lemma~\ref{le:kF I closed} that $\Hom_{\Diag_{\BS,I_0,I_1}^\oplus(\h,W)}^\bullet(B_\uv,B_\uw)$ is free as a left and as a right graded $R$-module, and that the images of the light leaves morphisms $\LL_{x,\mathbf{f},\mathbf{e}}^{\uv,\uw}$ with $x \in I_0 \smallsetminus I_1$ form a graded basis of this space (both as a left and as a right $R$-module).  More generally, this implies that for arbitrary $B,B'$ in $\Diag_{\BS,I_0,I_1}^\oplus(\h,W)$, the space $\Hom_{\Diag_{\BS,I_0,I_1}^\oplus(\h,W)}^\bullet(B,B')$ is graded free both as a left and as a right $R$-module.

\begin{lema}
\label{lem:DiagI-indep}
Up to canonical equivalence, the category $\Diag_{\BS,I_0,I_1}^\oplus(\h,W)$ only depends on the locally closed subset $I_0 \smallsetminus I_1$.
\end{lema}

\begin{proof}
Let $I:=I_0 \smallsetminus I_1$. Then $I_0$ contains $\overline{I}:=\{z \in W \mid \exists x \in I, \, z \leq x\}$, so that we have a natural inclusion of categories
\[
\Diag_{\BS,\overline{I}}^\oplus(\h,W) \subset \Diag_{\BS,I_0}^\oplus(\h,W),
\]
which induces a functor
\[
\Diag_{\BS,\overline{I},\overline{I} \smallsetminus I}^\oplus(\h,W) \to \Diag_{\BS,I_0,I_1}^\oplus(\h,W).
\]
The description of morphism spaces in $\Diag_{\BS,I_0,I_1}^\oplus(\h,W)$ in terms of light leaves morphisms considered above implies that this functor is fully faithful. By the remarks above it is also essentially surjective, and hence an equivalence.
\end{proof}

From Lemma~\ref{lem:DiagI-indep} it follows that it makes sense to define, for any locally closed subset $I \subset W$, the category $\Diag_{\BS,I}^\oplus(\h,W)$ as
\[
\Diag_{\BS,I}^\oplus(\h,W) = \Diag_{\BS,I_0,I_1}^\oplus(\h,W) = \Diag_{\BS,I_0}^{\oplus}(\h,W)\qq\Diag_{\BS,I_1}^{\oplus}(\h,W)
\]
where $I_1 \subset I_0$ are any closed subsets of $W$ such that $I=I_0 \smallsetminus I_1$. Of course, in case $I$ is closed, the category we obtain coincides with the category defined in~\S\ref{ss:Diag-closed}. It is clear that the autoequivalences $\D_{I_0}$ and $(1)$ of $\Diag_{\BS,I_0}^{\oplus}(\h,W)$ induce autoequivalences of $\Diag_{\BS,I}^\oplus(\h,W)$, which will be denoted $\D_I$ and $(1)$ respectively.

\subsection{The case of a singleton} 
\label{ss:Diag-singleton}

In this subsection we consider the special case $I=\{w\}$ for $w \in W$. (This subset is obviously locally closed in $W$.) For any choice of a reduced expression $\uw$ for $w$ we can consider the image of the corresponding object $B_\uw$ in $\Diag^\oplus_{\BS,\{w\}}(\h,W)$. If $\uw'$ is another reduced expression for $w$, then $\uw$ and $\uw'$ can be related by a ``rex move,'' i.e.~a sequence of braid relations (meaning the replacement of a subword $(s,t, \ldots)$ by the word $(t,s,\ldots)$, where the words have length the order $m_{s,t}$ of $st$ and their entries alternate between $s$ and $t$). See~\cite[\S 4.2]{ew} for details. To each such braid relation is associated (by definition) a morphism in $\DiagBS^\oplus(\h,W)$; composing these morphisms we obtain a ``rex move morphism'' $B_\uw \to B_{\uw'}$. By~\cite[Lemma~7.4 and Lemma~7.5]{ew}, the image of this morphism in $\Diag_{\BS,\{w\}}^\oplus(\h,W)$ does not depend on the choice of rex move, and is an isomorphism. 
In particular, the images of $B_\uw$ and $B_{\uw'}$ in $\Diag^\oplus_{\BS,\{w\}}(\h,W)$ are canonically isomorphic. Hence they define a canonical object in $\Diag^\oplus_{\BS,\{w\}}(\h,W)$, which will be denoted $b_w$.

\begin{lema}
\label{le:Dw is equivalent to free R}
There exists a canonical equivalence of categories 
\[
\gamma:\Diag_{\BS,\{w\}}^{\oplus}(\h,W)\xrightarrow{\sim}\Free^{\fgen,\Z}(R)
\]
such that $\gamma(b_w) = R$.
Under this equivalence, the autoequivalence $(1)$ identifies with the ``shift of grading'' autoequivalence of $\Free^{\fgen,\Z}(R)$ defined by $\bigl( M(1) \bigr)^n = M^{n+1}$.
\end{lema}

\begin{proof}
It follows from the definition and the comments above that any object of $\Diag_{\BS,\{w\}}^{\oplus}(\h,W)$ is isomorphic to a direct sum of shifts of $b_w$. Since moreover we have $\End^\bullet_{\Diag_{\BS,\{w\}}^{\oplus}(\h,W)}(b_w)=R$ by Lemma~\ref{le:kF I closed}, we deduce that the functor $\gamma := \Hom^\bullet_{\Diag_{\BS,\{w\}}^{\oplus}(\h,W)}(b_w, -)$ provides the desired equivalence.
\end{proof}

\subsection{Closed and open inclusions}
\label{ss:inclusions}

Let $I \subset W$ be a locally closed subset, and write $I=I_0 \smallsetminus I_1$ for some closed subsets $I_1 \subset I_0 \subset W$. Any subset $J \subset I$ that is closed as a subset of $I$ can be written as $J_0 \smallsetminus (J_0 \cap I_1)$ for some closed subset $J_0 \subset I_0$. There exists a natural embedding
\[
\Diag_{\BS,J_0}^\oplus(\h,W) \subset \Diag_{\BS,I_0}^\oplus(\h,W),
\]
which induces a 
functor
\[
\Diag_{\BS,J_0}^\oplus(\h,W) \qq \Diag_{\BS,J_0 \cap I_1}^\oplus(\h,W) \to \Diag_{\BS,I_0}^\oplus(\h,W) \qq \Diag_{\BS,I_1}^\oplus(\h,W).
\]
The description of morphism spaces in terms of light leaves morphisms in~\S\ref{ss:Diag-loc-closed} shows that this functor is fully faithful.
As explained in~\S\ref{ss:Diag-loc-closed}, the categories involved here
do not depend on the choices of $I_0$ and $J_0$. It is clear that, under these identifications, the functor does not depend on these choices either; it will be denoted
\[
(i_J^I)_* : \Diag_{\BS,J}^\oplus(\h,W) \to \Diag_{\BS,I}^\oplus(\h,W).
\]
It is clear that this functor satisfies
\[
(i_J^I)_* \circ \D_J \cong \D_I \circ (i_J^I)_*,
\]
and that this construction is compatible with composition of closed inclusions in the obvious way.

Now, let $K \subset I$ be a subset that is open in the order topology on $I$. Let $J=I \smallsetminus K$ be the complementary closed subset, and write $J=J_0 \smallsetminus (J_0 \cap I_1)$ as above, so that $K=I_0 \smallsetminus K_1$, where $K_1 = J_0 \cup I_1$. Then by definition there exists a natural full functor
\[
\Diag_{\BS,I_0}^\oplus(\h,W) \qq \Diag_{\BS,I_1}^\oplus(\h,W) \to \Diag_{\BS,I_0}^\oplus(\h,W) \qq \Diag_{\BS,K_1}^\oplus(\h,W).
\]
Once again, this functor does not depend on the choices of $I_0$ and $J_0$; it will be denoted
\[
(i_K^I)^* : \Diag_{\BS,I}^\oplus(\h,W) \to \Diag_{\BS,K}^\oplus(\h,W).
\]
This functor satisfies
\[
(i_K^I)^* \circ \D_I = \D_K \circ (i_K^I)^*,
\]
and this construction is compatible with composition of open inclusions in the obvious way.

It is clear from this construction that if $J \subset I$ is closed we have
\begin{equation}
\label{eqn:comp-open-closed-0}
(i_{I \smallsetminus J}^I)^* \circ (i_J^I)_*=0.
\end{equation}

\begin{example}
\label{ex:bases of x w and y w}
Let $I \subset W$ be a locally closed subset, and let $w \in I$ be a minimal element. Then the subset $\{w\} \subset I$ is closed. Let us fix a reduced expression $\uw$ for $w$. If $\ux$ and $\uy$ are reduced expressions for elements of $I$, then by Lemma~\ref{le:kF I closed} (see also~\S\ref{ss:Diag-loc-closed}) the subsets
\[
\{ \LL_{w,\mathbf{1},\mathbf{e}}^{\ux,\uw} : \mathbf{e} \in M(\ux,w) \} \quad \text{and} \quad \{ \LL_{w,\mathbf{f},\mathbf{1}}^{\uw,\uy} : \mathbf{f} \in M(\uy,w) \}
\]
(where $\mathbf{1}$ means the subexpression consisting only of $1$'s) form $R$-bases of the modules $\Hom^\bullet_{\Diag_{\BS,I}^\oplus(\h,W)}(B_\ux,B_\uw)$ and $\Hom^\bullet_{\Diag_{\BS,I}^\oplus(\h,W)}(B_\uw,B_\uy)$ respectively (both for the left and for the right actions). Moreover, composition induces a morphism
\[
\Hom^\bullet_{\Diag_{\BS,I}^\oplus(\h,W)}(B_\uw,B_\uy) \otimes_R \Hom^\bullet_{\Diag_{\BS,I}^\oplus(\h,W)}(B_\ux,B_\uw) \to \Hom^\bullet_{\Diag_{\BS,I}^\oplus(\h,W)}(B_\ux,B_\uy)
\]
(where the right $R$-module structure on $\Hom^\bullet_{\Diag_{\BS,I}^\oplus(\h,W)}(B_\uw,B_\uy)$ and the left $R$-module structure on $\Hom^\bullet_{\Diag_{\BS,I}^\oplus(\h,W)}(B_\ux,B_\uw)$ are both given either by adding a ``box'' to the right of diagrams, or by adding a ``box'' to the left of diagrams). The ``light leaves basis'' considerations of~\S\ref{ss:Diag-loc-closed} also show that this morphism is injective (for both choices of conventions for $R$-actions).
\end{example}

\begin{rmk}
\label{rmk:open-closed}
 Let $I \subset W$ be a locally closed subset, and let $J \subset I$ be a subset which is both open and closed. Then from the definitions we see that
 \[
  (i_J^I)^* \circ (i_J^I)_* = \id.
 \]
 Moreover, using the light leaves basis for morphisms in $\Diag_{\BS,I}^\oplus(\h,W)$ (see~\S\ref{ss:Diag-loc-closed}), it is not difficult to check that for any $B$ in $\Diag_{\BS,J}(\h,W)$ and $B'$ in $\Diag_{\BS,I \smallsetminus J}^\oplus(\h,W)$ we have
 \[
  \Hom_{\Diag_{\BS,I}^\oplus(\h,W)} \bigl( (i_J^I)_* B, (i_{I \smallsetminus J}^I)_* B' \bigr) = 0.
 \]
 It follows that any object $B$ of $\Diag_{\BS,I}^\oplus(\h,W)$ has a canonical decomposition
 \[
  B \cong (i_J^I)_* B' \oplus (i_{I \smallsetminus J}^I)_* B''
 \]
 with $B'$ in $\Diag_{\BS,J}^\oplus(\h,W)$ and $B''$ in $\Diag_{\BS,I \smallsetminus J}^\oplus(\h,W)$, and that we have $B'=(i_J^I)^* B$ and $B''=(i_{I \smallsetminus J}^I)^* B$. From this we deduce that the pairs
 \[
 \bigl( (i_J^I)^*, (i_J^I)_* \bigr) \quad \text{and} \quad \bigl( (i_J^I)_*, (i_J^I)^* \bigr)
 \]
 are adjoint pairs of functors.
\end{rmk}

\section{Recollement}
\label{sec:BE}

We continue with the setting of Sections~\ref{sec:EW categories}--\ref{sec:categories-subsets}. Our goal in this (rather technical) section is to construct a recollement formalism (in the sense of~\cite[\S 1.4.3]{bbd}) for the category $\BE(\h,W)$, which will allow us to describe this category in terms of ``local versions'' associated with locally closed subsets of $W$.

\subsection{The biequivariant category associated with a locally closed subset}
\label{ss:BE-category}

If $I \subset W$ is a locally closed subset, we define the triangulated category $\BE_I(\h,W)$ by setting
\[
\BE_I(\h,W):=\Kb\Diag_{\BS,I}^{\oplus}(\h,W).
\]
As for $\BE(\h,W)$, this category admits ``shift'' autoequivalences $[n]$, $\langle n \rangle$, $(n)$ defined as above (for $n \in \Z$). The contravariant autoequivalence $\D_I$ of $\Diag_{\BS,I}^\oplus(\h,W)$ also induces a (contravariant) autoequivalence of $\BE_I(\h,W)$, which will be denoted similarly. By definition we have
\[
\D_I \circ [n] = [-n] \circ \D_I, \quad \D_I \circ \langle n \rangle = \langle -n \rangle \circ \D_I, \quad \D_I \circ (n) = (-n) \circ \D_I.
\]

If $J \subset I$ is a closed subset, then the functor $(i_J^I)_*$ defined in~\S\ref{ss:inclusions} induces a fully faithful functor from $\BE_J(\h,W)$ to $\BE_I(\h,W)$, which will also be denoted $(i_J^I)_*$. Whenever convenient, we will identify $\BE_J(\h,W)$ with its image in $\BE_I(\h,W)$, and omit the functor $(i_J^I)_*$.

Similarly, if $K \subset I$ is an open subset, then the functor $(i_K^I)^*$ defined in~\S\ref{ss:inclusions} induces a functor from $\BE_I(\h,W)$ to $\BE_K(\h,W)$, which will also be denoted $(i_K^I)^*$. As in~\S\ref{ss:inclusions}, we have
\begin{equation}
\label{eqn:D-incl}
(i_J^I)_* \circ \D_J = \D_I \circ (i_J^I)_*, \quad (i_K^I)^* \circ \D_I = \D_K \circ (i_K^I)^*.
\end{equation}

Note that the functors $(i_J^I)_*$ identify the category $\BE_I(\h,W)$ with the inductive limit of the categories $\BE_J(\h,W)$ for $J \subset I$ a finite closed subset.
This observation will allow us to generalize some of our constructions below from finite subsets of $W$ to arbitrary subsets.

\subsection{Closed embedding of a singleton}

In this subsection we fix a locally closed subset $I \subset W$ and a minimal element $w \in I$ (so that $\{w\}$ is a closed subset of $I$). Our goal is to prove Lemma~\ref{le:recollement step 1a} below.

The statement of this lemma involves the ``$*$'' operation from~\cite[\S1.3.9]{bbd}.  We recall the definition of this notation: if $\mathcal{D}$ is a triangulated category, and if $\mathcal{A}, \mathcal{B} \subset \mathcal{D}$ are two full subcategories, then $\mathcal{A} * \mathcal{B}$ denotes the strictly full subcategory of $\mathcal{D}$ whose objects $X$ are those that fit into a distinguished triangle $A \to X \to B \xrightarrow{[1]}$ with $A \in \mathcal{A}$ and $B \in \mathcal{B}$.

\begin{lema}
\label{le:recollement step 1a}
The functor $(i_{I \smallsetminus \{w\}}^I)^*$ admits a left 
adjoint $(i_{I \smallsetminus \{w\}}^I)_!$ and a right adjoint $(i_{I \smallsetminus \{w\}}^I)_*$. Moreover, the adjunction morphisms 
\[
(i_{I \smallsetminus \{w\}}^I)^*(i_{I \smallsetminus \{w\}}^I)_*\to\id \quad \text{and} \quad \id\rightarrow (i_{I \smallsetminus \{w\}}^I)^*(i_{I \smallsetminus \{w\}}^I)_!
\]
are isomorphisms, and we have
\begin{align*}
\BE_I(\h,W)&= (i_{I \smallsetminus \{w\}}^I)_!(\BE_{I\smallsetminus\{w\}}(\h,W)) * (i_{\{w\}}^I)_*(\BE_{\{w\}}(\h,W)),\\
\BE_I(\h,W)&= (i_{\{w\}}^I)_*(\BE_{\{w\}}(\h,W)) * (i_{I \smallsetminus \{w\}}^I)_*(\BE_{I\smallsetminus\{w\}}(\h,W)).
\end{align*}
\end{lema}

The proof of this lemma will use the following construction. We fix once and for all a reduced expression $\uw$ for $w$. Then for any reduced expression $\ux$ for an element in $I \smallsetminus \{w\}$ we consider the complex $B_\ux^+$ given by
\[
\cdots \to 0 \to B_\uw \Tenint_R \Hom^\bullet_{\Diag_{\BS,I}^\oplus(\h,W)}(B_\uw,B_\ux) \to B_\ux \to 0 \to \cdots,
\]
where $B_\uw \Tenint_R \Hom^\bullet_{\Diag_{\BS,I}^\oplus(\h,W)}(B_\uw,B_\ux)$ is in cohomological degree $-1$, $B_\ux$ is in cohomological degree $0$, all the other terms are $0$, and the only nontrivial differential is given by the morphism defined in~\eqref{eqn:canonical-morph}. (In particular, the $R$-module structure on $\Hom^\bullet_{\Diag_{\BS,I}^\oplus(\h,W)}(B_\uw,B_\ux)$ that we consider here is as defined before~\eqref{eqn:canonical-morph}.) Note that we have a canonical distinguished triangle
\begin{equation}
\label{eqn:triangle-B+}
B_\ux \to B_\ux^+ \to B_\uw \Tenint_R \Hom^\bullet_{\Diag_{\BS,I}^\oplus(\h,W)}(B_\uw,B_\ux)[1] \xrightarrow{[1]}
\end{equation}
in $\BE_I(\h,W)$.

\begin{lema}\label{le:iso para the recollement}
If $\ux$ is a reduced expression for an element in $I \smallsetminus \{w\}$ and $\uy$ is a reduced expression for an element in $I$, then for any $n,m \in \Z$ the functor $(i_{I \smallsetminus \{w\}}^I)^*$ induces an isomorphism
\begin{multline*}
\Hom_{\BE_I(\h,W)}(B_{\uy},B_{\ux}^+(m)[n])\xrightarrow{\sim}\\
\Hom_{\BE_{I\smallsetminus\{w\}}(\h,W)}((i_{I \smallsetminus \{w\}}^I)^*B_{\uy},(i_{I \smallsetminus \{w\}}^I)^*B_{\ux}^+(m)[n]).
\end{multline*}
Moreover, these $\bk$-modules are zero unless $\pi(\uy)\neq w$ and $n=0$, in which case they are isomorphic to $\Hom_{\Diag^\oplus_{\BS,I\smallsetminus\{w\}}(\h,W)}(B_{\uy},B_{\ux}(m))$.
\end{lema}

\begin{proof}
It is clear that in the morphism under consideration, the left-hand side vanishes unless $n \in \{-1,0\}$, and the right-hand side vanishes unless $n=0$ (because $(i_{I \smallsetminus \{w\}}^I)^*B_{\ux}^+ = (i_{I \smallsetminus \{w\}}^I)^*B_{\ux}$). In particular, the claim is obvious unless $n \in \{-1,0\}$.

From~\eqref{eqn:triangle-B+} we deduce an exact sequence
\begin{multline*}
0 \to \Hom_{\BE_I(\h,W)}(B_{\uy},B_{\ux}^+(m)[-1]) \\
\to \Hom_{\Diag^\oplus_{\BS,I}(\h,W)} \bigl( B_{\uy},B_\uw \Tenint_R \Hom^\bullet_{\Diag_{\BS,I}^\oplus(\h,W)}(B_\uw,B_\ux)(m) \bigr) \\
\to \Hom_{\Diag_{\BS,I}^\oplus(\h,W)}(B_{\uy},B_{\ux}(m)) \to \Hom_{\BE_I(\h,W)}(B_{\uy},B_{\ux}^+(m)) \to 0.
\end{multline*}
Now by definition (see~\S\ref{ss:tensor-product}) the term on the middle line identifies with
\[
\bigl( \Hom_{\Diag_{\BS,I}^\oplus(\h,W)}^\bullet(B_{\uy},B_\uw(m)) \otimes_R \Hom^\bullet_{\Diag_{\BS,I}^\oplus(\h,W)}(B_\uw,B_\ux) \bigr)^0,
\]
and the differential to $\Hom_{\Diag_{\BS,I}^\oplus(\h,W)}(B_{\uy},B_{\ux}(m))$ identifies with the natural composition morphism. As explained in Example~\ref{ex:bases of x w and y w} this map is injective. It follows that
\[
\Hom_{\BE_I(\h,W)}(B_{\uy},B_{\ux}^+(m)[-1])=0,
\]
proving the desired isomorphism in this case. The fact that our morphism is an isomorphism when $n=0$ also follows from this exact sequence, together with the ``light leaves basis'' considerations in~\S\ref{ss:Diag-loc-closed}.
\end{proof}

\begin{proof}[Proof of Lemma~{\rm \ref{le:recollement step 1a}}]
We will explain the construction of the functor $(i_{I \smallsetminus \{w\}}^I)_*$ and prove that it satisfies the desired properties; then in view of~\eqref{eqn:D-incl} the functor
\begin{equation}
\label{eqn:def-!-pf}
(i_{I \smallsetminus \{w\}}^I)_! := \D_I \circ (i_{I \smallsetminus \{w\}}^I)_* \circ \D_{I \smallsetminus \{w\}}
\end{equation}
will also satisfy the corresponding properties.

Let $\fD^+ \subset \BE_I(\h,W)$ be the full graded (i.e.~stable by $(1)$) triangulated subcategory generated by the objects $B_{\ux}^+$ for all reduced expressions $\ux$ for elements in $I\smallsetminus\{w\}$, and let 
$\iota:\fD^+\rightarrow\BE_I(\h,W)$ be the inclusion. By Lemma \ref{le:iso para the recollement} and using the five-lemma, it follows that the functor $(i_{I \smallsetminus \{w\}}^I)^*$ induces an isomorphism
\begin{align}\label{eq:1 le:recollement step 1}
\Hom_{\BE_I(\h,W)}(Y,\iota X)\xrightarrow{\sim}\Hom_{\BE_{I\smallsetminus\{w\}}(\h,W)} \bigl( (i_{I \smallsetminus \{w\}}^I)^* Y, (i_{I \smallsetminus \{w\}}^I)^* \iota X \bigr)
\end{align}
for all $X$ in $\fD^+$ and $Y$ in $\BE_I(\h,W)$. 
In particular, this shows that
the functor $(i_{I \smallsetminus \{w\}}^I)^*\circ\iota$ is fully faithful. Moreover, since this functor sends $B_\ux^+(m)$ to $B_\ux(m)$ for any reduced expression $\ux$ of an element in $I \smallsetminus \{w\}$, and since these objects generate $\fD^+$ and $\BE_{I\smallsetminus\{w\}}(\h,W)$ respectively as triangulated categories, we even obtain that $(i_{I \smallsetminus \{w\}}^I)^*\circ\iota$ is an equivalence of categories. This fact allows us to set
\[
(i_{I \smallsetminus \{w\}}^I)_*:=\iota\circ \bigl( (i_{I \smallsetminus \{w\}}^I)^*\circ\iota \bigr)^{-1}:\BE_{I\smallsetminus\{w\}}(\h,W)\longrightarrow\BE_{I}(\h,W).
\]
What remains to be proved is that this functor satisfies the desired properties.

By definition we have a canonical isomorphism
\[
(i_{I \smallsetminus \{w\}}^I)^* \circ (i_{I \smallsetminus \{w\}}^I)_* \cong \id.
\]
To prove that $(i_{I \smallsetminus \{w\}}^I)_*$ is right adjoint to $(i_{I \smallsetminus \{w\}}^I)^*$ we have to prove that the composition 
\begin{multline*}
\Hom_{\BE_I(\h,W)} \bigl( X,(i_{I \smallsetminus \{w\}}^I)_* Y \bigr)\xrightarrow{(i_{I \smallsetminus \{w\}}^I)^*}\\
\Hom_{\BE_{I\smallsetminus\{w\}}(\h,W)} \bigl( (i_{I \smallsetminus \{w\}}^I)^* X, (i_{I \smallsetminus \{w\}}^I)^* (i_{I \smallsetminus \{w\}}^I)_* Y \bigr)\\
\cong\Hom_{\BE_{I\smallsetminus\{w\}}(\h,W)} \bigl( (i_{I \smallsetminus \{w\}}^I)^*X,Y \bigr)
\end{multline*}
is an isomorphism for all $X$ in $\BE_I(\h,W)$ and $Y$ in $\BE_{I \smallsetminus \{w\}}(\h,W)$. In fact, this is clear from the isomorphism~\eqref{eq:1 le:recollement step 1}.

To conclude, it remains to prove that
\begin{equation}
\label{eqn:BEI-closed-open}
\BE_I(\h,W)= (i_{\{w\}}^I)_*(\BE_{\{w\}}(\h,W)) * (i_{I \smallsetminus \{w\}}^I)_*(\BE_{I\smallsetminus\{w\}}(\h,W)).
\end{equation}
However, by construction we have 
\begin{align}\label{eq:definition of j!}
(i_{I \smallsetminus \{w\}}^I)_* B_{\ux}=B^+_{\ux} 
\end{align}
for any reduced expression $\ux$ for an element in $I \smallsetminus \{w\}$. In view of the triangle~\eqref{eqn:triangle-B+} and the comments at the beginning of~\S\ref{ss:Diag-singleton}, it follows that
the triangulated category $\BE_{I}(\h,W)$ is generated by the essential images of the functors $(i_{I \smallsetminus \{w\}}^I)_*$ and $(i_{\{w\}}^I)_*$. Since there exists 
no nonzero morphism from an object of $(i_{\{w\}}^I)_*(\BE_{\{w\}}(\h,W))$ to an object of $(i_{I \smallsetminus \{w\}}^I)_*(\BE_{I\smallsetminus\{w\}}(\h,W))$
(by adjunction and the fact that $(i_{I \smallsetminus \{w\}}^I)^* \circ (i_{\{w\}}^I)_*=0$, see~\eqref{eqn:comp-open-closed-0}), we deduce~\eqref{eqn:BEI-closed-open}.
\end{proof}

\begin{rmk}
\label{rmk:ff}
 The claims in Lemma~\ref{le:recollement step 1a} about the adjunction morphisms amount to saying that the functors $(i_{I \smallsetminus \{w\}}^I)_*$ and $(i_{I \smallsetminus \{w\}}^I)_!$ are fully faithful.
\end{rmk}

\begin{example}
\label{ex:w-ws}
Let $w\in W$ and $s\in S$ be such that $ws>w$. Then $\{w\}$ is closed in $\{w,ws\}$, and its open complement is $\{ws\}$. If $\uw$ is a reduced expression for $w$, then there exist canonical distinguished triangles
\begin{gather}
\label{eq:distinguished triangle in BE w ws}
B_{\uw}\langle-1\rangle\to \left( i_{\{ws\}}^{\{w,ws\}} \right)_{!} B_{\uw s}\to B_{\uw s}\xrightarrow{[1]} \\
\label{eq:distinguished triangle in BE w ws-2}
B_{\uw s} \to \left( i_{\{ws\}}^{\{w,ws\}} \right)_* B_{\uw s}\to B_{\uw}\langle 1\rangle \xrightarrow{[1]}
\end{gather}
in $\BE_{\{w,ws\}}(\h,W)$. 
In fact, the $R$-module $\Hom^\bullet_{\Diag^\oplus_{\BS,\{w,ws\}}(\h,W)}(B_\uw,B_{\uw s})$ is generated by
\[
\id_{B_{\uw}} \star        \begin{tikzpicture}[thick,scale=0.07,baseline]
      \draw (0,5) to (0,0);
      \node at (0,0) {$\bullet$};
      \node at (0,6.7) {\tiny $s$};
    \end{tikzpicture},
\]
which has degree $1$.
Hence~\eqref{eq:distinguished triangle in BE w ws-2} is a special case of the triangle~\eqref{eqn:triangle-B+}, and~\eqref{eq:distinguished triangle in BE w ws} is deduced by applying $\D_{\{w,ws\}}$ (see also~\eqref{eq:definition of j!}).
\end{example}

Below we will need the following technical result.

\begin{lema}
\label{le:recollement step 1b}
Let $I$ and $w$ be as in Lemma~{\rm \ref{le:recollement step 1a}}, and let $J \subset I$ be a closed subset containing $w$. Then there exist canonical isomorphisms
\[
(i_{I \smallsetminus \{w\}}^I)_! \circ (i_{J \smallsetminus \{w\}}^{I \smallsetminus \{w\}})_* \cong (i_J^I)_* \circ (i_{J \smallsetminus \{w\}}^J)_!, \quad (i_{I \smallsetminus \{w\}}^I)_* \circ (i_{J \smallsetminus \{w\}}^{I \smallsetminus \{w\}})_* \cong (i_J^I)_* \circ (i_{J \smallsetminus \{w\}}^J)_*.
\]
\end{lema}

\begin{proof}
As for Lemma~\ref{le:recollement step 1a}, we only prove the second isomorphism; the first one follows by composing on the left with $\D_I$ and on the right with $\D_{J \smallsetminus \{w\}}$ (see~\eqref{eqn:D-incl} and~\eqref{eqn:def-!-pf}). We consider the subcategories
\[
\fD^+_I\subset\BE_{I}(\h,W) \quad \text{and} \quad \fD^+_J\subset\BE_{J}(\h,W)
\]
constructed in the proof of Lemma~\ref{le:recollement step 1a} (applied to the ``ambient'' locally closed subsets $I$ and $J$, respectively), and the corresponding embeddings $\iota_I$ and $\iota_J$. It is clear that the functor $(i_J^I)_* \circ \iota_J$ factors through a functor $(i_J^I)_*^+ : \fD_J^+ \to \fD_I^+$. It is clear also that $(i_{I \smallsetminus \{w\}}^I)^* \circ (i_J^I)_* = (i_{J \smallsetminus \{w\}}^{I\smallsetminus\{w\}})_* \circ (i_{J \smallsetminus \{w\}}^J)^*$. We deduce that
\[
(i_{I \smallsetminus \{w\}}^I)^* \circ \iota_I \circ (i_J^I)_*^+ = (i_{J \smallsetminus \{w\}}^{I\smallsetminus\{w\}})_* \circ (i_{J \smallsetminus \{w\}}^J)^* \circ \iota_J.
\]
Composing on the left with $(i_{I \smallsetminus \{w\}}^I)_*:=\iota_I\circ \bigl( (i_{I \smallsetminus \{w\}}^I)^*\circ\iota_I \bigr)^{-1}$ and on the right with $\bigl( (i_{J \smallsetminus \{w\}}^J)^* \circ \iota_J \bigr)^{-1}$, we deduce the desired isomorphism.
\end{proof}

\subsection{Recollement}

We now formulate and prove the main result of the section. 

\begin{prop}
\label{prop:recollement}
Let $I \subset W$ be a locally closed subset, and let $J\subset I$ be a finite closed subset. Then the functor $(i_{I \smallsetminus J}^I)^*:\BE_I(\h,W)\to\BE_{I\smallsetminus J}(\h,W)$ admits a left adjoint 
$(i_{I \smallsetminus J}^I)_!$ and a right adjoint 
$(i_{I \smallsetminus J}^I)_*$. 
Similarly, the functor $(i_J^I)_*:\BE_J(\h,W)\to\BE_{I}(\h,W)$ admits a left adjoint $(i_J^I)^*$ and a right adjoint $(i_J^I)^!$. Together, these functors give a recollement diagram
\[
 \xymatrix@C=2cm{
 \BE_J(\h,W) \ar[r]|-{(i_J^I)_*} & \BE_I(\h,W) \ar[r]|-{(i_{I \smallsetminus J}^I)^*} \ar@/^0.5cm/[l]^-{(i_J^I)^!} \ar@/_0.5cm/[l]_-{(i_J^I)^*} & \BE_{I\smallsetminus J}(\h,W). \ar@/^0.5cm/[l]^-{(i_{I \smallsetminus J}^I)_*} \ar@/_0.5cm/[l]_-{(i_{I \smallsetminus J}^I)_!}
 }
\]
\end{prop}

\begin{proof}
We begin by showing, by induction on $|J|$, that 
\begin{enumerate}
 \item 
 the functor $(i_{I \smallsetminus J}^I)_!$ exists;
 \item
 the adjunction morphism $\id\to (i_{I \smallsetminus J}^I)^* \circ (i_{I \smallsetminus J}^I)_!$ is an isomorphism;
 \item
 we have
\begin{equation}
\label{eq:1 prop:recollement}
\BE_I(\h,W) = (i_{I \smallsetminus J}^I)_!(\BE_{I\smallsetminus J}(\h,W)) * (i_J^I)_*(\BE_{J}(\h,W)).
\end{equation}
\end{enumerate}

If $|J|=1$, these assertions are part of the statement of Lemma \ref{le:recollement step 1a}. If $|J|>1$, we pick $w\in J$ minimal. By induction the 
functors 
\[
 (i_{I \smallsetminus J}^{I \smallsetminus \{w\}})^{*}:\BE_{I\smallsetminus\{w\}}(\h,W)\rightarrow\BE_{I\smallsetminus J}(\h,W)
\]
and
\[
 (i_{I \smallsetminus \{w\}}^I)^{*}:\BE_{I}(\h,W)\rightarrow\BE_{I\smallsetminus\{w\}}(\h,W)
\]
admit left adjoints $(i_{I \smallsetminus J}^{I \smallsetminus \{w\}})_!$ and $(i_{I \smallsetminus \{w\}}^I)_!$ respectively. Hence their composition, which is $(i_{I \smallsetminus J}^I)^*$ (see~\S\ref{ss:inclusions}), also admits a left adjoint $(i_{I \smallsetminus J}^I)_!$, and we have
\begin{equation}
\label{eqn:adjunction-open}
(i_{I \smallsetminus J}^I)_! = (i_{I \smallsetminus \{w\}}^I)_! \circ (i_{I \smallsetminus J}^{I \smallsetminus \{w\}})_!.
\end{equation}
From the corresponding claims for the embeddings $i_{I \smallsetminus J}^{I \smallsetminus \{w\}}$ and $i_{I \smallsetminus \{w\}}^I$ it is not difficult to deduce that the adjunction morphism
\begin{equation}
\label{eqn:adj-morphism-construction}
 \id\to (i_{I \smallsetminus J}^I)^* \circ (i_{I \smallsetminus J}^I)_!
\end{equation}
is an isomorphism. Finally, by induction we have
\begin{align*}
\BE_{I\smallsetminus\{w\}}(\h,W) &= (i_{I \smallsetminus J}^{I \smallsetminus \{w\}})_!(\BE_{I\smallsetminus J}(\h,W))* (i_{J \smallsetminus \{w\}}^{I \smallsetminus \{w\}})_*(\BE_{J\smallsetminus\{w\}}(\h,W)),\\
\BE_{I}(\h,W) &= (i_{I \smallsetminus \{w\}}^I)_!(\BE_{I\smallsetminus\{w\}}(\h,W))*(i_{\{w\}}^I)_*(\BE_{\{w\}}(\h,W)).
\end{align*}
Using the associativity of the operation ``$*$'' (see~\cite[Lemme 1.3.10]{bbd}), Lemma~\ref{le:recollement step 1a} and Lem\-ma~\ref{le:recollement step 1b} we deduce~\eqref{eq:1 prop:recollement}, which finishes the induction.

Now, we prove the existence of the functor $(i_J^I)^*$ and construct a distinguished triangle
\begin{equation}
\label{eqn:triangle-closed-open}
(i_{I \smallsetminus J}^I)_! (i_{I \smallsetminus J}^I)^* \mathscr{F} \to \mathscr{F} \to (i_J^I)_* (i_J^I)^* \mathscr{F} \xrightarrow{[1]}
\end{equation}
for any $\mathscr{F}$ in $\BE_I(\h,W)$.
We first observe that both the functors $(i_J^I)_*$ and $(i_{I \smallsetminus J}^I)_!$ are fully faithful (see~\S\ref{ss:BE-category} for $(i_J^I)_*$; for $(i_{I \smallsetminus J}^I)_!$ this follows from the invertibility of~\eqref{eqn:adj-morphism-construction}.)
Using~\eqref{eq:1 prop:recollement}, it then follows that for any 
$\rF\in\BE_I(\h,W)$ there exist unique objects $\rF'\in\BE_{I\smallsetminus J}(\h,W)$ and $\rF''\in\BE_J(\h,W)$ and a unique distinguished triangle
\begin{equation}
\label{eqn:triangle-closed-open-2}
 (i_{I \smallsetminus J}^I)_!\rF'\to\rF\to (i_J^I)_*\rF''\xrightarrow{[1]}.
 \end{equation}
(Here, the uniqueness claims follow from~\cite[Proposition~1.1.9]{bbd}). Since we have $(i_{I \smallsetminus J}^I)^*(i_J^I)_*=0$ (see~\eqref{eqn:comp-open-closed-0}) and $(i_{I \smallsetminus J}^I)^*(i_{I \smallsetminus J}^I)_!\cong\id$ (see~\eqref{eqn:adj-morphism-construction}), we have a canonical isomorphism $(i_{I \smallsetminus J}^I)^*\rF\cong\rF'$. We set $(i_J^I)^*\rF:=\rF''$. Another application of~\cite[Proposition~1.1.9]{bbd} shows that this defines a functor $(i_J^I)^*$. Then this functor is left adjoint to $(i_J^I)_*$ 
thanks to the distinguished triangle~\eqref{eqn:triangle-closed-open} and the fact that $(i_{I \smallsetminus J}^I)^*(i_J^I)_*=0$.

Finally, we remark that $(i_J^I)^*(i_J^I)_*\rG\cong\rG$ for all $\rG\in\BE_J(\h,W)$ by uniqueness of the distinguished triangle~\eqref{eqn:triangle-closed-open-2}. Composing with the appropriate duality functors, from the existence of the functors $(i_{I \smallsetminus J}^I)_!$ and $(i_J^I)^*$ we deduce the existence of the functors $(i_{I \smallsetminus J}^I)_*$ and $(i_J^I)^!$ (see~\eqref{eqn:D-incl}), and from the properties we proved for the former functors we deduce similar properties for the latter functors; this finishes the proof of the proposition.
\end{proof}

\begin{rmk}
\label{rmk:recollement-quotient}
 Once the recollement formalism is constructed, we see from~\cite[Proposition~1.4.5]{bbd} that if $I=I_0 \smallsetminus I_1$ with $I_1 \subset I_0 \subset W$ closed subsets and $I_0$ finite, then the functor $(i_I^{I_0})^*$ identifies the category $\BE_I(\h,W)$ with the Verdier quotient of $\BE_{I_0}(\h,W)$ by the full triangulated subcategory $\BE_{I_1}(\h,W)$.  This remark provides an alternative perspective on $\BE_I(\h,W)$, separate from that coming from~\S\ref{ss:Diag-loc-closed}.
\end{rmk}

Let us point out once again that in the setting of Proposition~\ref{prop:recollement} we have canonical
isomorphisms
\begin{align}
\label{eq:duality and recollement}
\D_I\circ (i_{I \smallsetminus J}^I)_!\cong (i_{I \smallsetminus J}^I)_*\circ\D_{I\smallsetminus J}\quad \text{and}\quad \D_J\circ (i_J^I)^!\cong (i_J^I)^*\circ\D_{I}.
\end{align}
Also, our functors are compatible with composition of inclusions in the sense of the following lemma.

\begin{lema}
\label{le:recollement compositions}
Let $I \subset W$ be a locally closed subset and let $J' \subset J \subset I$ be finite closed subsets. 
Then for $\dag\in\{!,*\}$ we have canonical isomorphisms
\[
 (i_J^I)^\dag \circ (i_{J'}^J)^\dag \cong (i_{J'}^I)^\dag, \qquad (i_{I \smallsetminus J'}^I)_\dag \circ (i_{I \smallsetminus J}^{I \smallsetminus J'})_\dag \cong (i_{I \smallsetminus J}^I)_\dag.
\]
\end{lema}

\begin{proof}
The claim follows by adjunction from the corresponding properties for the functors $(i_J^I)_*$ and $(i_{I \smallsetminus J}^I)^*$ (and the similar functors for the other embeddings), see~\S\ref{ss:inclusions}.
\end{proof}

\begin{rmk}\phantomsection
\label{rmk:!*}
\begin{enumerate}
\item
\label{it:open-closed-2}
 Assume that $I$ is a finite locally closed subset of $W$, and that $J \subset I$ is both open and closed. Then we have the ``naive'' functors $(i_J^I)_*$ and $(i_J^I)^*$ defined as in~\S\ref{ss:BE-category}, and also the functors constructed (by adjunction) in Proposition~\ref{prop:recollement}, that we will denote provisionally $(i_J^I)_{(*)}$, $(i_J^I)_{(!)}$, $(i_J^I)^{(*)}$ and $(i_J^I)^{(!)}$. It follows from Remark~\ref{rmk:open-closed} that we have canonical isomorphisms
 \[
  (i_J^I)_{(*)} \cong (i_J^I)_{(!)} \cong (i_J^I)_{*} \quad \text{and} \quad (i_J^I)^{(*)} \cong (i_J^I)^{(!)} \cong (i_J^I)^{*},
 \]
so that we can stop distinguishing these functors.
\item
\label{it:iw-1degree}
We note for later use that if $w$ is minimal in $I$ then for any $B$ in $\Diag_{\BS,I}(\h,W)$
we have
\[
(i_{\{w\}}^I)^! B \cong b_w \Tenint_R \Hom^\bullet_{\Diag_{\BS,I}^\oplus(\h,W)}(B_\uw,B)
\]
(so that, in particular, $(i_{\{w\}}^I)^! B$ is isomorphic to a complex concentrated in degree $0$). In fact, it suffices to prove this isomorphism when $B=B_\ux$ for $\ux$ a reduced expression for an element in $I$. If this element is not $w$, then the isomorphism is obtained from the triangle~\eqref{eqn:triangle-B+}. If now $\ux$ is a reduced expression for $w$, then the isomorphism follows from the fact that $(i_{\{w\}}^I)^! (i_{\{w\}}^I)_* \cong \id$ (because $(i_{\{w\}}^I)_*$ is fully faithful).
\end{enumerate}
\end{rmk}

\subsection{Pushforward and pullback under locally closed inclusions}
\label{ss:pf-pb-locally-closed}

Our next goal is to define pullback and pushforward functors for any finite \emph{locally closed} inclusion $J \subset I$ (where $I$ is locally closed in $W$).

\begin{lema}
\label{le:recollement and locally closed inclusion}
Let $I$ be a finite locally closed subset of $W$, let $J \subset I$ be a closed subset, let $K \subset I$ be an open subset, and let $L \subset J \cap K$ be a subset which is open in $J$ and closed in $K$. Then for $\dag\in\{!,*\}$ there exist canonical isomorphisms
\[
 (i_K^I)_\dag \circ (i_L^K)_* \cong (i_J^I)_* \circ (i_L^J)_\dag, \qquad (i_L^K)^\dag \circ (i_K^I)^* \cong (i_L^J)^* \circ (i_J^I)^\dag.
\]
\end{lema}

\begin{proof}
We will show (by induction on $|I|$) the first isomorphism for $\dag=!$. Then, as in the proof of~\cite[Lemma~2.6]{modrap2}, the other isomorphisms follow by duality and adjunction.

We have to consider three cases. First, we assume that $I=J=K$. Then $L$ is open and closed in $I$, and the desired isomorphism follows from Remark~\ref{rmk:!*}\eqref{it:open-closed-2}.

Now, we assume $K\neq I$. Let $w\in I\smallsetminus K$ be minimal, 
so that $\{w\}$
is closed in $I \smallsetminus K$, and hence in $I$. Then $I'=I\smallsetminus 
\{w\}$ is open in $I$ and $J':=J\cap I'=J\smallsetminus\{w\}$ is closed in $I'$. 
By induction we have
\[
 (i_K^{I'})_! \circ (i_L^K)_* \cong (i_{J'}^{I'})_* \circ (i_L^{J'})_!,
\]
so by Lemma~\ref{le:recollement compositions} to conclude it suffices to prove that
\[
 (i_{I'}^I)_! \circ (i_{J'}^{I'})_* \cong (i_J^I)_* \circ (i_{J'}^J)_!.
\]
If $w\in J$, then this isomorphism was proved in Lemma \ref{le:recollement step 1b}. If now $w\notin J$, then $J'=J$ and $J$ is both open and closed in $J\cup\{w\}$. By Remark~\ref{rmk:!*}\eqref{it:open-closed-2}, this implies that
\[
 (i_J^{J \cup \{w\}})_* \cong (i_J^{J \cup \{w\}})_!,
\]
and then using Lemma~\ref{le:recollement step 1b} (applied to $J \cup \{w\}$ instead of $J$) that
\[
(i_{I'}^I)_! \circ (i_{J}^{I'})_* \cong (i_{J \cup \{w\}}^I)_* \circ (i_J^{J \cup \{w\}})_! \cong (i_{J \cup \{w\}}^I)_* \circ (i_J^{J \cup \{w\}})_* \cong (i_J^I)_*.
\]

Finally, we consider the case $I=K$ but $J\neq I$. Then $L$ is closed in $I$, and hence also in $J$, and by assumption it is also open in $J$. Hence by Remark~\ref{rmk:!*}\eqref{it:open-closed-2} we have
\[
 (i_L^J)_! \cong (i_L^J)_*,
\]
and the desired isomorphism follows from the compatibility of pushforward functors (for closed embeddings) with composition.
\end{proof}

Lemma~\ref{le:recollement and locally closed inclusion} allows us to define pullback and pushforward functors for any locally closed embedding (in case $I$ is finite). More precisely, let $I \subset W$ be a finite locally closed subset, and let $J \subset I$ be a locally closed subset. Then we can write $J=J_0 \smallsetminus J_1$ with $J_1 \subset J_0 \subset I$ closed subsets. (Here, since $J$ is fixed, $J_1$ is determined by $J_0$, and $J_0$ is determined by $J_1$.) By Lemma~\ref{le:recollement and locally closed inclusion} we have a canonical isomorphism
\[
 (i_{J_0}^I)_* \circ (i_J^{J_0})_* \cong (i_{I \smallsetminus J_1}^I)_* \circ (i_J^{I \smallsetminus J_1})_*.
\]
We claim that moreover these functors do not depend on the choice of $J_0$ or $J_1$ (up to canonical isomorphism). In fact, for any choice we have $J_0 \supset \overline{J}$, where $\overline{J}:=\{w \in I \mid \exists x \in J, \, w \leq x\}$. Lemma~\ref{le:recollement and locally closed inclusion} applied to the diagram
\[
\xymatrix{
J \ar@{^{(}->}[r] & J_0 \\
J \ar@{^{(}->}[r] \ar@{=}[u] & \overline{J} \ar@{^{(}->}[u]
}
\]
implies that $(i_{J}^{J_0})_* \cong (i_{\overline{J}}^{J_0})_* (i_J^{\overline{J}})_*$, from which we deduce that
\[
 (i_{J_0}^I)_* \circ (i_J^{J_0})_* \cong (i_{\overline{J}}^I)_* \circ (i_J^{\overline{J}})_*,
\]
which clearly does not depend on $J_0$. These considerations show that it is legitimate to set
\[
 (i_J^I)_* := (i_{J_0}^I)_* \circ (i_J^{J_0})_*.
\]
Similar arguments show that one can also set
\[
 (i_J^I)_! := (i_{J_0}^I)_* \circ (i_J^{J_0})_!, \quad
 (i_J^I)^* := (i_J^{I \smallsetminus J_1})^* \circ (i_{I \smallsetminus J_1}^I)^*, \quad 
 (i_J^I)^! := (i_J^{I \smallsetminus J_1})^! \circ (i_{I \smallsetminus J_1}^I)^*
\]
(i.e.~that these functors do not depend on the choice of $J_0$ or $J_1$, and can be expressed in a way where open and closed embeddings play an opposite role). Moreover, the pairs
\[
 \bigl( (i_J^I)_!, (i_J^I)^! \bigr) \quad \text{and} \quad \bigl( (i_J^I)^*, (i_J^I)_* \bigr)
\]
are adjoint pairs of functors.

In view of~\eqref{eqn:D-incl} and~\eqref{eq:duality and recollement}, we have canonical isomorphisms
\begin{equation}
 \label{eq:duality and recollement for locally closed inclusions}
\D_I\circ (i_J^I)_! \cong (i_J^I)_* \circ\D_{J} \quad \text{and} \quad \D_J\circ (i_J^I)^! \cong (i_J^I)^* \circ\D_{I}.
\end{equation}
Moreover, since this is true for open and closed embeddings (by the axioms of recollement), the adjunction morphisms
\begin{equation}
\label{eqn:pullback-pushforward-inverse}
 (i_J^I)^* \circ (i_J^I)_* \to \id \quad \text{and} \quad \id \to (i_J^I)^! \circ (i_J^I)_!
\end{equation}
are isomorphisms; in other words the functors $(i_J^I)_*$ and $(i_J^I)_!$ are fully faithful (see in particular Remark~\ref{rmk:ff}). Finally, we note that
\begin{equation}
\label{eqn:pushforward-closed}
(i_J^I)_* = (i_J^I)_! \quad \text{if $J \subset I$ is closed}
\end{equation}
and
\begin{equation}
\label{eqn:pullback-open}
 (i_J^I)^! = (i_J^I)^* \quad \text{if $J \subset I$ is open.}
\end{equation}

\begin{rmk}\label{rmk:loc-closed-tri}
Recall that an adjoint of a triangulated functor is triangulated (see, e.g.,~\cite[Lemma~5.3.6]{neeman}).  Thus, all six functors in Proposition~\ref{prop:recollement} are triangulated.  Since the functors $(i_J^I)_*$, $(i_J^I)_!$, $(i_J^I)^*$, and $(i_J^I)^!$ defined above are all compositions of functors coming from Proposition~\ref{prop:recollement}, they are again triangulated.
\end{rmk}

These constructions are also compatible with composition in the sense of the following lemma.

\begin{lema}
\label{le:recollement compositions for locally closed inclusions}

Let $I \subset W$ be a finite locally closed subset, and let $J \subset I$ and $K \subset J$ be locally closed subsets. Then there exist canonical isomorphisms
\begin{align*}
 (i_J^I)_* \circ (i_K^J)_* \cong (i_K^I)_*, &\qquad (i_J^I)_! \circ (i_K^J)_! \cong (i_K^I)_! \\
 (i_K^J)^* \circ (i_J^I)^* \cong (i_K^I)^*, &\qquad (i_K^J)^! \circ (i_J^I)^! \cong (i_K^I)^!.
\end{align*}
\end{lema}

\begin{proof}
One can choose closed subsets $J_1 \subset J_0 \subset I$ and $K_1 \subset K_0 \subset I$ such that
\[
 J=J_0 \smallsetminus J_1, \quad K=K_0 \smallsetminus K_1, \quad J_1 \subset K_1 \subset K_0 \subset J_0.
\]
(For instance, with $J_0=\overline{J}$ and $K_0=\overline{K} \cup (\overline{J} \smallsetminus J)$ these conditions are satisfied.) Then we have a diagram of embeddings
\[
 \xymatrix{
 K \ar@{^{(}->}[r]^-{o} & K_0 \cap J \ar@{^{(}->}[r]^-{c} \ar@{^{(}->}[d]^-{o} & J \ar@{^{(}->}[d]^-{o} & \\
 & K_0 \ar@{^{(}->}[r]^-{c} & J_0 \ar@{^{(}->}[r]^-{c} & I \\
 }
\]
where the arrows decorated with ``$o$'' are open embeddings, and those decorated with ``$c$'' are closed embeddings. (To justify the claim about the embedding $K \subset K_0 \cap J$, we observe that the complement of this embedding is $K_1 \cap J$, which is closed in $K_0 \cap J$. For the embedding $K_0 \cap J \subset K_0$, one simply observes that $K_0 \cap J = K_0 \smallsetminus J_1$.) Then the desired isomorphisms follow from Lemma~\ref{le:recollement and locally closed inclusion}, the compatibility of pushforward under closed embeddings and pullback under open embeddings with composition (see~\S\ref{ss:inclusions}), and Lemma~\ref{le:recollement compositions}.
\end{proof}

\section{Study of standard and costandard objects}
\label{sec:standard-costandard}

\subsection{Generation of the categories by reduced expressions}

We begin with the following lemma. Recall the notion of ``rex moves'' (see~\cite[\S 4.2]{ew} or~\cite[\S 4.3]{rw}), and the associated morphisms in $\Diag_{\BS}^\oplus(\h,W)$.

\begin{lema}
\label{lem:cone-rex-smaller}
Let $\ux$ and $\uy$ be reduced expressions for an element $w \in W$. Consider a rex move $\ux \leadsto \uy$, and denote by $f : B_{\ux} \to B_{\uy}$ the associated morphism. Then the cone of $f$ belongs to $\BE_{\{<w\}}(\h,W)$.
\end{lema}

\begin{proof}
Consider also the ``reversed'' rex move $\uy \leadsto \ux$, and denote by $g : B_{\uy} \to B_{\ux}$ the associated morphism. Then by~\cite[Lemma~7.4 and Lemma~7.5]{ew}, there exists an object $B$ in $\Diag_{\BS,\{<w\}}^\oplus(\h,W)$ and morphisms $h_1 : B_{\ux} \to B$ and $h_2 : B \to B_\ux$ such that
\[
gf = \id_{B_{\ux}} + h_2 \circ h_1.
\]
Then we can consider the morphisms of complexes
\[
\xymatrix@C=1.4cm{
\cdots \ar[r] & 0 \ar[r] & B_{\ux} \ar[r]^-{gf} \ar@<2pt>[d]^-{h_1} & B_{\ux} \ar[r] \ar@<2pt>[d]^-{h_1} & 0 \ar[r] & \cdots \\
\cdots \ar[r] & 0 \ar[r] & B \ar[r]^-{\id_B + h_1h_2} \ar@<2pt>[u]^-{-h_2} & B \ar[r] \ar@<2pt>[u]^-{-h_2} & 0 \ar[r] & \cdots
}
\]
It is not difficult to check that the images of these morphisms are inverse isomorphisms in $\BE(\h,W)$. In particular, the cone of $gf$ belongs to $\BE_{\{<w\}}(\h,W)$. Similar arguments show that the cone of $fg$ belongs to $\BE_{\{<w\}}(\h,W)$, and this implies that the image of $f$ in the Verdier quotient $\BE(\h,W) / \BE_{\{<w\}}(\h,W)$ is an isomorphism, i.e.~that the image of the cone $\mathscr{C}$ of $f$ in $\BE(\h,W) / \BE_{\{<w\}}(\h,W)$ is trivial.
In view of~\cite[Proposition~4.6.2]{krause}, this means that there exists an object $\mathscr{F}$ in $\BE_{\{<w\}}(\h,W)$ such that the identity of $\mathscr{C}$ factors as a composition $\mathscr{C} \to \mathscr{F} \to \mathscr{C}$. We deduce that $(i_{\{<w\}}^{\{\leq w\}})^* \mathscr{C}=0$. By the recollement formalism (see Proposition~\ref{prop:recollement}) it follows that $\mathscr{C}$ belongs to $\BE_{\{<w\}}(\h,W)$, as desired.
\end{proof}

Let us denote by ``$*$'' the \emph{Hecke product} on $W$ studied e.g.~in~\cite[\S 3]{bm}. (Recall in particular that this product is associative.)
For $\uw=(s_1, \ldots, s_r)$ an expression, we set
\[
*\uw := s_1 * \cdots * s_r \quad \in W.
\]

\begin{lema}
\label{lem:Buw-leq*uw}
For any expression $\uw$, the object $B_\uw$ belongs to $\BE_{\{ \leq *\uw \}}(\h,W)$.
\end{lema}

\begin{proof}
We argue by induction on $\ell(\uw)$. Of course the claim is obvious if $\uw$ is a reduced expression, and in particular when $\ell(\uw)=0$. Now, let $\uw$ be a nonempty expression, and assume the claim is known for expressions of strictly smaller length. Write $\uw = \uy s$ for some $s \in S$; then by induction we know that $B_{\uy} \in \BE_{\{\leq *\uy\}}(\h,W)$. In view of the definition of $\BE_{\{\leq *\uy\}}(\h,W)$, we therefore have to show that if $\uz$ is a reduced expression for an element $z \leq *\uy$, then $B_{\uz s} \in \BE_{\{\leq *\uw\}}(\h,W)$. 

If $\ell(\uz) < \ell(\uy)$, then $\ell(\uz s) < \ell(\uw)$, so by induction we know that $B_{\uz s}$ belongs to $\BE_{\leq *(\uz s)}(\h,W)$. On the other hand, by \cite[Proposition~3.1]{bm} we have $*(\uz s) = z * s \leq (* \uy) * s = *\uw$, and hence the desired claim follows in this case.

Assume now that $\ell(\uz)=\ell(\uy)$ (so that $z=*\uy$ and $\uy$ is a reduced expression for $z$).
If $zs>z$ then $*\uw = (*\uy) s$ and $\uz s$ is a reduced expression for $*\uw$; hence the claim is clear from definitions. Now, assume that $zs<z$ (so that $*\uw=z$). Choose a reduced expression $\uz'$ for $z$ ending with $s$, and a rex move $\uz \leadsto \uz'$.
By Lemma~\ref{lem:cone-rex-smaller}, the cone of the associated morphism $f : B_{\uz} \to B_{\uz'}$ belongs to $\BE_{\{<z\}}(\h,W)$; as above, using the induction hypothesis, this implies that the cone of $f \star B_s$ belongs to $\BE_{ \{\leq *\uw\}}(\h,W)$. Since
\[
B_{\uz'} \star B_s \cong B_{\uz'}(1) \oplus B_{\uz'}(-1)
\]
by~\eqref{eqn:BsBs}, so that $B_{\uz'} \star B_s$ belongs to $\BE_{\{\leq z\}}(\h,W)$, and since $z=*\uw$, we finally deduce that $B_{\uz s}$ belongs to $\BE_{\{\leq * \uw\}}$, as desired.
\end{proof}

\begin{rmk}\phantomsection
\label{rmk:generate}
\begin{enumerate}
\item
\label{it:generate}
Note that Lemma~\ref{lem:Buw-leq*uw} implies in particular that the category $\BE(\h,W)$ is generated (as a triangulated category) by the objects $B_\uw$ where $\uw$ is a reduced expression; in other words the canonical embedding $\BE_W(\h,W) \to \BE(\h,W)$ is an equivalence of categories. (Of course, this fact follows readily from~\cite[Theorem~6.26]{ew} when this result applies, i.e.~when $\bk$ is a field or a complete local ring.) In the rest of the paper we will identify these categories without further notice.
\item
Statements closely related to Lemma~\ref{lem:Buw-leq*uw} and the comment in~\eqref{it:generate}
appear as~\cite[Lemmas~5.23 and 5.24]{rw}. But the proof in~\cite{rw} has a gap (since a variant of Lemma~\ref{lem:cone-rex-smaller} is asserted without details). It turns out that the recollement formalism exactly provides the tools needed to fill this gap.
\end{enumerate}
\end{rmk}

Below we will also use the following consequence of Lemma~\ref{lem:Buw-leq*uw}.

\begin{cor}
\label{cor:multiplication-Bs}
Let $I \subset W$ be a closed subset, and let $s \in S$ be such that $I$ is stable under the map $x \mapsto xs$. Then the subcategory $\BE_I(\h,W)$ of $\BE(\h,W)$ is stable under right multiplication by $B_s$.
\end{cor}

\begin{proof}
We have to prove that if $\uw$ is a reduced expression for an element in $I$, then $B_{\uw} \ustar B_s = B_{\uw s}$ belongs to $\BE_I(\h,W)$. However Lemma~\ref{lem:Buw-leq*uw} implies that this object belongs to $\BE_{\{\leq *(\uw s)\}}(\h,W)$. Under our assumption $*(\uw s) \in I$, so $\{\leq *(\uw s)\} \subset I$, and the claim follows.
\end{proof}

\subsection{Inclusions of singletons}
\label{ss:inclu-singletons}

Let $I \subset W$ be a finite locally closed subset. Then for any $x \in I$, the subset $\{x\} \subset I$ is locally closed. Hence we can consider in particular the functors associated with this inclusion, which for simplicity will be denoted
\[
 (i_x^I)_*, \quad (i_x^I)_!, \quad (i_x^I)^*, \quad (i_x^I)^!.
\]

\begin{lema}
\label{lem:pushforward-pullback-singleton}
 If $J \subset I$ is a closed subset and if $x \notin J$ then
 \[
  (i_x^I)^! \circ (i_J^I)_* = 0 \quad \text{and} \quad (i_x^I)^* \circ (i_J^I)_* = 0.
 \]
\end{lema}

\begin{proof}
 The first equality follows from the second one by duality, using~\eqref{eq:duality and recollement for locally closed inclusions} and~\eqref{eqn:pushforward-closed}. And to prove the second equality we remark that $(i_x^I)^* = (i_x^{I \smallsetminus J})^* \circ (i_{I \smallsetminus J}^I)^*$ by Lemma~\ref{le:recollement compositions for locally closed inclusions}, so that
 \[
  (i_x^I)^* \circ (i_J^I)_* = (i_x^{I \smallsetminus J})^* \circ (i_{I \smallsetminus J}^I)^* \circ (i_J^I)_* = 0
 \]
by~\eqref{eqn:comp-open-closed-0}.
\end{proof}

Lemma~\ref{lem:pushforward-pullback-singleton} implies that if $x \neq y$ are both in $I$, we have
\begin{equation}
\label{eqn:restriction-pushforward}
(i_x^I)^! \circ (i_y^I)_* = 0, \qquad (i_x^I)^* \circ (i_y^I)_! = 0
\end{equation}
(because $(i_y^I)_* = (i_{\{z \in I \mid z \leq y\}}^I)_* \circ (i_y^{\{z \in I \mid z \leq y\}})_*$ and similarly for $(i_y^I)_!$).
On the other hand,
for any $x \in I$ we have
\begin{equation}
\label{eqn:restriction-pushforward-2}
(i_x^I)^! \circ (i_x^I)_* \cong \id, \qquad (i_x^I)^* \circ (i_x^I)_! \cong \id.
\end{equation}
For instance, for the first isomorphism we remark that $(i_x^I)_* \cong (i_{\{z \in I \mid z \leq x\}}^I)_! \circ (i_x^{\{z \in I \mid z \leq x\}})_*$ by Lemma~\ref{le:recollement compositions for locally closed inclusions} and~\eqref{eqn:pushforward-closed} and $(i_x^I)^! \cong (i_x^{\{z \in I \mid z \leq x\}})^* \circ (i_{\{z \in I \mid z \leq x\}}^I)^!$ by Lemma~\ref{le:recollement compositions for locally closed inclusions} and~\eqref{eqn:pullback-open}. Then the claim follows from the invertibility of the morphisms in~\eqref{eqn:pullback-pushforward-inverse}.

\subsection{Definition of standard and costandard objects}
\label{ss:def-D-N}

Now, recall the object $b_w$ of $\Diag_{\BS,\{w\}}^\oplus(\h,W)$ defined in~\S\ref{ss:Diag-singleton}. Identifying this object with the complex concentrated in degree $0$ and with $0$th term $b_w$, it can be considered as an object in $\BE_{\{w\}}(\h,W)$. The corresponding \emph{standard} and \emph{costandard} objects in $\BE_I(\h,W)$ are defined by
\[
 \Delta^I_w := (i_w^I)_! b_w, \qquad \nabla^I_w := (i_w^I)_* b_w.
\]
The main property of these objects is the following.

\begin{lema}
\label{lem:Hom-BE-D-N}
Let $I \subset W$ be a finite locally closed subset, and let $x,y \in I$. Then we have
\begin{equation*}
 \Hom_{\BE_I(\h,W)}(\Delta^I_x, \nabla^I_y \langle n \rangle [m]) \cong
 \begin{cases}
   R^m=\mathrm{S}^{m/2}(V^*) & \text{if $x=y$ and $m=-n \in 2\Z_{\geq 0}$;}\\
   0 & \text{otherwise.}
 \end{cases}
\end{equation*}
\end{lema}

\begin{proof}
This follows from adjunction, isomorphisms~\eqref{eqn:restriction-pushforward}--\eqref{eqn:restriction-pushforward-2} and Lemma~\ref{le:Dw is equivalent to free R}.
\end{proof}

\begin{example}\phantomsection
\label{ex:Des}
\begin{enumerate}
\item
\label{it:D-N-e}
If $w$ is minimal in $I$, then $\Delta_w^I = \nabla_w^I$ by~\eqref{eqn:pushforward-closed}, and this object is the image of $B_{\uw}$ in $\BE_I(\h,W)$, where $\uw$ is any reduced expression for $w$. In particular, if $I$ contains the neutral element $e \in W$, then $\Delta^I_e=\nabla^I_e$ is the image of $B_{\varnothing}$.
\item
\label{it:Ds}
Let $s\in S$. 
In view of Example~\ref{ex:w-ws}, the complex $\Delta_s^{\{e,s\}}$ coincides with the complex
\[
 \cdots 0 \to B_s \xrightarrow{\begin{tikzpicture}[thick,scale=0.07,baseline]
      \draw (0,-5) to (0,0);
      \node at (0,0) {$\bullet$};
      \node at (0,-6.7) {\tiny $s$};
    \end{tikzpicture}} B_\varnothing(1) \to 0 \cdots
\]
where the nonzero terms are in degrees $0$ and $1$ respectively. Similarly, $\nabla_s^{\{e,s\}}$ is the complex
\[
 \cdots 0 \to B_\varnothing(-1) \xrightarrow{\begin{tikzpicture}[thick,baseline,xscale=0.07,yscale=-0.07]
      \draw (0,-5) to (0,0);
      \node at (0,0) {$\bullet$};
      \node at (0,-6.7) {\tiny $s$};
    \end{tikzpicture}} B_s \to 0 \cdots
\]
where the nonzero terms are in degrees $-1$ and $0$ respectively. These complexes in fact describe $\Delta_s^I$ and $\nabla_s^I$ for any $I$ containing $e$ and $s$. In particular, our present notation is compatible with that used in~\cite[Example~4.2.2]{amrw1}.
\end{enumerate}
\end{example}

It will sometimes be convenient to have standard and costandard objects also when $I$ is not finite. For a general $I$ and any $w \in I$, we define $\Delta_w^I$ and $\nabla_w^I$ by
\[
\Delta_w^I := (i_J^I)_* \Delta_w^J, \qquad \nabla_w^I := (i_J^I)_* \nabla_w^J
\]
where $J \subset I$ is any finite closed subset containing $w$. It is easy to check that these objects do not depend on the choice of $J$, up to canonical isomorphism, and that Lemma~\ref{lem:Hom-BE-D-N} still holds in this generality. When $I=W$, we will sometimes omit the superscript in this notation.

\subsection{First properties}
\label{ss:prop-D-N}

\begin{lema}
\label{lem:pf-pb-D-N}
Let $I \subset W$ be a finite locally closed subset, and let $J \subset I$ be a locally closed subset. Then for any $w \in J$ we have
\[
(i_J^I)_! \Delta_w^J \cong \Delta_w^I, \qquad (i_J^I)_* \nabla_w^J \cong \nabla_w^I,
\]
and for any $w \in I$ we have
\[
(i_J^I)^* \Delta_w^I \cong \begin{cases}
\Delta_w^J & \text{if $w \in J$;} \\
0 & \text{otherwise,} 
\end{cases} \qquad
(i_J^I)^! \nabla_w^I \cong \begin{cases}
\nabla_w^J & \text{if $w \in J$;} \\
0 & \text{otherwise.}
\end{cases}
\]
\end{lema}

\begin{proof}
The first two isomorphisms follow from Lemma~\ref{le:recollement compositions for locally closed inclusions}. For the other isomorphisms, we treat the case of $(i_J^I)^* \Delta_w^I$; the case of $(i_J^I)^! \nabla_w^I$ is similar. It suffices to prove these isomorphisms when $J$ is either closed or open. First, assume that $J$ is closed.  If $w \notin J$, then the desired vanishing follows from the first isomorphism in Lemma~\ref{lem:pushforward-pullback-singleton} and adjunction. 
If $w \in J$, then we have
\[
(i_J^I)^* \Delta_w^I \cong (i_J^I)^* (i_J^I)_* \Delta_w^J
\]
by~\eqref{eqn:pushforward-closed} and Lemma~\ref{le:recollement compositions for locally closed inclusions}, and the claim follows from the invertibility of the first morphism in~\eqref{eqn:pullback-pushforward-inverse}.  Now, assume that $J$ is open. If $w \in J$, then using~\eqref{eqn:pullback-open} we have
\[
(i_J^I)^* \Delta_w^I \cong (i_J^I)^! (i_J^I)_! \Delta_w^J \cong \Delta_w^J.
\]
And if $w \notin J$ then
\[
(i_J^I)^* \Delta_w^I \cong (i_J^I)^* (i_{I \smallsetminus J}^I)_* \Delta_w^{I \smallsetminus J}
\]
and the desired vanishing holds by~\eqref{eqn:comp-open-closed-0}.
\end{proof}

Another important property of standard and costandard objects is provided by the following observation.

\begin{lema}
\label{lema:BE-generated-D-N}
For any locally closed subset $I \subset W$, the category $\BE_I(\h,W)$ is generated, as a triangulated category, by the objects of the form $\Delta^I_w(m)$ with $w \in I$ and $m \in \Z$, as well as by the objects of the form $\nabla^I_w(m)$ with $w \in I$ and $m \in \Z$.
\end{lema}

\begin{proof}
We treat the case of the standard objects; the other case is similar (or follows by duality). We can clearly assume that $I$ is finite, and proceed by induction on $|I|$.

In case $|I|=1$,
the lemma is clear from Lemma~\ref{le:Dw is equivalent to free R}. Now, assume $|I|>1$, and choose $w \in I$ minimal. Then any object $\mathscr{F}$ in $\BE_I(\h,W)$ fits in a distinguished triangle
\[
(i_{I \smallsetminus \{w\}}^I)_! (i_{I \smallsetminus \{w\}}^I)^* \mathscr{F} \to \mathscr{F} \to (i_w^I)_* (i_w^I)^* \mathscr{F} \xrightarrow{[1]}.
\]
By induction $(i_{I \smallsetminus \{w\}}^I)^* \mathscr{F}$ belongs to the triangulated subcategory of $\BE_{I \smallsetminus \{w\}}(\h,W)$ generated by the objects $\Delta_x^{I \smallsetminus \{w\}}(m)$ with $x \in I \smallsetminus \{w\}$. Since $(i_{I \smallsetminus \{w\}}^I)_! \Delta_x^{I \smallsetminus \{w\}} \cong \Delta_x^I$ for such $x$ (see Lemma~\ref{lem:pf-pb-D-N}), we deduce that $(i_{I \smallsetminus \{w\}}^I)_! (i_{I \smallsetminus \{w\}}^I)^* \mathscr{F}$ belongs to the triangulated subcategory of $\BE_I(\h,W)$ under consideration. It is easy to see that $(i_w^I)_* (i_w^I)^* \mathscr{F}$ belongs to the triangulated subcategory generated by the objects $\Delta_w^I(m)=\nabla_w^I(m)$, and the proof is complete.
\end{proof}

\subsection{Convolution of standard and costandard objects}

\begin{lema}
\label{lem:triangle-Dw-ws}
Let $w\in W$ and $s\in S$ be such that $ws>w$. Then there exist distinguished triangles
\begin{gather*}
\Delta_{w}\langle-1\rangle \to \Delta_{ws} \to \Delta_w\ustar B_{s} \xrightarrow{[1]}, \quad
\nabla_w\ustar B_{s} \to \nabla_{ws} \to \nabla_{w}\langle1\rangle \xrightarrow{[1]}
\end{gather*}
in $\BE(\h,W)$, in which the third arrows
are generators of the free rank-$1$ $\bk$-modules
\[
\Hom_{\BE(\h,W)}(\Delta_w\ustar B_{s}, \Delta_w \langle -1 \rangle [1]) \quad \text{and} \quad \Hom_{\BE(\h,W)}(\nabla_{w}\langle1\rangle, \nabla_w\ustar B_{s}[1])
\]
respectively.
\end{lema}

\begin{proof}
 We will construct the first triangle; the second one can then be obtained by duality (or by similar arguments). We set $I:=\{z \in W \mid z \leq ws\}$. In this triangle all the objects live in $\BE_I(\h,W)$ (see Corollary~\ref{cor:multiplication-Bs} for the third term); therefore we can perform all the computations in this subcategory. To simplify notation, we will also set $J:=I \smallsetminus \{w,ws\}$ (a closed subset of $I$).
 
 Let $\uw$ be a reduced expression for $w$, and recall the triangle constructed in Example~\ref{ex:w-ws}. Applying the functor $(i_{\{w,ws\}}^I)_!$ we deduce (cf.~Remark~\ref{rmk:loc-closed-tri}) a distinguished triangle
 \begin{equation}
 \label{eqn:triangle-w-ws}
  \Delta^I_{w}\langle-1\rangle \to \Delta^I_{ws} \to \left( i_{\{w,ws\}}^I \right)_! (B_{\uw s})\xrightarrow{[1]}
 \end{equation}
in $\BE_I(\h,W)$ (where we write $B_{\uw s}$ for the image of this object in the category $\BE_{\{w,ws\}}(\h,W)$). Hence to conclude the construction of the triangle it suffices to construct an isomorphism
\begin{equation}
\label{eqn:isom-triangle-Dw-Dws}
 \left( i_{\{w,ws\}}^I \right)_! (B_{\uw s}) \cong \Delta^I_w\ustar B_{s}.
\end{equation}

First, we remark that
\begin{equation}
\label{eqn:isom-triangle-Dw-Dws-2}
( i_{J}^I )^*(\Delta^I_w\ustar B_{s})=0.
\end{equation}
In fact, if $\mathscr{F}$ belongs to $\BE_{J}(\h,W)$ we have
\begin{align*}
 \Hom_{\BE_{J}(\h,W)} \bigl( ( i_{J}^I \bigr)^*(\Delta^I_w\ustar B_{s}), \mathscr{F} \bigr)
 &\cong \Hom_{\BE_{I}(\h,W)}\bigl( \Delta^I_w\ustar B_{s}, ( i_{J}^I )_* \mathscr{F} \bigr) \\
 &\cong \Hom_{\BE_{I}(\h,W)}\Bigl( \Delta^I_w, \bigl( ( i_{J}^I )_* \mathscr{F} \bigr) \ustar B_{s} \Bigr).
\end{align*}
It follows from Corollary~\ref{cor:multiplication-Bs} that $\bigl( ( i_{J}^I )_* \mathscr{F} \bigr) \ustar B_{s}$ belongs to the essential image of $\BE_J(\h,W)$, and then from~\eqref{eqn:comp-open-closed-0} we deduce that
\[
 \Hom_{\BE_{J}(\h,W)} \bigl( ( i_{J}^I \bigr)^*(\Delta^I_w\ustar B_{s}), \mathscr{F} \bigr)=0,
\]
which implies~\eqref{eqn:isom-triangle-Dw-Dws-2}.

From~\eqref{eqn:isom-triangle-Dw-Dws-2} we deduce that adjunction induces an isomorphism
\[
 (i_{\{w,ws\}}^I)_! (i_{\{w,ws\}}^I)^* (\Delta^I_w\ustar B_{s}) \xrightarrow{\sim} \Delta^I_w\ustar B_{s}.
\]
Hence to prove~\eqref{eqn:isom-triangle-Dw-Dws} it suffices to prove that
\begin{equation}
\label{eqn:isom-triangle-Dw-Dws-3}
(i_{\{w,ws\}}^I)^* (\Delta^I_w\ustar B_{s}) \cong B_{\uw s}
\end{equation}
in $\BE_{\{w,ws\}}(\h,W)$.
However there exists a natural distinguished triangle
\[
 \Delta^I_w \to B_{\uw} \to (i_J^I)_* (i_J^I)^* B_{\uw} \xrightarrow{[1]}.
\]
Applying the functor $(i_{\{w,ws\}}^I)^* ( - \ustar B_{s})$ and Corollary~\ref{cor:multiplication-Bs} once again we deduce the isomorphism~\eqref{eqn:isom-triangle-Dw-Dws-3}, and hence finally~\eqref{eqn:isom-triangle-Dw-Dws}.

To conclude the proof, it remains to prove that the $\bk$-module
\[
\Hom_{\BE(\h,W)} \bigl( ( i_{\{w,ws\}}^I )_! (B_{\uw s}), \Delta^I_{w}\langle-1\rangle[1] \bigr)
\]
is free of rank $1$, and generated by the morphism
appearing in~\eqref{eqn:triangle-w-ws}. However, 
as noted after~\eqref{eqn:pullback-pushforward-inverse}, the functor $(i_{\{w,ws\}}^I)_!$ is fully faithful. Hence it suffices to prove the corresponding claim for
\[
\Hom_{\BE_{\{w,ws\}}(\h,W)}( B_{\uw s}, B_{\uw} \langle-1\rangle[1]).
\]
This claim is clear from the construction in Example~\ref{ex:w-ws}.
\end{proof}

The next proposition solves a question raised in~\cite[Remark~4.2.3]{amrw1}.

\begin{prop}
\label{prop:D-N-convolution}
Let $w \in W$.
\begin{enumerate}
\item
\label{it:Delta-nabla-convolution-1}
If $(s_1, \ldots, s_r)$ is a reduced expression for $w$, then we have
\[
\Delta_w \cong \Delta_{s_1} \ustar \Delta_{s_2} \ustar \cdots \ustar \Delta_{s_r}, \qquad
\nabla_w \cong \nabla_{s_1} \ustar \nabla_{s_2} \ustar \cdots \ustar \nabla_{s_r}.
\]
\item
\label{it:Delta-nabla-convolution-2}
We have isomorphisms
\[
\Delta_w \ustar \nabla_{w^{-1}} \cong \nabla_{w^{-1}} \ustar \Delta_w \cong B_\varnothing.
\]
\end{enumerate}
\end{prop}

\begin{proof}
We will prove the claims by induction on $\ell(w)$. We note here that~\eqref{it:Delta-nabla-convolution-2} holds when $\ell(w)=1$ by~\cite[Lemma~4.2.4]{amrw1}. In particular, it follows that~\eqref{it:Delta-nabla-convolution-2} is a consequence of~\eqref{it:Delta-nabla-convolution-1} (applied to $w$ and $w^{-1}$).

Of course, if $\ell(w)=0$ there is nothing to prove. Now let $w \in W \smallsetminus \{e\}$, and assume the claims are known for elements of length strictly smaller than that of $w$. We will prove the first isomorphism in~\eqref{it:Delta-nabla-convolution-1} for $w$; the second one can be proved similarly (or follows by duality), and as noted above~\eqref{it:Delta-nabla-convolution-2} will follow. Let $(s_1, \ldots, s_r)$ be a reduced expression for $w$, and let $y:=s_1 \cdots s_{r-1}$ and $s := s_r$ (so that $w=ys$). Using~\eqref{it:Delta-nabla-convolution-1} for $y$ (which is known by induction) we know that
\[
\Delta_y \cong \Delta_{s_1} \ustar \Delta_{s_2} \ustar \cdots \ustar \Delta_{s_{r-1}},
\]
hence to conclude it suffices to prove that
\[
\Delta_w \cong \Delta_y \ustar \Delta_s.
\]

The special case of Lemma~\ref{lem:triangle-Dw-ws} for the neutral element $e$ provides a distinguished triangle
\[
B_\varnothing \langle -1 \rangle \to \Delta_s \to B_s \xrightarrow{[1]}
\]
in which the third arrow is a generator of $\Hom_{\BE(\h,W)}(B_s, B_\varnothing \langle -1 \rangle[1])$, a free rank-$1$ $\bk$-module. Now~\eqref{it:Delta-nabla-convolution-2} for $y$ implies that the functor
\[
\Delta_y \ustar (-) : \BE(\h,W) \to \BE(\h,W)
\]
is an equivalence of triangulated categories (with quasi-inverse $\nabla_{y^{-1}} \ustar (-)$). Hence applying this functor we obtain a distinguished triangle
\[
\Delta_y \langle -1 \rangle \to \Delta_y \ustar \Delta_s \to \Delta_y \ustar B_s \xrightarrow{[1]}
\]
in which the third arrow is a generator of $\Hom_{\BE(\h,W)}(\Delta_y \ustar B_s, \Delta_y \langle -1 \rangle[1])$ (a free rank-$1$ $\bk$-module). Comparing with the triangle of Lemma~\ref{lem:triangle-Dw-ws} (now for $y$) we deduce the isomorphism $\Delta_w \cong \Delta_y \ustar \Delta_s$, as expected.
\end{proof}

\begin{rmk}\phantomsection
\label{rmk:rouquier-complexes}
\begin{enumerate}
\item
\label{it:rouquier-complexes}
Proposition~\ref{prop:D-N-convolution}\eqref{it:Delta-nabla-convolution-1} shows that the objects $\Delta_w$ ($w \in W$) are generalizations of the \emph{Rouquier complexes} from~\cite{rouquier} associated with canonical lifts of elements of $W$ to the braid group of $(W,S)$.
More precisely, consider a reflection faithful representation $V$ of $(W,S)$ over $\bk=\mathbb{R}$ as constructed by Soergel or arising from a symmetrizable Kac--Moody group; see~\cite[Proposition~1.1]{riche-bourbaki}. Then, as explained in~\cite[Example~3.3(2)--(4)]{ew}, there exists a natural realization $\h$ of $(W,S)$ with underlying vector space $V$. Moreover, by~\cite[Theorem~6.30]{ew} there exists a canonical equivalence of graded additive categories between the Karoubian envelope of $\DiagBS^\oplus(\h,W)$ and the associated category of Soergel bimodules. Under the induced equivalence between bounded homotopy categories (see~\cite[Lemma~4.9.1]{amrw1}), Proposition~\ref{prop:D-N-convolution} shows that $\Delta_w$ corresponds to the Rouquier complex (as defined in~\cite[Proposition~9.4]{rouquier}) associated with the canonical lift of $w$ to the braid group. From this point of view, Lemma~\ref{lem:Hom-BE-D-N} is a generalization of the main result of~\cite{lw}.
\item
Proposition~\ref{prop:D-N-convolution}\eqref{it:Delta-nabla-convolution-1} suggests a different approach to our study, starting with a ``direct'' definition of standard and costandard objects. However, from such a definition it seems to be difficult (at least to the authors) to prove that such objects have the properties they ought to possess (like independence of the reduced expression, or Lemma~\ref{lem:Hom-BE-D-N}).
\end{enumerate}
\end{rmk}

\subsection{Application}
\label{ss:Groth-gp}

In this subsection we apply the results of this section to describe the split Grothendieck group of the additive category $\Diag_{\BS}^\oplus(\h,W)$. If $\mathcal{A}$ is an essentially small triangulated category, resp.~
additive category, we denote by $[\mathcal{A}]_\Delta$, resp.~$[\mathcal{A}]_\oplus$, the Grothendieck group of $\mathcal{A}$, resp.~the split Grothendieck group of $\mathcal{A}$. Recall also the Hecke algebra $\mathcal{H}_{(W,S)}$ associated with the Coxeter system $(W,S)$, where we follow the conventions of~\cite{soergel}. With this notation introduced, we can state our result more precisely, as follows.

\begin{theorem}
\label{thm:Groth-gp}
There exists a unique ring isomorphism
\[
\mathcal{H}_{(W,S)} \xrightarrow{\sim} [\Diag_{\BS}^\oplus(\h,W)]_\oplus
\]
sending $v$ to $[B_\varnothing(1)]$ and $\underline{H}_s=H_s+v$ to $[B_s]$ (for any $s \in S$).
\end{theorem}

\begin{rmk}
In the case $\bk$ is a complete local ring, this result appears as~\cite[Corollary~6.27]{ew}. (In this reference, the result is stated in terms of the Karoubian hull of $\DiagBS^\oplus(\h,W)$. However, it is easy to deduce from~\cite[Theorem~6.26]{ew} that under their assumption the natural functor from $\DiagBS^\oplus(\h,W)$ to its Karoubian hull induces an isomorphism on split Grothendieck groups.) The fact that our methods might allow one to generalize this result was suggested to one of us by G.~Williamson.
\end{rmk}

The proof of Theorem~\ref{thm:Groth-gp} will use the following lemma, which is the main result of~\cite{rose}.

\begin{lema}
\label{lem:rose}
For any essentially small additive category $\mathcal{A}$, the natural group morphism
\[
[\mathcal{A}]_\oplus \to [\Kb \mathcal{A}]_\Delta
\]
is an isomorphism.
\end{lema}

\begin{proof}[Proof of Theorem~{\rm \ref{thm:Groth-gp}}]
In view of Lemma~\ref{lem:rose}, the natural morphism
\[
[\Diag_{\BS}^\oplus(\h,W)]_\oplus \to [\BE(\h,W)]_\Delta
\]
is an isomorphism. Moreover this morphism is clearly a ring morphism. Therefore, to prove the theorem we only have to prove that there exists a unique isomorphism
\[
\mathcal{H}_{(W,S)} \xrightarrow{\sim} [\BE(\h,W)]_\Delta
\]
sending $v$ to $[B_\varnothing(1)]$ and $H_s+v$ to $[B_s]$. Uniqueness is clear, since $\mathcal{H}_{(W,S)}$ is generated (as a ring) by $v$ and the elements $H_s+v$ for $s \in S$.

To prove existence, we first remark that the classes of the standard objects $[\Delta_w(m)]$ form a $\Z$-basis of $[\BE(\h,W)]_\Delta$. In fact, Lemma~\ref{lema:BE-generated-D-N} and Remark~\ref{rmk:generate}\eqref{it:generate} imply that these classes span $[\BE(\h,W)]_\Delta$. On the other hand, assume for a contradiction that there exists a relation
\[
\sum_{\substack{x \in Y_1 \\ m \in \Z}} \lambda_{x,m} \cdot [\Delta_x(m)] = \sum_{\substack{y \in Y_2 \\ n \in \Z}} \lambda_{y,n} \cdot [\Delta_y(n)]
\]
for some disjoint finite subsets $Y_1,Y_2 \subset W$ (with $Y_1 \neq \varnothing$ and $\lambda_{x,m} \neq 0$ for at least one $x \in Y_1$ and $m \in \Z$) and some integers $\lambda_{y,n} \in \Z_{\geq 0}$ (with $\lambda_{x,m}=0$ and $\lambda_{y,n}=0$ for all but finitely many $m$'s and $n$'s). Then, if we set
 \[
 \mathscr{X}_1:= \bigoplus_{\substack{x \in Y_1 \\ m \in \Z}} \bigl( \Delta_x(m) \bigr)^{\oplus \lambda_{x,m}}, \quad
 \mathscr{X}_2:= \bigoplus_{\substack{y \in Y_2 \\ n \in \Z}} \bigl( \Delta_y(n) \bigr)^{\oplus \lambda_{y,n}},
 \]
 by~\cite[Lemma~2.4]{thomason} there exist objects $\mathscr{C}$, $\mathscr{C}'$, $\mathscr{C}''$ and distinguished triangles
 \[
 \mathscr{C} \oplus  \mathscr{X}_1 \to \mathscr{C}' \to \mathscr{C}'' \xrightarrow{[1]}, \quad
 \mathscr{C} \oplus  \mathscr{X}_2 \to \mathscr{C}' \to \mathscr{C}'' \xrightarrow{[1]}.
 \]
 There exists a finite closed subset $I \subset W$ such that all the objects above belong to $\BE_I(\h,W)$. Then
 choose $x \in Y_1$ such that $\lambda_{x,m} \neq 0$ for at least one $m$. Applying $(i_x^I)^*$ and using Lemma~\ref{lem:pf-pb-D-N} we obtain distinguished triangles
 \[
 (i_x^I)^*\mathscr{C} \oplus  (i_x^I)^* \mathscr{X}_1 \to (i_x^I)^*\mathscr{C}' \to (i_x^I)^*\mathscr{C}'' \xrightarrow{[1]}, \quad
 (i_x^I)^*\mathscr{C} \to (i_x^I)^*\mathscr{C}' \to (i_x^I)^*\mathscr{C}'' \xrightarrow{[1]}.
 \]
 Hence the class of $(i_x^I)^* \mathscr{X}_1$ in $[\BE_{\{x\}}(\h,W)]_\Delta$ vanishes. But Lemma~\ref{le:Dw is equivalent to free R} and Lem\-ma~\ref{lem:rose} imply that the classes $[b_x(m)]$ with $m \in \Z$ form a basis of $[\BE_{\{x\}}(\h,W)]_\Delta$, and by construction the coefficient of $(i_x^I)^* \mathscr{X}_1$ on $[b_x(m)]$ is $\lambda_{x,m}$. One of these coefficients is nonzero, providing the desired contradiction.
 
 We now prove that the assignment
 \[
 v \mapsto [B_\varnothing(1)], \qquad H_w \mapsto [\Delta_w] \ \text{ ($w \in W$)}
 \]
 induces a ring morphism $\mathcal{H}_{(W,S)} \to [\BE(\h,W)]_\Delta$. For this we have to prove that
 \begin{equation}
 \label{eqn:rel-H-1}
 \bigl( [\Delta_s] \bigr)^2 = [\Delta_e] + [\Delta_s(-1)] - [\Delta_s(1)]
 \end{equation}
 for $s \in S$ and that for $x,y \in W$ such that $\ell(xy)=\ell(x)+\ell(y)$ we have
 \begin{equation}
 \label{eqn:rel-H-2}
 [\Delta_{xy}] = [\Delta_x \ustar \Delta_y].
 \end{equation}
 Here~\eqref{eqn:rel-H-2} follows from Proposition~\ref{prop:D-N-convolution}, while~\eqref{eqn:rel-H-1} follows from the fact that $[\Delta_s] = [B_s] - [\Delta_e(1)]$ (see Example~\ref{ex:Des}\eqref{it:Ds}) and the isomorphism~\eqref{eqn:BsBs}.

 Finally we argue that our morphism $\mathcal{H}_{(W,S)} \to [\BE(\h,W)]_\Delta$ is invertible because it sends a $\Z$-basis of $\mathcal{H}_{(W,S)}$ to a $\Z$-basis of $[\BE(\h,W)]_\Delta$; moreover it sends $v$ to $B_{\varnothing}(1)$ by definition, and $H_s+v$ to $[B_s]$ since (as already noticed above) we have $[\Delta_s] = [B_s] - [\Delta_e(1)]$ in $[\BE(\h,W)]_\Delta$.
\end{proof}

\begin{rmk}
 Let us notice that, viewed as an isomorphism $\mathcal{H}_{(W,S)} \xrightarrow{\sim} [\BE(\h,W)]_\Delta$, the isomorphism of Theorem~\ref{thm:Groth-gp} is very explicit: it sends $H_w$ to $[\Delta_w]$.
\end{rmk}

\section{The perverse t-structure}\label{sec:BE t structure}

From now on we assume that $\bk$ satisfies the assumptions of~\S\ref{ss:rings}.
The goal of this section is to endow the biequivariant category $\BE(\h,W)$ with a bounded 
t-structure and investigate its heart. 

\subsection{t-structure for categories associated with singletons}
\label{ss:t-structure-singleton}

We start by considering singleton sets, in analogy with \cite[Lemmas~3.1 and Lemma~3.18]{modrap2}. 
Recall the equivalence $\gamma$ of Lemma~\ref{le:Dw is equivalent to free R}. Passing to bounded homotopy categories, we obtain an equivalence
\[
 \BE_{\{w\}}(\h,W) \cong \Kb \Free^{\fgen,\Z}(R).
\]
Composing with the equivalence of Lemma~\ref{lem:DbMod-KbFree} we deduce an equivalence of triangulated categories
\begin{equation}
\label{eqn:equiv-BEw}
\BE_{\{w\}}(\h,W) \cong \Db \Mod^{\fgen,\Z}(R).
\end{equation}
Here, the autoequivalence $(1)$ of $\BE_{\{w\}}(\h,W)$ corresponds to the autoequivalence of $\Db \Mod^{\fgen,\Z}(R)$ sending a complex $(M^n, d^n)_{n \in \mathbb{Z}}$ to the complex $(M^n(1),-d^n)_{n \in \mathbb{Z}}$ where $(1)$ is as in Lemma~\ref{le:Dw is equivalent to free R}. This autoequivalence will also be denoted $(1)$.

Now, let us recall the ``linear Koszul duality'' construction of~\cite[Section~4]{prinblock} (see also~\cite{mr} for a slightly different and more general construction).
Let $\mathsf{\Lambda}$ be the differential graded algebra defined as the exterior algebra of the free $\bk$-module $V$ placed in degree $-1$, with trivial differential. We will consider $\mathsf{\Lambda}$ as a $\mathbb{Z}$-graded dg-algebra (sometimes called a differential graded graded algebra or a $\mathbb{G}_{\mathbf{m}}$-equivariant dg-algebra), where $V$ is in degree $-2$ for this new grading. Then, composing the ``Koszul duality'' equivalence of~\cite[Theorem~4.1]{prinblock} with the ``regrading'' equivalence denoted $\xi$ in~\cite[\S 4.2]{prinblock} we obtain an equivalence of triangulated categories
\[
 \Db \Mod^{\fgen,\Z}(R) \xrightarrow{\sim} D^{\fgen}_{\Z}(\mathsf{\Lambda}),
\]
where the right-hand side is the derived category of $\Z$-graded $\mathsf{\Lambda}$-dg-modules whose cohomology is finitely generated over $\bk$.\footnote{In~\cite{prinblock}, for simplicity this claim is stated only in the case $\bk$ is a field. But the same arguments apply in the present generality; see e.g.~\cite{mr} for similar constructions.} 
Composing with~\eqref{eqn:equiv-BEw} we deduce an equivalence
\begin{equation}
\label{eqn:equiv-t-structure}
\BE_{\{w\}}(\h,W) \xrightarrow{\sim} D^{\fgen}_{\Z}(\mathsf{\Lambda}).
\end{equation}

Since $\mathsf{\Lambda}$ is concentrated in non-positive cohomological degrees, the right-hand side has a canonical t-structure defined by
\begin{align*}
 \bigl( D^{\fgen}_{\Z}(\mathsf{\Lambda}) \bigr)^{\leq 0} &= \{M \in D^{\fgen}_{\Z}(\mathsf{\Lambda}) \mid \mathsf{H}^{>0}(M)=0\}, \\
 \bigl( D^{\fgen}_{\Z}(\mathsf{\Lambda}) \bigr)^{\geq 0} &= \{M \in D^{\fgen}_{\Z}(\mathsf{\Lambda}) \mid \mathsf{H}^{<0}(M)=0\}.
\end{align*}
The \emph{perverse t-structure} on $\BE_{\{w\}}(\h,W)$, denoted
\[
 \bigl( {}^p \hspace{-1pt} \BE_{\{w\}}(\h,W)^{\leq 0}, {}^p \hspace{-1pt} \BE_{\{w\}}(\h,W)^{\geq 0} \bigr),
\]
is defined as the transport of the t-structure on $D^{\fgen}_{\Z}(\mathsf{\Lambda})$ considered above along the equivalence~\eqref{eqn:equiv-t-structure}. It can be checked from the definitions that, under this equivalence,
the autoequivalence $\langle 1 \rangle$ of $\BE_{\{w\}}(\h,W)$ corresponds to the autoequivalence of the category $D^{\fgen}_{\Z}(\mathsf{\Lambda})$ sending a $\Z$-graded dg-module $M$ to the same dg-module, with degree-$j$ part (for the ``extra'' $\Z$-grading) the degree-$(j-1)$ part of $M$. The latter equivalence is clearly t-exact; hence so is the autoequivalence $\langle 1 \rangle$ on $\BE_{\{w\}}(\h,W)$.

It is clear also that the object $b_w$ considered in~\S\ref{ss:def-D-N} belongs to the heart of the perverse t-structure on $\BE_{\{w\}}(\h,W)$. In fact, this object characterizes the 
t-structure in the following sense.

\begin{lema}
\label{lem:subcat-generated-1}
The subcategory ${}^p \hspace{-1pt} \BE_{\{w\}}(\h,W)^{\leq 0}$ is generated under extensions by the objects $b_w \langle m \rangle[n]$ with $m \in \Z$ and $n \in \Z_{\geq 0}$.
\end{lema}

\begin{proof}
The equivalence~\eqref{eqn:equiv-t-structure} sends $b_w$ to $\bk$ (the trivial $\mathsf{\Lambda}$-dg-module, concentrated in degree $0$). Hence the statement amounts to the claim that $\bigl( D^{\fgen}_{\Z}(\mathsf{\Lambda}) \bigr)^{\leq 0}$ is generated by the dg-modules which are concentrated in one cohomological degree $n \leq 0$, and free as a graded $\bk$-module. However, using truncation functors we see that any object of $\bigl( D^{\fgen}_{\Z}(\mathsf{\Lambda}) \bigr)^{\leq 0}$ is an extension of dg-modules concentrated in one cohomological degree $n$ (where $n$ varies in $\Z_{\leq 0}$). Choosing finite free resolutions of these $\bk$-modules (which exist under our assumptions) we obtain the desired claim.
\end{proof}

\subsection{Definition of the t-structure}
\label{ss:def-t-structure}

We are now ready to introduce our main definition, following~\cite[Definition~3.18]{modrap2}. 

\begin{definition}
\label{def:t structure}
Let $I$ be a finite locally closed subset of $W$. 
The {\it perverse $t$-structure} on $\BE_{I}(\h,W)$ is the bounded $t$-structure given by
\begin{align*}
{}^p \hspace{-1pt} \BE_{I}(\h,W)^{\leq0}=&\bigl\{\rF\in\BE_{I}(\h,W)\mid \forall w\in I, \ (i_w^I)^*(\rF)\in{}^p \hspace{-1pt} \BE_{w}(\h,W)^{\leq0} \bigr\},\\ 
{}^p \hspace{-1pt} \BE_{I}(\h,W)^{\geq0}=&\bigl\{\rF\in\BE_{I}(\h,W)\mid \forall w\in I, \ (i_w^I)^!(\rF)\in{}^p \hspace{-1pt} \BE_{w}(\h,W)^{\geq0}\}.
\end{align*}
\end{definition}

Here, the fact that this pair of subcategories indeed forms a bounded t-structure follows from the general theory of recollement (see~\cite[Th\'eor\`eme~1.4.10]{bbd}) together with
Lemma~\ref{le:recollement compositions for locally closed inclusions}.

\begin{lema}
\label{lem:basic-functors-exact}
 The following functors are t-exact:
 \begin{enumerate}
  \item $\langle 1 \rangle$;
  \item $(i_J^I)_*$ for $I \subset W$ a finite locally closed subset and $J \subset I$ a closed subset;
  \item
  $(i_K^I)^*$ for $I \subset W$ a finite locally closed subset and $K \subset I$ an open subset.
 \end{enumerate}
\end{lema}

\begin{proof}
 The case of $\langle 1 \rangle$ is an immediate consequence of the special case when $I$ is a singleton, which was justified in~\S\ref{ss:t-structure-singleton}, and the case of $(i_K^I)^*$ follows from Lemma~\ref{le:recollement compositions for locally closed inclusions} and~\eqref{eqn:pullback-open}. To justify the exactness of $(i_J^I)_*$, it suffices to prove that for $w \in I$ we have
 \[
  (i_w^I)^* (i_J^I)_* \cong \begin{cases}
                             (i_w^J)^* & \text{if $w \in J$;}\\
                             0 & \text{otherwise}
                            \end{cases}
\quad \text{and} \quad   (i_w^I)^! (i_J^I)_* \cong \begin{cases}
                             (i_w^J)^! & \text{if $w \in J$;}\\
                             0 & \text{otherwise.}
                            \end{cases}
 \]
Here the isomorphisms on the right-hand side follow from those on the left-hand side by duality. For the left-hand side, if $w \in J$ then
 \[
  (i_w^I)^* (i_J^I)_* \cong (i_w^J)^* (i_J^I)^* (i_J^I)_* \cong (i_w^J)^*
 \]
by Lemma~\ref{le:recollement compositions for locally closed inclusions} and the invertibility of the first morphism in~\eqref{eqn:pullback-pushforward-inverse}. If $w \notin J$ then the claim follows from Lemma~\ref{lem:pushforward-pullback-singleton}.
\end{proof}

Using Lemma~\ref{lem:basic-functors-exact}, the definition of the perverse t-structure can be generalized to any locally closed subset $I \subset W$ as follows. By definition, $\BE_{I}(\h,W)$ is the direct limit of the categories $\BE_{J}(\h,W)$ for $J \subset I$ a finite closed subset (for the embeddings $(i_J^{J'})_* : \BE_{J}(\h,W) \to \BE_{J'}(\h,W)$ for $J \subset J' \subset I$ closed subsets). Since under these embeddings
we have
\begin{align*}
 {}^p \hspace{-1pt} \BE_{J}(\h,W)^{\leq 0} &= \BE_J(\h,W) \cap {}^p \hspace{-1pt} \BE_{J'}(\h,W)^{\leq 0}, \\
 {}^p \hspace{-1pt} \BE_{J}(\h,W)^{\geq 0} &= \BE_J(\h,W) \cap {}^p \hspace{-1pt} \BE_{J'}(\h,W)^{\geq 0}
\end{align*}
(see in particular Lemma~\ref{lem:basic-functors-exact}), we can define ${}^p \hspace{-1pt} \BE_{I}(\h,W)^{\leq0}$ and ${}^p \hspace{-1pt} \BE_{I}(\h,W)^{\geq0}$ as the direct limits of the categories ${}^p \hspace{-1pt} \BE_{J}(\h,W)^{\leq0}$ and ${}^p \hspace{-1pt} \BE_{J}(\h,W)^{\geq0}$ respectively, for $J$ running over finite closed subsets of $I$. It is clear 
Lemma~\ref{lem:basic-functors-exact} then holds also without the assumption that $I$ is finite.

Let us immediately note the following consequence of the existence of the perverse t-structure, in view of the main result of~\cite{lc}.

\begin{cor}
\label{cor:Karoubian}
For any locally closed subset $I \subset W$,
the category $\BE_I(\h,W)$ is Karoubian.
\end{cor}

The subcategory ${}^p \hspace{-1pt} \BE_{I}(\h,W)^{\leq0}$ can be described in more concrete terms as follows.

\begin{lema}
\label{lem:t-structure-D-N}
Let $I \subset W$ be a locally closed subset.
\begin{enumerate}
\item
\label{it:D-}
The subcategory
 ${}^p \hspace{-1pt} \BE_{I}(\h,W)^{\leq0}$
 is generated under extensions by the objects $\Delta^I_w \langle m \rangle [n]$ with $w \in W$, $m \in \Z$ and $n \in \Z_{\geq 0}$.
\item
\label{it:D+}
The subcategory 
 ${}^p \hspace{-1pt} \BE_{I}(\h,W)^{\geq0}$
 contains the objects $\nabla^I_w \langle m \rangle [n]$ with $w \in W$, $m \in \Z$ and $n \in \Z_{\leq 0}$.
\end{enumerate}
\end{lema}

\begin{proof}
We can (and shall) assume that $I$ is finite. Observe first that~\eqref{eqn:restriction-pushforward} and~\eqref{eqn:restriction-pushforward-2} imply that $\Delta^I_w$ belongs to $\BE_{I}(\h,W)^{\leq0}$, and that $\nabla^I_w$ belongs to $\BE_{I}(\h,W)^{\geq0}$. 
 
 Since $\langle 1 \rangle$ is a t-exact equivalence (see Lemma~\ref{lem:basic-functors-exact}), we deduce the containments
 \begin{align*}
  \langle \Delta^I_w \langle m \rangle [n] : w \in W, \, m \in \Z, \, n \in \Z_{\geq 0} \rangle_{\mathrm{ext}} \subset {}^p \hspace{-1pt} \BE_{I}(\h,W)^{\leq0} \\
  \langle \nabla^I_w \langle m \rangle [n] : w \in W, \, m \in \Z, \, n \in \Z_{\leq 0} \rangle_{\mathrm{ext}} \subset {}^p \hspace{-1pt} \BE_{I}(\h,W)^{\geq0}
 \end{align*}
 (where the left-hand side denotes the subcategory generated under extensions by the objects indicated).
 
We prove the reverse containment by induction on $|I|$.
If $|I|=1$, then the desired claim was proved in Lemma~\ref{lem:subcat-generated-1}. Then for a general $I$, choose $w \in I$ maximal (so that $\{w\}$ is open), and for $\mathscr{F}$ in ${}^p \hspace{-1pt} \BE_{I}(\h,W)^{\leq0}$ consider the distinguished triangle
\[
 (i_w^I)_! (i_w^I)^* \mathscr{F} \to \mathscr{F} \to (i_{I \smallsetminus \{w\}}^I)_! (i_{I \smallsetminus \{w\}}^I)^* \mathscr{F} \xrightarrow{[1]}.
\]
From the definitions we see that $(i_w^I)^* \mathscr{F}$ belongs to ${}^p \hspace{-1pt} \BE_{\{w\}}(\h,W)^{\leq0}$ and that $(i_{I \smallsetminus \{w\}}^I)^*$ belongs to ${}^p \hspace{-1pt} \BE_{I \smallsetminus \{w\}}(\h,W)^{\leq0}$. Using induction and Lemma~\ref{lem:pf-pb-D-N},
we deduce that $\mathscr{F}$ belongs to $\langle \Delta^I_w \langle m \rangle [n] : w \in W, \, m \in \Z, \, n \in \Z_{\geq 0} \rangle_{\mathrm{ext}}$, as desired.
\end{proof}

\begin{rmk}
When $\bk$ is a field, Lemma~\ref{lem:t-structure-D-N} can be made symmetric: in that case, ${}^p \hspace{-1pt} \BE_{I}(\h,W)^{\geq0}$ is generated under extensions by the objects $\nabla^I_w \langle m \rangle [n]$ with $w \in W$, $m \in \Z$ and $n \in \Z_{\leq 0}$ (and as a consequence, the functor $\D_I$ is t-exact). But this statement is \emph{not} true for general coefficients. (Indeed, it can fail already when $I$ is a singleton.)
\end{rmk}

\subsection{Standard and costandard objects are perverse}

The heart of the t-structure on $\BE_{I}(\h,W)$ constructed in~\S\ref{ss:def-t-structure} will be denoted 
\begin{align*}
\fP_I^{\BE}(\h,W).
\end{align*}
(When $I=W$, the subscript will sometimes be omitted.)
The objects which belong to this heart will be called {\it perverse}.

Our next goal is to show that standard and costandard objects 
are perverse.

\begin{lema}
\label{lem:D-N-exactness}
If $w\in W$, then
\begin{enumerate}
 \item 
 \label{it:N-exact}
 the functors 
 \[
  (-)\ustar\nabla_w,\,\nabla_w\ustar(-):\BE(\h,W)\longrightarrow\BE(\h,W)
 \]
are right t-exact with respect to the perverse t-structure;
\item
\label{it:D-exact}
the functors 
 \[
  (-)\ustar\Delta_w,\,\Delta_w\ustar(-):\BE(\h,W)\longrightarrow\BE(\h,W)
 \]
are left t-exact with respect to the perverse t-structure.
\end{enumerate}
\end{lema}

\begin{proof}
\eqref{it:N-exact} We prove the right exactness of $\nabla_w \ustar (-)$; the other functor can be treated similarly. In view of Proposition~\ref{prop:D-N-convolution} we can assume that $\ell(\uw)=1$, i.e.~that $\uw=(s)$ for some $s \in S$. Then, Lemma~\ref{lem:t-structure-D-N} shows that to conclude it suffices to prove that for any $w \in W$ we have
\[
\nabla_s \ustar \Delta_w \in {}^p \hspace{-1pt} \BE(\h,W)^{\leq 0}.
\]
If $sw<w$ then $\nabla_s \ustar \Delta_w \cong \Delta_{sw}$ by Proposition~\ref{prop:D-N-convolution}, so the claim is clear in this case. If $sw>w$, then we use the triangles of Lemma~\ref{lem:triangle-Dw-ws} (for $w=e$) to deduce distinguished triangles
\[
B_s \ustar \Delta_w \to \nabla_s \ustar \Delta_w \to \Delta_w \langle 1 \rangle \xrightarrow{[1]}, \qquad \Delta_w \langle -1 \rangle \to \Delta_s \ustar \Delta_w \to B_s \ustar \Delta_w \xrightarrow{[1]}.
\]
In the second triangle, the second term is isomorphic to $\Delta_{sw}$ by Proposition~\ref{prop:D-N-convolution}, so that the third term belongs to ${}^p \hspace{-1pt} \BE(\h,W)^{\leq 0}$. Once this information is known, the first triangle shows that $\nabla_s \ustar \Delta_w$ belongs to ${}^p \hspace{-1pt} \BE(\h,W)^{\leq 0}$, as desired.

\eqref{it:D-exact} 
The left exactness of our functors follows from the right-exact of their inverses (proved in~\eqref{it:N-exact}) in view of~\cite[Corollary~10.1.18]{ks}.
\end{proof}

\begin{prop}
\label{prop:D-N-perverse}
If $w,y\in W$, then $\Delta_{w}\ustar\nabla_{y}$ and $\nabla_{y}\ustar\Delta_{w}$ are perverse.
In particular, $\Delta_w$ and $\nabla_w$ belong to $\fP^{\BE}(\h,W)$.
\end{prop}

\begin{proof}
Lemma~\ref{lem:t-structure-D-N}\eqref{it:D-} and Lemma~\ref{lem:D-N-exactness}\eqref{it:N-exact} imply that $\Delta_{w}\ustar\nabla_{y}$ belongs to the subcategory ${}^p \hspace{-1pt} \BE(\h,W)^{\leq 0}$, and Lemma~\ref{lem:t-structure-D-N}\eqref{it:D+} and Lemma~\ref{lem:D-N-exactness}\eqref{it:D-exact} imply that $\Delta_{w}\ustar\nabla_{y}$ belongs to ${}^p \hspace{-1pt} \BE(\h,W)^{\geq 0}$. Hence this object is perverse. Similar considerations show that $\nabla_{y}\ustar\Delta_{w}$ is perverse. The final claims are obtained by setting $y=e$.
\end{proof}

Once Proposition~\ref{prop:D-N-perverse} is established, its final claim can be extended to the categories $\BE_I(\h,W)$, as follows.

\begin{cor}
\label{cor:D-N-perverse-I}
For any locally closed subset $I \subset W$ and any $w \in I$, the objects $\Delta^I_w$ and $\nabla^I_w$ are perverse.
\end{cor}

\begin{proof}
We can (and shall) assume that $I$ is finite. Choose some finite closed subset $J \subset W$ containing $I$ and in which $I$ is open. Then since the functor $(i_J^W)_*$ is t-exact (see Lemma~\ref{lem:basic-functors-exact}) and does not kill any object (since it is fully faithful), we see that for any $w \in W$ the object $\Delta_w^J$ belongs to $\fP^{\BE}_J(\h,W)$. Then, since the functor $(i_I^J)^*$ is t-exact (see again Lemma~\ref{lem:basic-functors-exact}), we obtain the desired claim.
\end{proof}

\section{The case of field coefficients}\label{sec:field}

In this section we assume that $\bk$ is a field. 

\subsection{Simple perverse objects}
\label{ss:simple-perverse}

In the present setting where $\bk$ is a field, the recollement formalism provides a description of the simple objects in $\fP^{\BE}_I(\h,W)$; see in particular~\cite[Proposition~1.4.26]{bbd}. More precisely, for any $w \in I$, by Lemma~\ref{lem:Hom-BE-D-N} there exists (up to scalar) a unique nonzero morphism $\Delta^I_w \to \nabla^I_w$. If we denote the image of this morphism (in the abelian category $\fP^\BE_I(\h,W)$) by $\rL^I_w$, then $\rL^I_w$ is simple, and the isomorphism classes of simple objects in $\fP_I^{\BE}(\h,W)$ are in bijection with $I \times \Z$ via the map $(w,n) \mapsto \rL^I_w \langle n \rangle$. Moreover, the same proof as in~\cite[Th\'eor\`eme~4.3.1(i)]{bbd} shows that $\fP^\BE_I(\h,W)$ is a finite length category. With this in hand, for any closed subset $J \subset I$ one can identify $\fP^{\BE}_J(\h,W)$ as the Serre subcategory of $\fP^{\BE}_I(\h,W)$ generated by the simple objects $\rL^I_w \langle n \rangle$ with $n \in \Z$ and $w \in J$; in this setting we will sometimes consider $\rL_w^I$ as an object of $\fP^\BE_J(\h,W)$. As usual, when $I=W$ we sometimes omit it from the notation.

By the general recollement formalism, the object $\rL^I_w$ is characterized by the conditions that it belongs to $\BE_{\{\leq w\} \cap I}(\h,W)$, that
\begin{equation}
\label{eqn:conditions-Lw0}
(i_w^{\{\leq w\} \cap I})^* \rL^I_w \cong b_w,
\end{equation}
and that for any $y \in I$ such that $y<w$ we have
\begin{equation}
\label{eqn:conditions-Lw}
(i_y^{\{\leq w\} \cap I})^* \rL^I_w \in {}^p \hspace{-1pt} \BE_{\{y\}}(\h,W)^{\leq -1}, \quad (i_y^{\{\leq w\} \cap I})^! \rL^I_w \in {}^p \hspace{-1pt} \BE_{\{y\}}(\h,W)^{\geq 1};
\end{equation}
see~\cite[Corollaire~1.4.24]{bbd}. From this characterization, we deduce in particular that if $J \subset I$ is an open subset containing $w$, then we have
\begin{equation}
\label{eqn:restriction-open-simple}
 (i_J^I)^* \rL_w^I \cong \rL_w^J.
\end{equation}

\begin{example}
\label{ex:Bs}
 When $w=s$, it is easy to check that $B_s$ satisfies conditions~\eqref{eqn:conditions-Lw0} and~\eqref{eqn:conditions-Lw}. Therefore, $\rL_s \cong B_s$.
\end{example}

\subsection{More properties of standard and costandard objects}

It is easy to see that $\rL_w$ is the head of $\Delta_w$ and the socle of $\nabla_w$. Let us record the following fact about the other possible composition factors of these objects.

\begin{lema}
\label{lem:ker-standard-object-simple}
If $w \in W$,
all the composition factors of the kernel of the surjection $\Delta_w\twoheadrightarrow \rL_w$ and the cokernel of the embedding $\rL_w\hookrightarrow\nabla_w$ are of the form $\rL_v \langle n \rangle$ with $n \in \Z$ and $v \in W$ which satisfies $v<w$.
\end{lema}

\begin{proof}
By definition, $\Delta_w$ and $\nabla_w$ belong to $\fP^{\BE}_{\{\leq w\}}(\h,W)$. Moreover, the image of the canonical morphism $\Delta_w \to \nabla_w$ under $(i_w^{\{\leq w\}})^*$ is the identity map of $b_w$. Hence the kernel of the surjection $\Delta_w \twoheadrightarrow \rL_w$ and the cokernel of the embedding of $\rL_w \hookrightarrow \nabla_w$ are annihilated by $(i_w^{\{\leq w\}})^*$, so they belong to $\BE_{\{< w\}}(\h,W)$. Since they are perverse, they in fact belong to $\fP^{\BE}_{\{< w\}}(\h,W)$, which finishes the proof.
\end{proof}

We will now prove the following claim, which is an analogue of a well-known result in usual category $\mathscr{O}$ (see~\cite[\S\S 4.1--4.2]{humphreys} for an algebraic proof and~\cite[\S 2.1]{bbm} for a geometric approach).

\begin{prop}
\label{prop:the socle of the standard objects}
Let $w\in W$. Then
\begin{enumerate}
\item
\label{it:socle-D}
The socle of $\Delta_w$ is isomorphic to $\rL_e\langle-\ell(w)\rangle$, and the cokernel of the inclusion $\rL_e\langle-\ell(w)\rangle\hookrightarrow\Delta_w$ has no composition factor 
of the form $\rL_e\langle n\rangle$.
\item 
\label{it:top-N}
The head of $\nabla_w$ is isomorphic to $\rL_e\langle\ell(w)\rangle$, and the kernel of the surjection $\nabla_w\twoheadrightarrow \rL_e\langle\ell(w)\rangle$ has no composition 
factor
of the form $\rL_e\langle n\rangle$.
\end{enumerate}
\end{prop}

The proof of Proposition~\ref{prop:the socle of the standard objects} will exploit two lemma.

\begin{lema}
\label{lem:ses-D-w-ws}
Let $w \in W$ and $s \in S$ be such that $ws>w$. Then $\Delta_w\ustar B_{s}$ is perverse, and there exists a short exact sequence
\[
\Delta_{w}\langle-1\rangle\hookrightarrow\Delta_{ws}\twoheadrightarrow\Delta_w\ustar B_{s}
\]
in $\fP^{\BE}(\h,W)$.
\end{lema}

\begin{proof}
Recall the first distinguished triangle in Lemma \ref{lem:triangle-Dw-ws}. By Proposition~\ref{prop:D-N-perverse} the first two terms in this triangle belong to $\fP^{\BE}(\h,W)$, so the third term must lie in ${}^p \hspace{-1pt} \BE(\h,W)^{\leq0}$.  
On the other hand, by Example~\ref{ex:Bs}, $B_s$ belongs to $\fP^{\BE}(\h,W)$, so Lemma~\ref{lem:D-N-exactness}\eqref{it:D-exact} tells us that $\Delta_w\ustar B_{s}$ belongs to ${}^p \hspace{-1pt} \BE(\h,W)^{\geq0}$. We conclude that $\Delta_w\ustar B_{s}$ in fact belongs to $\fP^{\BE}(\h,W)$, and that the triangle under consideration is a short exact sequence in $\fP^{\BE}(\h,W)$.
\end{proof}

The following lemma is more subtle; its proof will be given in~\S\ref{ss:perverse-Ws} below.

\begin{lema}
\label{lem:DwBs-composition-factors}
 Let $w \in W$ and $s \in S$ be such that $ws>w$. Then all the composition factors of $\Delta_w \ustar B_s$ are of the form $\rL_y \langle n \rangle$ with $ys<y$.
\end{lema}

\begin{proof}[Proof of Proposition~{\rm \ref{prop:the socle of the standard objects}}]
We will prove~\eqref{it:socle-D} by induction on $\ell(w)$; then~\eqref{it:top-N} follows by duality. 

If $w=e$, we have $\Delta_e\cong \rL_e$; see Example~\ref{ex:Des}\eqref{it:D-N-e}. Thus, there is nothing to prove in this case.
Now, let $w \in W \smallsetminus \{e\}$, and assume the claim is known for elements $y \in W$ with $\ell(y)<\ell(w)$. Choose $s \in S$ such that $ws<w$, and consider the exact sequence
\[
\Delta_{ws}\langle-1\rangle\hookrightarrow\Delta_{w}\twoheadrightarrow\Delta_{ws}\ustar B_{s}
\]
provided by Lemma~\ref{lem:ses-D-w-ws}. By induction we know that there exists an embedding $\rL_e \langle -\ell(w) \rangle \hookrightarrow \Delta_{ws}\langle-1\rangle$ whose cokernel has no composition factor of the form $\rL_e \langle n \rangle$. On the other hand, Lemma~\ref{lem:DwBs-composition-factors} ensures that $\Delta_{ws}\ustar B_{s}$ has no composition factor of this form either. Hence we obtain an embedding $\rL_e \langle -\ell(w) \rangle \hookrightarrow \Delta_w$ whose cokernel has no composition factor of the form $\rL_e \langle n \rangle$. To finish the proof, it suffices to show that $\Delta_{w}$ has no subobject of the form $\rL_x\langle n\rangle$ with $x\neq e$.

Assume for a contradiction that there exists an injective morphism $\rL_x\langle n\rangle \hookrightarrow \Delta_w$ with $x \neq e$. Using induction we see that this morphism does not factor through $\Delta_{ws} \langle -1 \rangle$; hence its composition with the surjection $\Delta_{w}\twoheadrightarrow\Delta_{ws}\ustar B_{s}$ is nonzero. In view of Lemma~\ref{lem:DwBs-composition-factors}, this implies that $xs<x$, and hence that $\Delta_x \cong \Delta_{xs} \ustar \Delta_s$ (see Proposition~\ref{prop:D-N-convolution}). Proposition~\ref{prop:D-N-convolution} also shows that the functor
\begin{equation}
\label{eqn:functor-Ds}
(-) \ustar \Delta_s : \BE(\h,W) \to \BE(\h,W)
\end{equation}
is invertible; in particular it induces an isomorphism
\[
\Hom_{\fP^{\BE}(\h,W)}(\Delta_{xs}\langle n \rangle, \Delta_{ws}) \xrightarrow{\sim} \Hom_{\fP^{\BE}(\h,W)}(\Delta_{x}\langle n \rangle, \Delta_{w}).
\]
By induction we know that any nonzero subobject of $\Delta_{ws}$ must admit $\rL_e \langle -\ell(ws) \rangle$ as a composition factor. Applying the induction hypothesis 
to $\Delta_{xs}$ also, we deduce that any nonzero morphism $\Delta_{xs} \langle n \rangle \to \Delta_{ws}$ must be injective. Since~\eqref{eqn:functor-Ds} is left t-exact (see 
Lemma~\ref{lem:D-N-exactness}\eqref{it:D-exact}), we finally obtain that any nonzero morphism $\Delta_{x} \langle n \rangle \to \Delta_{w}$ is injective. However, since (by 
assumption) there exists an embedding $\rL_x\langle n\rangle \hookrightarrow \Delta_w$, we can construct a nonzero and noninjective morphism of this form as the composition
\[
\Delta_x \langle n \rangle \twoheadrightarrow \rL_x\langle n\rangle \hookrightarrow \Delta_w
\]
where the first morphism is the natural one. This provides the desired contradiction.
\end{proof}

For completeness we record the following consequence of Proposition~\ref{prop:the socle of the standard objects}.

\begin{prop}
\label{prop:morphisms between standard objects}
Let $w,y\in W$. Then
\begin{align*}
\dim\Hom_{\BE(\h,W)}(\Delta_w,\Delta_y\langle n\rangle)=\begin{cases}
                                                          1 & \mbox{if $w\leq y$ and $n=\ell(y)-\ell(w)$,} \\
                                                          0 & \mbox{otherwise.} 
                                                          \end{cases}
\end{align*}
Moreover, if $w \leq y$, any nonzero morphism $\Delta_w\rightarrow\Delta_y\langle\ell(y)-\ell(w)\rangle$ is injective.
\end{prop}

\begin{proof}
If $w\not\leq y$, then the $\Hom$-space under consideration vanishes by adjunction and Lemma~\ref{lem:pushforward-pullback-singleton}.

Assume now that $w\leq y$ and set $m=\ell(y)-\ell(w)$. If $f : \Delta_w \to \Delta_y \langle n \rangle$ is a nonzero morphism, its image must admit $\rL_e \langle n-\ell(y) \rangle$ as a composition factor; therefore its kernel cannot contain the socle of $\Delta_w$. This means that the kernel is trivial, and $f$ is injective. Moreover we must have $n-\ell(y) = -\ell(w)$, i.e.~$n=m$.

To conclude, it remains to show that $\dim\Hom_{\BE(\h,W)}(\Delta_w,\Delta_y\langle m\rangle)=1$ (where, as above, we assume that $w\leq y$ and set $m=\ell(y)-\ell(w)$). We proceed by induction on $\ell(y)$, the case $\ell(y)=0$ being obvious. Assume that $\ell(y)>0$, and choose $s \in S$ such that $ys<y$. If $ws<w$, then as in the proof of Proposition~\ref{prop:the socle of the standard objects} we have
\[
\Hom_{\fP^\BE(\h,W)}(\Delta_w, \Delta_y \langle m \rangle) \cong \Hom_{\fP^\BE(\h,W)}(\Delta_{ws}, \Delta_{ys} \langle m \rangle),
\]
and the result follows from the induction hypothesis. If now $ws>w$, then the exact sequence of Lemma~\ref{lem:ses-D-w-ws} (applied to $ys$) induces an exact sequence of $\bk$-vector spaces
\begin{multline*}
0\to\Hom_{\fP^{\BE}(\h,W)}(\Delta_{w},\Delta_{ys}\langle m-1\rangle)\to\Hom_{\fP^{\BE}(\h,W)}(\Delta_{w},\Delta_{y}\langle m\rangle)\\
\to\Hom_{\fP^{\BE}(\h,W)}(\Delta_{w},\Delta_{ys}\ustar B_{s}\langle m\rangle).
\end{multline*}
Here the last space must vanish, because $\Delta_{ys} \ustar B_s \langle m \rangle$ does not admit $\rL_w$ as a composition factor (see Lemma~\ref{lem:DwBs-composition-factors}).
We deduce that
\[
\Hom_{\fP^{\BE}(\h,W)}(\Delta_{w},\Delta_{ys}\langle m-1\rangle)\cong\Hom_{\fP^{\BE}(\h,W)}(\Delta_{w},\Delta_{y}\langle m\rangle),
\]
and again the desired result follows from the induction hypothesis.
\end{proof}

By duality we have analogous properties for costandard objects.

\begin{prop}
\label{prop:morphisms between costandard objects}
Let $w,y\in W$. Then
\begin{align*}
\dim\Hom_{\BE(\h,W)}(\nabla_y\langle n\rangle,\nabla_w)=\begin{cases}
                                                          1 & \mbox{if $w\leq y$ and $n=\ell(w)-\ell(y)$,} \\
                                                          0 & \mbox{otherwise.} 
                                                          \end{cases}
\end{align*}
Moreover, if $w \le y$, any nonzero morphism $\nabla_y\langle\ell(w)-\ell(y)\rangle\rightarrow\nabla_w$ is surjective.\qed
\end{prop}

\subsection{A category attached to a simple reflection.}
\label{ss:Bs-preliminaries}

The goal of \S\S\ref{ss:Bs-preliminaries}--\ref{ss:perverse-Ws} is to prove Lemma~\ref{lem:DwBs-composition-factors}. These results will not be used in the rest of the paper. Most of our constructions could be performed for general coefficients; but for simplicity we continue to assume that $\bk$ is a field.

We fix $s \in S$ and denote by $\BE(\h,W|s)$ the full triangulated subcategory of $\BE(\h,W)$ generated by the image of the functor $(-) \ustar B_s$. Our first objective is to endow this category with the same kind of structure (local versions, recollement and perverse t-structure) as for $\BE(\h,W)$.

Let $W^s = \{ w \in W \mid ws < w \}$.  A locally closed subset $I \subset W$ is said to be \emph{right $s$-stable} if $w \in I$ implies $ws \in I$.

Recall from Corollary~\ref{cor:multiplication-Bs} that if $I$ is closed and right $s$-stable, then the full subcategory $\BE_I(\h,W)$ of $\BE(\h,W)$ is stable under the functor $(-) \ustar B_s$. If $I \subset W$ is now \emph{locally} closed and finite, one can write $I=I_0 \smallsetminus I_1$ with $I_1 \subset I_0 \subset W$ finite, closed, and right $s$-stable. By Remark~\ref{rmk:recollement-quotient}, the category $\BE_I(\h,W)$ identifies with the Verdier quotient $\BE_{I_0}(\h,W) / \BE_{I_1}(\h,W)$. Then the functor $(-) \ustar B_s : \BE_{I_0}(\h,W) \to \BE_{I_0}(\h,W)$ induces an endofunctor of $\BE_I(\h,W)$, which will also be denoted $(-) \ustar B_s$. This functor is clearly self-adjoint.

In this setting, we define
\[
 \BE_I(\h,W|s)
\]
to be the full triangulated subcategory of $\BE_I(\h,W)$ generated by the image of the functor $(-) \ustar B_s$.

\begin{lema}
\label{lem:bws}
Let $w \in W^s$, and let $\uw$ and $\uw'$ be two reduced expressions for $w$.  The images of $B_\uw$ and $B_{\uw'}$ in $\BE_{\{ws,w\}}(\h,W)$ are canonically isomorphic.
\end{lema}

\begin{proof}
We will use the calculations from Example~\ref{ex:w-ws}.  Let us rewrite the triangle~\eqref{eq:distinguished triangle in BE w ws} as
\[
b_{ws}\langle -1\rangle \to \Delta^{\{ws,w\}}_w \to B_{\uw} \xrightarrow{[1]}.
\]
There is another version of this triangle in which the third term is replaced by $B_{\uw'}$.  We claim that there exist unique vertical maps $p$ and $q$ making the following diagram commute:
\[
\begin{tikzcd}
b_{ws}\langle -1\rangle \ar[r] \ar[d, "p"'] &
  \Delta^{\{ws,w\}}_w \ar[d, equal] \ar[r] &
  B_{\uw} \ar[d, "q"] \ar[r, "{[1]}"] & {} \\
b_{ws}\langle -1\rangle \ar[r] &
  \Delta^{\{ws,w\}}_w \ar[r] &
  B_{\uw'} \ar[r, "{[1]}"] & {}.
\end{tikzcd}
\]
According to~\cite[Proposition~1.1.9]{bbd}, the existence and uniqueness of $p$ and $q$ would follow if we knew the following two claims:
\begin{align*}
\Hom_{\BE_{\{ws,w\}}(\h,W)}(b_{ws}\langle -1\rangle, B_{\uw'}) &= 0, \\
\Hom_{\BE_{\{ws,w\}}(\h,W)}(b_{ws}\langle -1\rangle, B_{\uw'}[-1]) &= 0.
\end{align*}
The first one is obvious for degree reasons.  The second one is equivalent to the vanishing of $\Hom(b_{ws}, B_{\uw'}(-1))$.  As we observed in Example~\ref{ex:w-ws}, the $R$-module $\Hom^\bullet_{\Diag^\oplus_{\BS,\{ws,w\}}(\h,W)}(b_{ws},B_{\uw'})$ is generated in degree $1$; in particular, it contains no nonzero element of degree $-1$, as desired.

The same reasoning with the roles of $\uw$ and $\uw'$ reversed leads to a similar diagram with vertical maps in the opposite directions.  Using the uniqueness of the various vertical maps, one concludes that $p$ and $q$ are isomorphisms, as desired.
\end{proof}

From now on, for $w \in W^s$, we set
\[
b^s_w = (i^{\{\le w\}}_{\{ws,w\}})^* B_{\uw} \qquad\text{for any reduced expression $\uw$ for $w$.}
\]
(By Lemma~\ref{lem:bws}, this definition is independent of the choice of $\uw$.) Choosing for $\uw$ a reduced expression of the form $\uy s$ (with $\uy$ a reduced expression for $ws$),
light leaves considerations show that the $R$-module $\Hom_{\Diag^\oplus_{\BS,\{ws,w\}}(\h,W)}^\bullet(b_{w}^s,b^s_w)$ is free of rank $2$, and generated by the identity (of degree $0$) and the degree-$2$ morphism
\[
\id_{B_{\uy}} \star
\begin{tikzpicture}[thick,baseline,xscale=0.1,yscale=0.1]
      \draw (0,2) to (0,6);
      \node at (0,2) {$\bullet$};
      \node at (0,6.7) {\tiny $s$};
      \draw (0,-6) to (0,-2);
      \node at (0,-2) {$\bullet$};
      \node at (0,-6.9) {\tiny $s$};
      \end{tikzpicture}.
\]
In the course of the proof of Lemma~\ref{lem:bws}, we saw that there are distinguished triangles
\begin{equation}
\label{eqn:bsw-perverse}
b_{ws}\langle -1\rangle \to \Delta^{\{ws,w\}}_w \to b^s_w \xrightarrow{[1]}, \qquad
b^s_w \to \nabla^{\{ws,w\}}_w \to b_{ws}\langle 1\rangle \xrightarrow{[1]}.
\end{equation}

\begin{lema}
\label{lem:bsw-generate}
For any $w \in W^s$, the triangulated category $\BE_{\{ws,w\}}(\h,W|s)$ is generated by the objects of the form $b^s_w(m)$ with $m \in \Z$.
\end{lema}

\begin{proof}
The category $\BE_{\{ws,w\}}(\h,W)$ is clearly generated by the objects of the form $b_{ws}(m)$ and $b^s_w(m)$.  It follows that $\BE_{\{ws,w\}}(\h,W|s)$ is generated by the objects
\[
b_{ws}(m) \ustar B_s \cong b^s_w(m)
\qquad\text{and}\qquad
b^s_w(m) \ustar B_s \cong b^s_w(m+1) \oplus b^s_w(m-1),
\]
where the latter isomorphism follows from~\eqref{eqn:BsBs}.
\end{proof}

\subsection{Recollement}

We now show that the categories $\BE_I(\h,W|s)$ (with $I$ right $s$-stable) satisfy the same recollement formalism as the categories $\BE_I(\h,W)$.

\begin{prop}\label{prop:recollement for Ws}
Let $I \subset W$ be a finite locally closed right $s$-stable subset, and let $J \subset I$ be a closed right $s$-stable subset.  Then the restriction of the functors from Proposition~{\rm \ref{prop:recollement}} give a recollement diagram
\[
 \xymatrix@C=2cm{
 \BE_J(\h,W|s) \ar[r]|-{(i_J^I)_*} & \BE_I(\h,W|s) \ar[r]|-{(i_{I \smallsetminus J}^I)^*} \ar@/^0.5cm/[l]^-{(i_J^I)^!} \ar@/_0.5cm/[l]_-{(i_J^I)^*} & \BE_{I\smallsetminus J}(\h,W|s). \ar@/^0.5cm/[l]^-{(i_{I \smallsetminus J}^I)_*} \ar@/_0.5cm/[l]_-{(i_{I \smallsetminus J}^I)_!}
 }
\]
\end{prop}

\begin{proof}
We have to show that the six functors from Proposition~\ref{prop:recollement} take the subcategory generated by $(-) \ustar B_s$ to the subcategory generated by $(-) \ustar B_s$.  For $(i^I_J)_*$ and $(i^I_{I \smallsetminus J})^*$, this is obvious.  

Now we consider the functors $(i^I_J)^*$ and $(i^I_{I \smallsetminus J})_!$. Let $\mathscr{G} \in \BE_I(\h,W)$, and set $\mathscr{F} := \mathscr{G} \ustar B_s$. Then we have a
distinguished triangle
\begin{equation}\label{eqn:s-recolle-1}
(i^I_{I \smallsetminus J})_!(i^I_{I \smallsetminus J})^*\mathscr{F} \to \mathscr{F} \to (i^I_J)_*(i^I_J)^*\mathscr{F} \xrightarrow{[1]}.
\end{equation}
On the other hand, we can also
form the distinguished triangle
\begin{equation}\label{eqn:s-recolle-2}
((i^I_{I \smallsetminus J})_!(i^I_{I \smallsetminus J})^*\mathscr{G}) \ustar B_s \to \mathscr{G} \ustar B_s \to ((i^I_J)_*(i^I_J)^*\mathscr{G}) \ustar B_s \xrightarrow{[1]}.
\end{equation}
We claim that the triangles~\eqref{eqn:s-recolle-1} and~\eqref{eqn:s-recolle-2} are canonically isomorphic.  This would follow from~\cite[Proposition~1.1.9]{bbd} if we knew that
\begin{align}
\Hom_{\BE_I(\h,W)}((i^I_{I \smallsetminus J})_!(i^I_{I \smallsetminus J})^*\mathscr{F}, ((i^I_J)_*(i^I_J)^*\mathscr{G}) \ustar B_s[n]) &= 0, \label{eqn:s-rec-van1}\\
\Hom_{\BE_I(\h,W)}(((i^I_{I \smallsetminus J})_!(i^I_{I \smallsetminus J})^*\mathscr{G}) \ustar B_s, (i^I_J)_*(i^I_J)^*\mathscr{F}[n]) &= 0 \label{eqn:s-rec-van2}
\end{align}
for all $n \in \Z$.  (Actually, we only need this for $n \in \{ 0, -1\}$.) 
Now $((i^I_J)_*(i^I_J)^*\mathscr{G}) \ustar B_s$ belongs to $\BE_J(\h,W)$, so~\eqref{eqn:s-rec-van1} holds by adjunction and basic properties of recollement.  For~\eqref{eqn:s-rec-van2}, because $(-) \ustar B_s$ is self-adjoint, we have
\begin{multline*}
\Hom_{\BE_I(\h,W)}(((i^I_{I \smallsetminus J})_!(i^I_{I \smallsetminus J})^*\mathscr{G}) \ustar B_s, (i^I_J)_*(i^I_J)^*\mathscr{F}[n]) \\
\cong
\Hom_{\BE_I(\h,W)}((i^I_{I \smallsetminus J})_!(i^I_{I \smallsetminus J})^*\mathscr{G}, ((i^I_J)_*(i^I_J)^*\mathscr{F}) \ustar B_s[n]).
\end{multline*}
This vanishes by the same reasoning as above.  

This result implies that for any $\mathscr{H}$ in $\BE_{I \smallsetminus J}(\h,W)$ we have
\[
 (i_{I \smallsetminus J}^I)_! (\mathscr{H} \ustar B_s) \cong (i_{I \smallsetminus J}^I)_! (\mathscr{H}) \ustar B_s.
\]
We deduce that $(i_{I \smallsetminus J}^I)_!$ sends $\BE_{I \smallsetminus J}(\h,W|s)$ to $\BE_{I}(\h,W|s)$. Similarly, since the functor $(i_J^I)_*$ is fully faithful and commutes with the functors $(-) \ustar B_s$, we obtain that
\[
 (i_J^I)^*(\mathscr{G} \ustar B_s) \cong (i_J^I)^*(\mathscr{G}) \ustar B_s.
\]
Again, this implies that $(i_J^I)^*$ sends $\BE_I(\h,W|s)$ to $\BE_J(\h,W|s)$.

The analogous claims for $(i^I_J)^!$ and $(i^I_{I \smallsetminus J})_*$ can be proved similarly, or deduced by duality, which finishes the proof.
\end{proof}

Now let $I \subset W$ and $J \subset I$ be finite locally closed right $s$-stable subsets.
In view of Proposition~\ref{prop:recollement for Ws}, we can also define the pushforward and pullback functors $(i_J^I)^*$, $(i_J^I)^!$, $(i_J^I)_*$, $(i_J^I)_!$ for the categories $\BE_J(\h,W|s)$ and $\BE_I(\h,W|s)$ as the restriction of those in~\S\ref{ss:pf-pb-locally-closed}. Then these functors also satisfy the properties of Lemma~\ref{le:recollement compositions for locally closed inclusions}.

\subsection{The perverse t-structure}
\label{ss:perverse-Ws}

We will denote by $\mathsf{C}$ the full subcategory of $\BE_{\{ws,w\}}(\h,W|s)$ whose object are direct sums of objects of the form $b^s_w\langle n\rangle$, $n\in\Z$.

\begin{lema}
\label{lem:bsw-tstruc}
Let $w \in W^s$.  Then if we set
\begin{align*}
{}^p \hspace{-1pt} \BE_{\{ws,w\}}(\h,W|s)^{\leq 0}&:={}^p \hspace{-1pt} \BE_{\{ws,w\}}(\h,W)^{\leq 0}\cap\BE_{\{ws,w\}}(\h,W|s),\\
{}^p \hspace{-1pt} \BE_{\{ws,w\}}(\h,W|s)^{\geq 0}&:={}^p \hspace{-1pt} \BE_{\{ws,w\}}(\h,W)^{\geq 0}\cap\BE_{\{ws,w\}}(\h,W|s),
\end{align*}
the pair $({}^p \hspace{-1pt} \BE_{\{ws,w\}}(\h,W|s)^{\leq 0},{}^p \hspace{-1pt} \BE_{\{ws,w\}}(\h,W|s)^{\geq 0})$
is a t-structure on the category $\BE_{\{ws,w\}}(\h,W|s)$, whose heart is $\mathsf{C}$.
\end{lema}

\begin{proof}
We claim that $\mathsf{C}$ is an {\it admissible abelian} subcategory of $\BE_{\{ws,w\}}(\h,W|s)$ in the sense of~\cite[Definition 1.2.5]{bbd}. It can be checked from the triangles in~\eqref{eqn:bsw-perverse} that $b_w^s$ lies in $\fP^{\BE}_{\{w,ws\}}(\h,W)$ (and thus that $\mathsf{C}$ is a subcategory of $\fP^{\BE}_{\{w,ws\}}(\h,W)$). It follows immediately that
\[
\Hom_{\BE_{\{w,ws\}}(\h,W|s)}(b_w^s, b_w^s \langle n \rangle [m])=0
\]
if $m < 0$.  Hence $\mathsf{C}$ satisfies~\cite[\S 1.2.0]{bbd}. On the other hand, to check that any morphism in $\mathsf{C}$ is admissible we have to check that
\[
[\mathsf{C}] * [\mathsf{C}[1]] \subset [\mathsf{C}[1]] * [\mathsf{C}],
\]
as explained in~\cite[Exemple 1.3.11(ii)]{bbd}. However, the objects whose class belongs to $[\mathsf{C}] * [\mathsf{C}[1]]$ are exactly the cones of morphisms in $\mathsf{C}$. From the remarks in~\S\ref{ss:Bs-preliminaries} we see that such a morphism is a direct sum of morphisms of the form
\[
b_w^s \to 0, \quad 0 \to b_w^s \quad \text{or} \quad b_w^s \xrightarrow{\id} b_w^s.
\]
It is easily checked that the class of the cone of such morphisms belongs to $[\mathsf{C}[1]] * [\mathsf{C}]$, and the claim follows.

Since $\Hom_{\BE_{\{w,ws\}}(\h,W|s)}(\mathscr{F},\mathscr{G}[1])=0$ for $\mathscr{F},\mathscr{G}$ in $\mathsf{C}$, this subcategory is also stable under extensions. Since $\mathsf{C}$ generates $\BE_{\{w,ws\}}(\h,W|s)$ as a triangulated category (see Lemma~\ref{lem:bsw-generate}), applying~\cite[Proposition~1.3.13]{bbd} we obtain a t-structure
\[
\left( {}^p \hspace{-1pt} \BE_{\{ws,w\}}(\h,W|s)^{\leq 0}, {}^p \hspace{-1pt} \BE_{\{ws,w\}}(\h,W|s)^{\geq 0} \right)
\]
on $\BE_{\{w,ws\}}(\h,W|s)$ whose nonnegative, resp.~nonpositive, part is generated under extensions by the objects of the form $\mathscr{F}[n]$ with $\mathscr{F}$ in $\mathcal{C}$ and $n\leq 0$, resp.~$n \geq 0$.

To conclude, it remains to prove that
\begin{align*}
{}^p \hspace{-1pt} \BE_{\{ws,w\}}(\h,W|s)^{\leq 0}&={}^p \hspace{-1pt} \BE_{\{ws,w\}}(\h,W)^{\leq 0}\cap\BE_{\{ws,w\}}(\h,W|s),\\
{}^p \hspace{-1pt} \BE_{\{ws,w\}}(\h,W|s)^{\geq 0}&={}^p \hspace{-1pt} \BE_{\{ws,w\}}(\h,W)^{\geq 0}\cap\BE_{\{ws,w\}}(\h,W|s).
\end{align*}
First, we noted above that $\mathsf{C} \subset \fP^{\BE}_{\{w,ws\}}(\h,W)$, so each left-hand side above is contained in the corresponding right-hand side. Now, let $\mathscr{F} \in {}^p \hspace{-1pt} \BE_{\{ws,w\}}(\h,W)^{\leq 0}\cap\BE_{\{ws,w\}}(\h,W|s)$. Consider the truncation triangle
\[
\tau_{\leq 0}(\mathscr{F}) \to \mathscr{F} \to \tau_{>0}(\mathscr{F}) \xrightarrow{[1]}
\]
for the t-structure we have just constructed on $\BE_{\{ws,w\}}(\h,W|s)$. From the containments we have already proved, we see that this triangle identifies with the truncation triangle for the perverse t-structure on $\BE_{\{ws,w\}}(\h,W)$. From our assumption we deduce that $\tau_{>0}(\mathscr{F})=0$, or in other words that $\mathscr{F}$ belongs to ${}^p \hspace{-1pt} \BE_{\{ws,w\}}(\h,W|s)^{\leq 0}$. The remaining equality can be proved similarly.
\end{proof}

For any finite locally closed right $s$-stable subset $I\subset W$, 
we now set
\begin{align*}
{}^p \hspace{-1pt} \BE_{I}(\h,W|s)^{\leq0}=&\\
=\bigl\{\rF\in\BE_{I}&(\h,W|s)\mid \forall w\in I\cap W^s \ (i_{\{ws,w\}}^I)^*(\rF)\in{}^p \hspace{-1pt} \BE_{\{ws,w\}}(\h,W|s)^{\leq0} \bigr\},\\ 
{}^p \hspace{-1pt} \BE_{I}(\h,W|s)^{\geq0}=&\\
=\bigl\{\rF\in\BE_{I}&(\h,W|s)\mid \forall w\in I\cap W^s, \ (i_{\{ws,w\}}^I)^!(\rF)\in{}^p \hspace{-1pt} \BE_{\{ws,w\}}(\h,W|s)^{\geq0}\}.
\end{align*}
The recollement formalism ensures that this defines a t-structure on $\BE_{I}(\h,W|s)$, which we will call the \emph{perverse t-structure}. The same arguments as in the proof of Lemma~\ref{lem:bsw-tstruc} show that we have
\begin{align*}
{}^p \hspace{-1pt} \BE_{I}(\h,W|s)^{\leq 0}&={}^p \hspace{-1pt} \BE_{I}(\h,W)^{\leq 0}\cap\BE_{I}(\h,W|s),\\
{}^p \hspace{-1pt} \BE_{I}(\h,W|s)^{\geq 0}&={}^p \hspace{-1pt} \BE_{I}(\h,W)^{\geq 0}\cap\BE_{I}(\h,W|s).
\end{align*}
In particular, the heart of this t-structure is $\fP_I^{\BE}(\h,W)\cap\BE_I(\h,W|s)$. As for $\fP_I^{\BE}(\h,W)$, any object in this abelian category has finite length.

We now investigate the simple objects in this heart. By the recollement formalism once again (see~\cite[Proposition~1.4.26]{bbd}), these objects can be classified as follows. For any $w \in W^s \cap I$, there exists a unique simple object $\rL_w^{I,s}$ in $\fP_I^{\BE}(\h,W)\cap\BE_I(\h,W|s)$ which belongs to $\BE_{\{\leq w\} \cap I}(\h,W|s)$ and satisfies $(i_{\{w,ws\}}^I)^* \rL_w^{I,s} \cong b_w^s$. Moreover, any simple object in $\fP_I^{\BE}(\h,W)\cap\BE_I(\h,W|s)$ is (up to isomorphism) of the form $\rL_w^{I,s} \langle n \rangle$ for $w \in W^s \cap I$ and $n \in \Z$.

\begin{lema}
\label{lem:simples-Ws}
 For any $w \in W^s \cap I$ we have $\rL_w^{I,s} \cong \rL_w^I$.
\end{lema}

\begin{proof}
We will show that $\rL_w^{I,s}$ satisfies the properties~\eqref{eqn:conditions-Lw0}--\eqref{eqn:conditions-Lw} which characterize $\rL_w^I$.

First, by definition we have $(i_{\{ws,w\}}^{I})^* \rL^{I,s}_w \cong b_w^s$. Using the triangles in~\eqref{eqn:bsw-perverse} we deduce that
\[
 (i_w^{I})^* \rL^{I,s}_w \cong b_w
\]
(so that $\rL_w^{I,s}$ satisfies~\eqref{eqn:conditions-Lw0}) and that
\[
 (i_{ws}^{I})^* \rL^{I,s}_w \cong b_{ws}\langle -1\rangle[1], \quad (i_{ws}^{I})^! \rL^{I,s}_w \cong b_{ws}\langle 1\rangle[-1].
\]
Hence $\rL^{I,s}_w$ satisfies~\eqref{eqn:conditions-Lw} for $y=ws$. 
Now if $y \in I\cap\{<w\}$ and $y \neq ws$, by the analogue~of \eqref{eqn:conditions-Lw} for $\rL^{I,s}_w$ we have
\[
(i_{\{ys,y\}}^{\{\leq w\} \cap I})^* \rL^{I,s}_w\in {}^p \hspace{-1pt} \BE_{\{ys,y\}}(\h,W|s)^{\leq -1}\subset
{}^p \hspace{-1pt} \BE_{\{ys,y\}}(\h,W)^{\leq -1}.
\]
Therefore, $(i_{y}^{\{\leq w\} \cap I})^* \rL^{I,s}_w\in{}^p \hspace{-1pt} \BE_{\{y\}}(\h,W)^{\leq -1}$. One proves similarly that
\[
(i_{y}^{\{\leq w\} \cap I})^! \rL^{I,s}_w\in{}^p \hspace{-1pt} \BE_{\{y\}}(\h,W)^{\geq 1},
\]
which concludes the proof.
\end{proof}

We can finally prove Lemma~\ref{lem:DwBs-composition-factors}.

\begin{proof}[Proof of Lemma~{\rm \ref{lem:DwBs-composition-factors}}]
 By Corollary~\ref{cor:multiplication-Bs}, $\Delta_w \ustar B_s$ belongs to $\BE_{\{\leq ws\}}(\h,W|s)$, and Lem\-ma~\ref{lem:ses-D-w-ws} ensures that it also belongs to $\fP^{\BE}(\h,W)$. Therefore, it belongs to the heart of the perverse t-structure on $\BE_{\{\leq ws\}}(\h,W|s)$. By Lemma~\ref{lem:simples-Ws}, any finite filtration of this object (in the abelian category given by the heart of this t-structure) with simple subquotients can also be viewed as a finite filtration with simple subquotients in $\fP^{\BE}(\h,W)$, and as such these subquotients are of the form $\rL_y \langle n \rangle$ with $ys<y$.
\end{proof}

\subsection{Description of some simple objects}

Under the present assumption that $\bk$ is a field,~\cite[Theorem~6.26]{ew} provides a description of the isomorphism classes of indecomposable objects in the Karoubian envelope $\Diag(\h,W)$ of $\DiagBS^{\oplus}(\h,W)$: for any $w \in W$ there exists a unique indecomposable object $B_w$ (up to isomorphism) which is a direct summand of $B_\uw$ for any reduced expression $\uw$ for $w$, but which is not isomorphic to a direct summand of an object of the form $B_\uv(n)$ with $n \in \Z$ and $\uv$ an expression such that $\ell(\uv)<\ell(w)$. Moreover, the assignment $(w,n) \mapsto B_w(n)$ induces a bijection between $W \times \Z$ and the set of isomorphism classes of indecomposable objects in $\Diag(\h,W)$. As explained in~\cite[Lemma~4.9.1]{amrw1}, the natural functor
\[
 \BE(\h,W) \to \Kb \Diag(\h,W)
\]
is an equivalence of triangulated categories; in particular, using this identification we can see the objects $B_w$ as living in $\BE(\h,W)$.

Recall the ring isomorphism
\begin{equation}
\label{eqn:isom-H-Groth}
\mathcal{H}_{(W,S)} \xrightarrow{\sim} [\BE(\h,W)]_\Delta
\end{equation}
constructed in~\S\ref{ss:Groth-gp}. Recall also the Kazhdan--Lusztig basis $(\underline{H}_w : w \in W)$ consi\-dered e.g.~in~\cite{soergel}. We conclude this section with the following claim, which provides a description of $\rL_w$ in a favorable situation.

\begin{prop}
\label{prop:uH-simple}
Let $w \in W$, and assume that the image of $\underline{H}_w$ under~\eqref{eqn:isom-H-Groth} is the class of $B_w$. Then $B_w \cong \rL_w$.
\end{prop}

\begin{rmk}
The assumption in Proposition~\ref{prop:uH-simple} is always satisfied if $\ell(w) \leq 2$, or if $W$ is finite and $w$ is the longest element in $W$. (In the latter case, this property follows from the fact that we have $B_w \star B_s \cong B_w(1) \oplus B_w(-1)$ for any $s \in S$, using~\cite[Proposition~2.9]{soergel}.) See~\cite{jw} for more examples of situations when this condition is satisfied or not satisfied (in the case when $(W,S)$ is crystallographic). 

Another setting where this assumption is known (for any $w \in W$) is the one considered
in Remark~\ref{rmk:rouquier-complexes}\eqref{it:rouquier-complexes}. Namely, the equivalence between $\Diag(\h,W)$ and the category of Soergel bimodules considered in this remark sends $B_w$ to the indecomposable Soergel bimodule $\mathsf{B}_w$ attached to $w$. In view of this identification, the condition in Proposition~\ref{prop:uH-simple} becomes Soergel's conjecture for $V$, which was proved in~\cite{ew-hodge}.
\end{rmk}

\begin{proof}[Proof of Proposition~{\rm \ref{prop:uH-simple}}]
Recall the characterization of $\rL_w$ given by~\eqref{eqn:conditions-Lw0}--\eqref{eqn:conditions-Lw}.
We will show that the object $B_w$ satisfies the first condition in~\eqref{eqn:conditions-Lw};
the second one can be either proved similarly or deduced by duality, and~\eqref{eqn:conditions-Lw0} is easy (and left to the reader).

Let $y \in W$ be such that $y<w$. Writing $\{y\}$ as an intersection $\{\leq y\} \cap I$ where $I \subset \{\leq w\}$ is open and $y$ is minimal in $I$ and using Remark~\ref{rmk:!*}\eqref{it:iw-1degree}, we see that $(i_y^{\{\leq w\}})^* B_w$ is isomorphic to an object of the form
\[
\bigoplus_{m \in \Z} \bigl( b_y (m) \bigr)^{\oplus \lambda_m}
\]
for some coefficients $\lambda_m \in \Z_{\geq 0}$ (with $\lambda_m=0$ for all but finitely many $m$'s). In terms of this decomposition, the class of this object in $[\BE_{\{y\}}(\h,W)]$ is then
\[
\sum_{m \in \Z} \lambda_{m} \cdot [b_y ( m )]
=
\sum_{m \in \Z} v^m \lambda_{m} \cdot [b_y]
\]
On the other hand, the coefficient of $\uH_w$ on $H_y$ (in the standard basis) belongs to $v \Z[v]$. Hence our assumption implies that $\lambda_m=0$ for $m \leq 0$, so that
\[
(i_y^{\{\leq w\}})^* B_w \cong \bigoplus_{m \in \Z_{>0}} \bigl( b_y \langle -m \rangle [m] \bigr)^{\oplus \lambda_m}.
\]
Here the right-hand side belongs to ${}^p \hspace{-1pt} \BE_{\{y\}}(\h,W)^{\leq -1}$, and the desired claim is proved.
\end{proof}

\section{The right-equivariant category}
\label{sec:RE-category}

Recall the categories $\oDiagBS^\oplus(\h,W)$ and $\RE(\h,W)$ introduced in~\S\ref{ss:BE-RE}. The goal of the present section is to briefly indicate how most of the results considered so far adapt to these categories, allowing us to define the category $\fP^{\RE}(\h,W)$ of right- equivariant perverse objects.

\subsection{Diagrammatic categories attached to locally closed subsets}
\label{ss:oDiag-loc-closed}

In~\S\S\ref{ss:oDiag-loc-closed}--\ref{ss:RE-loc-closed}, $\bk$ is an arbitrary integral domain.

Let $I \subset W$ be a closed subset. We define $\oDiag_{\BS,I}^{\oplus}(\h,W)$ to be the full subcategory of $\oDiagBS^\oplus(\h,W)$ whose objects are direct sums of objects of the form $\oB_\uw(n)$ with $\uw$ a reduced expression for an element in $I$. 
The autoequivalence $(1)$ of $\DiagBS^\oplus(\h,W)$ induces an autoequivalence of $\oDiagBS^\oplus(\h,W)$, which in turn restricts to an autoequivalence of $\oDiag_{\BS,I}^{\oplus}(\h,W)$. All of these autoequivalences will be denoted similarly, and we will use the notation
$\Hom^{\bullet}$ in these categories with the same conventions as in~\eqref{eqn:def-Hombullet}.

If $J \subset I \subset W$ are closed subsets, then there exists a natural embedding
\[
 (i_J^I)_* : \oDiag_{\BS,J}^{\oplus}(\h,W) \to \oDiag_{\BS,I}^{\oplus}(\h,W).
\]

Next, if $I \subset W$ is a locally closed subset, and if we write $I=I_0 \smallsetminus I_1$ with $I_1 \subset I_0 \subset W$ closed, then we set
\[
 \oDiag_{\BS,I}^\oplus(\h,W) := \oDiag_{\BS,I_0}^{\oplus}(\h,W)\qq\oDiag_{\BS,I_1}^{\oplus}(\h,W),
\]
where the symbol ``$\qq \,$'' has the same meaning as in~\S\ref{ss:Diag-loc-closed}.
The natural functor $\DiagBS^\oplus(\h,W) \to \oDiagBS^\oplus(\h,W)$ restricts to a functor $\Diag_{\BS,I_0}^\oplus(\h,W) \to \oDiag_{\BS,I_0}^\oplus(\h,W)$, which in turn induces a functor
\[
 \Diag_{\BS,I}^\oplus(\h,W) \to \oDiag_{\BS,I}^\oplus(\h,W).
\]
From the definitions we see that this functor (which we will denote $M \mapsto \overline{M}$) induces an isomorphism
\[
 \bk \otimes_R \Hom^{\bullet}_{\Diag_{\BS,I}^\oplus(\h,W)}(M,N) \xrightarrow{\sim} \Hom^{\bullet}_{\oDiag_{\BS,I}^\oplus(\h,W)}(\overline{M},\overline{N}).
\]
Using this, the considerations of~\S\ref{ss:Diag-loc-closed} show that the category $\oDiag_{\BS,I}^\oplus(\h,W)$ does not depend on the choice of $I_0$ and $I_1$ (up to canonical equivalence), and that the morphism spaces in this category are free of finite rank over $\bk$.

If $w \in W$, we denote by $\overline{b}_w$ the image of $b_w$ in $\oDiag_{\BS,\{w\}}^{\oplus}(\h,W)$.
Lemma~\ref{le:Dw is equivalent to free R} also implies the following claim.

\begin{lema}
\label{le:oDw}
There exists a canonical equivalence of categories 
\[
\overline{\gamma}:\oDiag_{\BS,\{w\}}^{\oplus}(\h,W)\xrightarrow{\sim}\Free^{\fgen,\Z}(\bk)
\]
such that $\overline{\gamma}(\overline{b}_w) = \bk$.
Under this equivalence, the autoequivalence $(1)$ identifies with the ``shift of grading'' autoequivalence of $\Free^{\fgen,\Z}(\bk)$ defined by $\bigl( M(1) \bigr)^n = M^{n+1}$.
\end{lema}

Finally, if $J \subset I$ is a closed (resp.~open) subset, then there exists a natural functor
\[
 (i_J^I)_* : \oDiag_{\BS,J}^\oplus(\h,W) \to \oDiag_{\BS,I}^\oplus(\h,W), \quad \text{resp.} \quad (i_J^I)^* : \oDiag_{\BS,I}^\oplus(\h,W) \to \oDiag_{\BS,J}^\oplus(\h,W).
\]
We also have a ``duality'' functor $\mathbb{D}_I$ on $\oDiag_{\BS,I}^\oplus(\h,W)$, and compatibility properties similar to those stated in~\S\ref{ss:inclusions}.

\subsection{Right-equivariant categories attached to locally closed subsets and recollement}
\label{ss:RE-loc-closed}

If $I \subset W$ is a locally closed subset, we set
\[
 \RE_I(\h,W) := \Kb \oDiag_{\BS,I}^\oplus(\h,W).
\]
All the constructions of~\S\ref{ss:BE-category} adapt to this setting, and we obtain functors that will be denoted by the same symbol as in the case of $\BE_I(\h,W)$. We also have a natural ``forgetful'' functor
\[
 \For^{\BE}_{\RE} : \BE_I(\h,W) \to \RE_I(\h,W).
\]

\begin{prop}
\label{prop:recollement-RE}
Let $I \subset W$ be a locally closed subset, and let $J\subset I$ be a finite closed subset. Then the functor $(i_{I \smallsetminus J}^I)^*:\RE_I(\h,W)\to\RE_{I\smallsetminus J}(\h,W)$ admits a left adjoint 
$(i_{I \smallsetminus J}^I)_!$ and a right adjoint 
$(i_{I \smallsetminus J}^I)_*$. 
Similarly, the functor $(i_J^I)_*:\RE_J(\h,W)\to\RE_{I}(\h,W)$ admits a left adjoint $(i_J^I)^*$ and a right adjoint $(i_J^I)^!$. Together, these functors give a recollement diagram
\[
 \xymatrix@C=2cm{
 \RE_J(\h,W) \ar[r]|-{(i_J^I)_*} & \RE_I(\h,W) \ar[r]|-{(i_{I \smallsetminus J}^I)^*} \ar@/^0.5cm/[l]^-{(i_J^I)^!} \ar@/_0.5cm/[l]_-{(i_J^I)^*} & \RE_{I\smallsetminus J}(\h,W). \ar@/^0.5cm/[l]^-{(i_{I \smallsetminus J}^I)_*} \ar@/_0.5cm/[l]_-{(i_{I \smallsetminus J}^I)_!}
 }
\]
\end{prop}

\begin{proof}
  The proof is identical to that of Proposition~\ref{prop:recollement}, starting with the case $|J|=1$ and then using induction on $|J|$. Details are left to the reader. (In the case where $J=\{w\}$, we replace the complex $B_\ux^+$ by its image $\oB_\ux^+$ in $\RE_I(\h,W)$, which fits into a distinguished triangle
\[
\oB_\ux \to \oB_\ux^+ \to \oB_\uw \Tenint_\bk \Hom^\bullet_{\oDiag_{\BS,I}^\oplus(\h,W)}(\oB_\uw,\oB_\ux)[1] \xrightarrow{[1]},
\]
where the third term is defined in the natural way.)
\end{proof}

Starting with Proposition~\ref{prop:recollement-RE}, we can define as in~\S\ref{ss:pf-pb-locally-closed} the functors $(i_J^I)^*$, $(i_J^I)^!$, $(i_J^I)_*$, $(i_J^I)_!$ for any locally closed embedding $J \subset I$ of finite subsets of $W$. These functors satisfy the appropriate analogue of Lemma~\ref{le:recollement compositions for locally closed inclusions}. Moreover, there exist canonical isomorphisms
\begin{multline*}
 (i_J^I)^* \circ \For^{\BE}_{\RE} \cong \For^{\BE}_{\RE} \circ (i_J^I)^*, \quad (i_J^I)^! \circ \For^{\BE}_{\RE} \cong \For^{\BE}_{\RE} \circ (i_J^I)^!, \\
 (i_J^I)_* \circ \For^{\BE}_{\RE} \cong \For^{\BE}_{\RE} \circ (i_J^I)_*, \quad (i_J^I)_! \circ \For^{\BE}_{\RE} \cong \For^{\BE}_{\RE} \circ (i_J^I)_!,
\end{multline*}
where in each case the functor on the left-hand side is defined for $\RE$ categories, while the functor on the right-hand side is defined for $\BE$ categories. In fact it suffices to prove these isomorphisms in case $J$ is either open or closed in $I$. In this case, the claim is either obvious or follows from the construction. (For instance, we observe that in the construction for Lemma~\ref{le:recollement step 1a} and its counterpart for the $\RE$ categories, the subcategories $\fD^+ \subset \BE_I(\h,W)$ and $\overline{\fD}^+ \subset \RE_I(\h,W)$ satisfy $\For^{\BE}_{\RE}(\fD^+) \subset \overline{\fD}^+$.)

\subsection{Perverse t-structure}

From now on we assume that $\bk$ satisfies the conditions of~\S\ref{ss:rings}. Then, as in Lemma~\ref{lem:DbMod-KbFree} we have a canonical equivalence of triangulated categories
\begin{equation}
\label{eqn:Kbfree-DbMod-k}
 \Kb \Free^{\fgen,\Z}(\bk) \xrightarrow{\sim} \Db \Mod^{\fgen,\Z}(\bk).
\end{equation}
This gives rise to an equivalence $\RE_{\{w\}}(\h,W) \cong \Db \Mod^{\fgen,\Z}(\bk)$, analogous to~\eqref{eqn:equiv-BEw}.  However, there is also a \emph{different} equivalence, described in the following lemma, that has no direct analogue in the setting of $\BE_{\{w\}}(\h,W)$.

\begin{lema}
\label{lem:REw-DbMod}
Let $w \in W$.
 There exists an equivalence of triangulated categories
 \[
  \RE_{\{w\}}(\h,W) \xrightarrow{\sim} \Db \Mod^{\fgen,\Z}(\bk)
 \]
such that the autoequivalence $\langle 1 \rangle$ of $\RE_{\{w\}}(\h,W)$ corresponds to the autoequivalence $\Db((-1))$ of $\Db \Mod^{\fgen,\Z}(\bk)$, where $(-1)$ is the inverse of the ``shift of grading'' autoequivalence of $\Mod^{\fgen,\Z}(\bk)$ defined as in Lemma~{\rm \ref{le:Dw is equivalent to free R}}.
 \end{lema}

 \begin{proof}
  We consider the composition
  \[
   \RE_{\{w\}}(\h,W) \xrightarrow[\sim]{\Kb(\overline{\gamma})} \Kb \Free^{\fgen,\Z}(\bk) \xrightarrow[\sim]{\eqref{eqn:Kbfree-DbMod-k}} \Db \Mod^{\fgen,\Z}(\bk) \xrightarrow[\sim]{\zeta} \Db \Mod^{\fgen,\Z}(\bk)
  \]
where $\zeta$ is the equivalence sending a complex $(M^{n,m})_{n \in \mathbb{Z}}$ of graded $\bk$-modules (where $M^{n,m}$ means the part in cohomological degree $n$ and ``internal'' degree $m$) to the complex with $\zeta(M)^{n,m}:=M^{n-m,m}$ (and the same differential). It is straightforward to check that this equivalence has the required property with respect to the functor $\langle 1 \rangle$.
 \end{proof}

 We now define the \emph{perverse t-structure} on $\RE_{\{w\}}(\h,W)$ as the transport of the tautological t-structure on $\Db \Mod^{\fgen,\Z}(\bk)$ under the equivalence of Lemma~\ref{lem:REw-DbMod}. Under the equivalence~\eqref{eqn:equiv-BEw} and that induced by the equivalence of Lemma~\ref{le:oDw}, the functor $\For^{\BE}_{\RE}$ corresponds to the functor $\bk \lotimes_R (-)$. Hence, in view of~\cite[Proposition~4.4]{prinblock}, we deduce that $\For^{\BE}_{\RE} : \BE_{\{w\}}(\h,W) \to \RE_{\{w\}}(\h,W)$ is t-exact for the perverse t-structures; more precisely, an object $\mathscr{F}$ of $\BE_{\{w\}}(\h,W)$ is perverse if and only if $\For^{\BE}_{\RE}(\mathscr{F})$ is perverse in $\RE_{\{w\}}(\h,W)$.
 
 Once the perverse t-structure is defined on $\RE_{\{w\}}(\h,W)$ for any $w \in W$, as in~\S\ref{ss:def-t-structure}, using recollement we can define a perverse t-structure on $\RE_I(\h,W)$ for any locally closed subset $I \subset W$. The heart of this t-structure will be denoted by $\fP^{\RE}_I(\h,W)$. The remarks above show that $\For^{\BE}_{\RE} : \BE_{I}(\h,W) \to \RE_{I}(\h,W)$ is t-exact for the perverse t-structures; in fact an object $\mathscr{F}$ of $\BE_{I}(\h,W)$ is perverse if and only if $\For^{\BE}_{\RE}(\mathscr{F})$ is perverse in $\RE_{I}(\h,W)$.
 
 We define the \emph{standard} and \emph{costandard} objects in $\RE_I(\h,W)$ by
 \[
  \oD_w^I := \For^{\BE}_{\RE}(\Delta^I_w), \qquad \oN^I_w := \For^{\BE}_{\RE}(\nabla^I_w).
 \]
Corollary~\ref{cor:D-N-perverse-I} and the t-exactness of $\For^{\BE}_{\RE}$ imply that these objects belong to $\fP^{\RE}_I(\h,W)$. Moreover, the same proof as for Lemma~\ref{lem:Hom-BE-D-N} shows that for $x,y \in I$ we have
\begin{equation}
\label{eqn:Hom-RE-D-N}
 \Hom_{\RE_I(\h,W)}(\oD^I_x, \oN^I_y \langle n \rangle [m]) \cong
 \begin{cases}
   \bk & \text{if $x=y$ and $m=n=0$;}\\
   0 & \text{otherwise.}
 \end{cases}
\end{equation}

\subsection{Right-equivariant and bi-equivariant perverse sheaves}

The goal of this subsection is to prove the following claim.

\begin{prop}
\label{prop:perv-equiv-const}
Let $I \subset W$ be a locally closed subset. Then the functor
\[
 \fP^{\BE}_I(\h,W) \to \fP^{\RE}_I(\h,W)
\]
obtained by restricting
$\For^{\BE}_{\RE}$ to perverse objects is fully faithful.
\end{prop}

The proof of this proposition will rely on the following lemma from homological algebra. We regard (as, for instance, in~\cite[\S 3.2]{amrw1}) $R$ as a $\Z$-graded dg-algebra\footnote{In what follows, the ``internal'' grading will play no role, and hence can be forgotten.} with generators in bidegree $(2,2)$. Consider also $\Lambda:=\mathrm{Sym}(V^*(-2)[1])$ (so that $\Lambda$ is the exterior algebra of $V^*$, regarded as a bigraded ring with generators in bidegree $(1,2)$). For any $\Z$-graded $R$-dg-module $M$, the bigraded $\bk$-module $\Lambda \otimes_\bk M$ admits a natural structure of $\Z$-graded $\bk$-dg-module, with (Koszul-type) differential given by
\begin{multline*}
 d \bigl( (r_1 \wedge \cdots \wedge r_k) \otimes m) = \\
 \sum_{j=1}^k (-1)^{j-1} \cdot (r_1 \wedge \cdots \wedge \widehat{r_j} \wedge \cdots \wedge r_k) \otimes (r_j \cdot m) + (-1)^k \cdot (r_1 \wedge \cdots \wedge r_k) \otimes d(m).
\end{multline*}
It is clear that the assignment $m \mapsto 1 \otimes m$ defines a morphism of $\Z$-graded dg-modules $M \to \Lambda \otimes_\bk M$.

\begin{lema}
\label{lem:morph-equiv-const}
 Assume that $\coH^{<0}(M)=0$. Then the morphism $M \to \Lambda \otimes_\bk M$ induces an isomorphism
 \[
  \coH^0(M) \xrightarrow{\sim} \coH^0(\Lambda \otimes_\bk M).
 \]
\end{lema}

\begin{proof}
 Let us fix a basis $(e_1, \ldots, e_m)$ of $V^*$. Then any homogenous element $x$ of $\Lambda \otimes_\bk M$ can be written uniquely as a sum
 \begin{equation}
  \label{eqn:decomp-x}
  x=\sum_{1 \leq i_1 < \cdots < i_k \leq m} (e_{i_1} \wedge \cdots \wedge e_{i_k}) \otimes m_{i_1, \ldots, i_k}
 \end{equation}
 with $m_{i_1, \ldots, i_k}$ homogeneous, of cohomological degree $\deg(x)-k$.
 
First we prove that our morphism is surjective. For this, assume that $x$ has cohomological degree $0$ and that $d(x)=0$, and choose a sequence $\underline{i}:=(i_1, \ldots, i_k)$ with $k$ maximal such that $m_{i_1, \ldots, i_k} \neq 0$. If $k=0$ then the class of $x$ belongs to the image of our morphism. Otherwise, let
\[
 y:= x - (e_{i_1} \wedge \cdots \wedge e_{i_k}) \otimes m_{\underline{i}}.
\]
Then $d(x)$ equals
\[
 d(y) + \sum_{j=1}^k (-1)^{j-1} \cdot (e_{i_1} \wedge \cdots \wedge \widehat{e_{i_j}} \wedge \cdots \wedge e_{i_k}) \otimes (e_{i_j} \cdot m_{\underline{i}}) + (-1)^k \cdot (e_{i_1} \wedge \cdots \wedge e_{i_k}) \otimes d(m_{\underline{i}}).
\]
Since $d(x)=0$, the maximality of $k$ implies that $d(m_{\underline{i}})=0$.
Now $m_{\underline{i}}$ has strictly negative cohomological degree, so by our assumption there exists $n_{\underline{i}}$ in $M$ such that $m_{\underline{i}} = d(n_{\underline{i}})$. We then set
\[
 x':= y - \sum_{j=1}^k (-1)^{k+j-1} \cdot (e_{i_1} \wedge \cdots \wedge \widehat{e_{i_j}} \wedge \cdots \wedge e_{i_k}) \otimes (e_{i_j} \cdot n_{\underline{i}}).
\]
The element $x-x'$ equals
\begin{multline*}
  (e_{i_1} \wedge \cdots \wedge e_{i_k}) \otimes m_{\underline{i}} + \sum_{j=1}^k (-1)^{k+j-1} \cdot (e_{i_1} \wedge \cdots \wedge \widehat{e_{i_j}} \wedge \cdots \wedge e_{i_k}) \otimes (e_{i_j} \cdot n_{\underline{i}}) \\
  = d \bigl( (-1)^k \cdot (e_{i_1} \wedge \cdots \wedge e_{i_k}) \otimes n_{\underline{i}} \bigr).
\end{multline*}
Hence $d(x')=0$, and $x$ and $x'$ have the same image in $\coH^0(\Lambda \otimes_\bk M)$. Repeating this procedure if necessary, we can decrease the number of terms in the decomposition~\eqref{eqn:decomp-x} attached to sequences of length $k$, and then decrease the maximal length of a sequence $\underline{j}$ such that $m_{\underline{j}} \neq 0$, and obtain finally that the image of $x$ in $\coH^0(\Lambda \otimes_\bk M)$ belongs to the image of $\coH^0(M)$.

Now we prove injectivity of our morphism. Let $x \in M^0$ be such that $d_M(x)=0$, and assume that the image of $x$ in $\coH^0(\Lambda \otimes_\bk M)$ vanishes, or in other words that $x=d(y)$ for some $y$ in $(\Lambda \otimes_\bk M)^{-1}$. Write $y$ as in~\eqref{eqn:decomp-x}. If we assume that there exists a sequence $\underline{i}:=(i_1, \ldots, i_k)$ with $k>0$ such that $m_{i_1, \ldots, i_k} \neq 0$, then we choose such a sequence with $k$ maximal. The fact that $d(y)=x$ implies that $d(m_{\underline{i}})=0$, and as above $y=y'+d(z)$ for some $z$ in $(\Lambda \otimes_\bk M)^{-2}$. Then $x=d(y')$, and repeating this procedure if necessary we obtain that the class of $x$ vanishes in $\coH^0(M)$.
\end{proof}

\begin{proof}[Proof of Proposition~{\rm \ref{prop:perv-equiv-const}}]
In this proof we assume the reader has some familiarity with the constructions of~\cite[Chap.~4]{amrw1}.

As in~\cite{amrw1}, one can define the notion of $\Diag_{\BS,I}^\oplus(\h,W)$-sequence and, for any such sequences $\mathscr{F}$ and $\mathscr{G}$, consider the bigraded $\bk$-module $\underline{\Hom}_{\BE,I}(\mathscr{F},\mathscr{G})$. Then as in~\cite[\S 4.2]{amrw1} one can describe the category $\BE_I(\h,W)$ as the category of pairs $(\mathscr{F},\delta)$ with $\delta$ in $\underline{\Hom}_{\BE,I}(\mathscr{F},\mathscr{F})^{(1,0)}$ which satisfies $\delta \circ \delta=0$, with appropriately defined morphisms. As in~\cite[\S 4.3]{amrw1} one also has a similar description for $\RE_I(\h,W)$, replacing $\underline{\Hom}_{\BE,I}(\mathscr{F},\mathscr{G})$ with $\bk \otimes_R \underline{\Hom}_{\BE,I}(\mathscr{F},\mathscr{G})$. Finally, replacing $\underline{\Hom}_{\BE,I}(\mathscr{F},\mathscr{G})$ with $\Lambda \otimes_\bk \underline{\Hom}_{\BE,I}(\mathscr{F},\mathscr{G})$, one obtains the category $\mathsf{LM}_I(\h,W)$ of ``left-monodromic complexes''; see~\cite[\S 4.4]{amrw1}. With this notation, the functor $\For^{\BE}_{\RE}$ factors as a composition
\[
 \BE_I(\h,W) \xrightarrow{\For^{\BE}_{\mathsf{LM}}} \mathsf{LM}_I(\h,W) \xrightarrow{\For^{\mathsf{LM}}_{\RE}} \RE_I(\h,W).
\]
Moreover, as in~\cite[Theorem~4.6.2]{amrw1}, the functor $\For^{\mathsf{LM}}_{\RE}$ is an equivalence of categories.

Now, let $\mathscr{F},\mathscr{G}$ be in $\fP^{\BE}_I(\h,W)$. Then the morphisms from $\mathscr{F}$ to shifts of $\mathscr{G}$ can be computed as the cohomology of the complex of $\Z$-graded $R$-dg-modules $\underline{\Hom}_{\BE,I}(\mathscr{F},\mathscr{G})$. Since $\mathscr{F}$ and $\mathscr{G}$ belong to the heart of a t-structure, this complex has no negative cohomology. On the other hand, with the same conventions as above, the complex $\Lambda \otimes_{\bk} \underline{\Hom}_{\BE,I}(\mathscr{F},\mathscr{G})$ computes the morphisms from $\For^{\BE}_{\mathsf{LM}}(\mathscr{F})$ to shifts of $\For^{\BE}_{\mathsf{LM}}(\mathscr{G})$. Hence Lemma~\ref{lem:morph-equiv-const} shows that $\For^{\BE}_{\mathsf{LM}}$ induces an isomorphism
\[
 \Hom_{\BE_I(\h,W)}(\mathscr{F}, \mathscr{G}) \xrightarrow{\sim} \Hom_{\mathsf{LM}_I(\h,W)}(\For^{\BE}_{\mathsf{LM}}(\mathscr{F}), \For^{\BE}_{\mathsf{LM}}(\mathscr{G})).
\]
Since $\For_{\RE}^{\mathsf{LM}}$ is an equivalence of categories, the claim of the proposition follows.
\end{proof}

\subsection{The case of field coefficients}
\label{ss:RE-field}

In this subsection we assume that $\bk$ is a field. In this case, as in~\S\ref{ss:simple-perverse}, the recollement formalism provides a description of the simple objects in the abelian category $\fP^{\RE}_I(\h,W)$. In fact, t-exactness of $\For^{\BE}_{\RE}$ implies that these simple objects are (up to isomorphism) exactly the objects $\overline{\rL}^I_w \langle n \rangle$ with $(w,n) \in W \times \mathbb{Z}$, where $\overline{\rL}^I_w:=\For^{\BE}_{\RE}(\rL^I_w)$. 

The following result is the main reason that motivates the generalization of our constructions to the $\RE$ categories. It uses the concept of (graded) \emph{highest weight category} due Cline--Parshall--Scott; see~\cite[Definition~A.1]{modrap2} for the definition we want to use (except that we replace Axiom~(1) by the weaker condition that for any $s \in \mathscr{S}$ the set $\{t \in \mathscr{S} \mid t \leq s\}$ is finite).

\begin{theorem}
\label{thm:hw}
 Let $I \subset W$ be a locally closed subset. The category $\fP^{\RE}_I(\h,W)$ is a graded highest weight category with weight poset $(I,\leq)$, normalized standard objects $(\oD^I_w : w \in I)$ and normalized costandard objects $(\oN^I_w : w \in I)$.
\end{theorem}

\begin{proof}
 The first axiom is obviously satisfied. For the second one, we observe that the surjection $\oD_w^I \twoheadrightarrow \overline{\rL}^I_w$ and the injection $\overline{\rL}^I_w \hookrightarrow \oN^I_w$ induce an embedding
 \[
  \Hom_{\fP_I^{\RE}(\h,W)}(\overline{\rL}^I_w, \overline{\rL}^I_w \langle n \rangle) \hookrightarrow \Hom_{\fP_I^{\RE}(\h,W)}(\oD^I_w, \oN^I_w \langle n \rangle);
 \]
then the desired claim follows from~\eqref{eqn:Hom-RE-D-N}. To check the third axiom, we consider $J \subset I$ closed and $w \in J$ maximal. Then $\oD^I_w$ belongs to (the essential image of) $\fP^{\RE}_J(\h,W)$, and if $M$ belongs to $\fP^{\RE}_J(\h,W)$ we have
\[
 \Ext^1_{\fP^{\RE}_J(\h,W)}(\oD^I_w, M) \cong \Hom_{\RE_J(\h,W)}(\oD^J_w, M[1]) \cong \Hom_{\RE_{\{w\}}(\h,W)}(\overline{b}_w, (i_w^J)^*M [1]),
\]
which vanishes since $(i_w^J)^*M$ is a perverse object. One checks similarly that
\[
 \Ext^1_{\fP^{\RE}_J(\h,W)}(M,\oN^I_w)=0.
\]
The fourth axiom follows from Lemma~\ref{lem:ker-standard-object-simple}. Finally, the fifth axiom in the definition of highest weight categories follows from~\eqref{eqn:Hom-RE-D-N} and~\cite[Lemma~3.2.4]{bgs}.
\end{proof}

Once Theorem~\ref{thm:hw} is established, one can consider the \emph{tilting} objects in the highest weight category $\fP^{\RE}_I(\h,W)$, i.e.~the objects which admit both a standard filtration and a costandard filtration; see~\cite[Definition~A.2]{modrap2}. As in~\cite{modrap2} we will use the notation $(\mathscr{T} : \oD_w^I\langle n \rangle)$ (or $(\mathscr{T} : \oN_w^I\langle n \rangle)$) for multiplicities of standard (or costandard) objects in a standard (or costandard) filtration. The indecomposable tilting objects are classified in the following way. For any $w \in W$, there exists a unique (up to isomorphism) indecomposable tilting object $\mathscr{T}^I_w$ in $\fP^{\RE}_I(\h,W)$ which satisfies
\[
 (\mathscr{T}^I_w : \oD_w^I)=1 \quad \text{and} \quad \bigl((\mathscr{T}^I_w : \oD_x^I \langle n \rangle) \neq 0 \Rightarrow x\leq w \bigr);
\]
moreover the assignment $(w,n) \mapsto \mathscr{T}^I_w \langle n \rangle$ induces a bijection between $I \times \mathbb{Z}$ and the set of isomorphism classes of indecomposable tilting objects. (See~\cite[Proposition~A.4]{modrap2}, and~\cite[Theorem~7.14]{riche-hab} for more details in the ungraded setting.) By uniqueness we have $\D_I(\mathscr{T}^I_w) \cong \mathscr{T}
^I_w$. Moreover there exists a surjection
\begin{equation}
\label{eqn:surj-T-nabla}
 \mathscr{T}_w^I \twoheadrightarrow \oN_w^I,
\end{equation}
resp.~an embedding
\begin{equation}
 \oD_w^I \hookrightarrow \mathscr{T}_w^I,
\end{equation}
whose kernel, resp.~cokernel, admits a costandard, resp.~standard, filtration.

The study of such objects is particularly important in view of the following result, which follows from~\cite[Lemma~A.5 and Lemma~A.6]{modrap2}. Here we denote by $\Tilt^{\RE}_I(\h,W)$ the full subcategory of $\fP^{\RE}_I(\h,W)$ consisting of tilting objects.

\begin{theorem}
The natural functors
\[
\Kb \Tilt_I^{\RE}(\h,W) \to \Db \fP_I^{\RE}(\h,W) \to \RE_I(\h,W)
\]
are equivalences of triangulated categories.
\end{theorem}

We conclude this section by noting that, with the theory we developed here in hand, the results obtained in~\cite[\S\S 10.5--10.7]{amrw1} generalize to the present setting. In particular, the tilting objects in $\fP^{\RE}(\h,W)$ can be produced via a ``Bott--Samelson type'' construction.

\section{Ringel duality and the big tilting object}
\label{sec:Ringel}

In this section we assume that $W$ is finite, and denote by $w_0$ the longest element of $W$. 

\subsection{Ringel duality}

By Proposition~\ref{prop:D-N-convolution}\eqref{it:Delta-nabla-convolution-2}, the functor
\[
\fR:=(-)\ustar\Delta_{w_0}:\RE(\h,W)\to\RE(\h,W)
\]
is an equivalence of triangulated categories, with quasi-inverse
\[
\fR^{-1}:=(-)\ustar\nabla_{w_0}:\RE(\h,W)\to\RE(\h,W).
\]

\begin{lema}
\label{lem:R-N-D}
For any $w \in W$ we have
\[
\fR(\onabla_x)\simeq\oDelta_{xw_0}.
\]
\end{lema}

\begin{proof}
The desired isomorphism follows from the sequence of isomorphisms
\begin{multline*}
\fR(\onabla_x) = \onabla_x \ustar \Delta_{w_0} \cong \For^\BE_\RE(\nabla_x \ustar \Delta_{w_0}) \cong \\
\For^\BE_\RE(\nabla_x \ustar \Delta_{x^{-1}} \ustar \Delta_{xw_0}) \cong \For^\BE_\RE(\Delta_{xw_0}) \cong \oDelta_{xw_0},
\end{multline*}
where the first isomorphism is a special case of~\eqref{eqn:For-convolution}, and the second and third ones follow from Proposition~\ref{prop:D-N-convolution}.
\end{proof}

\subsection{Projective and tilting perverse objects}

From now on we assume that $\bk$ is a field.
Then the category $\fP^{\RE}(\h,W)$ has enough projective and injective objects,
and any projective (resp.~injective) object admits a standard (resp.~costandard) filtration, cf. \cite[Theorem A.3]{modrap2}. For $x \in W$, we will denote by $\rP_x$, resp.~$\rI_x$, the projective cover, resp.~injective hull, of $\overline{\rL}_x$.
Recall the reciprocity formula
\begin{align}\label{eq:reciprocity formula}
(\rP_x:\oDelta_y\langle n\rangle)=[\onabla_y\langle n\rangle:\overline{\rL}_x],
\end{align}
where $x,y\in W$ and $n\in\Z$ (see~\cite[remarks after Theorem~3.2.1]{bgs}).

It is a direct consequence of Lemma~\ref{lem:R-N-D} that if $\mathscr{M}$ is a perverse object which admits a costandard filtration, then $\fR(\mathscr{M})$ is also perverse, and it admits a standard filtration; moreover we have
\begin{equation}
\label{eqn:multiplicities-R}
 \left(\fR(\mathscr{M}):\oDelta_{xw_0}\langle n\rangle\right)=\left(\mathscr{M}:\onabla_x\langle n\rangle\right) 
\end{equation}
for all $x\in W$ and $n\in\Z$.

\begin{prop}
\label{prop:ringel duality}
For any $x \in W$ we have
\[
 \fR(\rT_x)\cong\rP_{xw_0}, \qquad \fR(\rI_x) \cong \rT_{xw_0}.
\]
\end{prop}

\begin{proof}
Let $\rT$ be a tilting object in $\fP^{\RE}(\h,W)$. Then $\fR(\rT)$ is perverse by the comments before the proposition. We claim that $\fR(\rT)$ is projective. 
In fact, by Lemma~\ref{lem:R-N-D}, for $y \in W$ and $n,m \in \Z$ we have
\[
\Hom_{\RE(\h,W)}(\fR(\rT),\oDelta_y\langle n\rangle[m])
\cong \Hom_{\RE(\h,W)}(\rT,\onabla_{yw_0}\langle n\rangle[m]).
\]
Since $\rT$ admits a standard filtration,~\eqref{eqn:Hom-RE-D-N} implies that this vector space vanishes unless $m=0$. Using the analogue of Lemma~\ref{lem:t-structure-D-N} for the right-equivariant category, we deduce that
\[
 \Hom_{\RE(\h,W)}(\fR(\rT),\mathscr{M})=0
\]
for any $\mathscr{M}$ in ${}^p \hspace{-1pt}\RE(\h,W)^{<0}$. In particular this shows that
\[
 \Ext^1_{\fP^{\RE}(\h,W)}(\fR(\rT), \mathscr{N})=0
\]
for any $\mathscr{N}$ in $\fP^{\RE}(\h,W)$, and hence that $\fR(\rT)$ is projective, as claimed.

If now $\rT=\rT_x$, then since $\fR$ is an equivalence of categories, $\fR(\rT_x)$ is indecomposable. Moreover, the kernel of the natural surjection $\rT_x \twoheadrightarrow \onabla_x$ admits a costandard filtration; hence its image $\fR(\rT_x) \to \oDelta_{xw_0}$ under $\fR$ is surjective. This shows that $\fR(\rT_x)$ surjects to $\rL_{xw_0}$, and hence that $\fR(\rT_x) \cong \rP_{xw_0}$.

Very similar arguments show that $\fR^{-1}(\rT_{xw_0})$ belongs to $\fP^{\RE}(\h,W)$, is indecomposable and injective therein, and contains $\rL_x$ as a simple subobject. Therefore, as above we have $\fR^{-1}(\rT_{xw_0}) \cong \rI_x$, which concludes the proof.
\end{proof}

\subsection{The big tilting object}

The following theorem is an analogue of a well-known result in category $\mathscr{O}$ which is the starting point of the ``Soergel-theoretic'' analysis of this category.

\begin{theorem}\label{thm:on tilting w0}
There exist isomorphisms
\begin{equation}
\label{eqn:isom-Pe-Tw0-Ie}
 \rT_{w_0}\cong\rP_e\langle\ell(w_0)\rangle\cong\rI_e\langle-\ell(w_0)\rangle.
\end{equation}
Moreover, for any $x \in W$ we have
\begin{equation}
\label{eq:1 on tilting w0}
\left(\rT_{w_0}:\onabla_{x}\langle-n\rangle\right)=\left(\rT_{w_0}:\oDelta_{x}\langle n\rangle\right)=\begin{cases}
                                                                                1 & \text{if $n=\ell(xw_0)$;}\\
                                                                                0 & \mbox{otherwise.}
                                                                               \end{cases}
\end{equation}
\end{theorem}

\begin{proof}
First, we note that $\left(\rP_e:\oDelta_x\langle -n\rangle\right)$ is $1$ if $n=\ell(x)$ and $0$ otherwise by the reciprocity formula~\eqref{eq:reciprocity formula} and Proposition~\ref{prop:the socle of the standard objects}. Then, using~\eqref{eqn:multiplicities-R}, Proposition~\ref{prop:ringel duality} and the fact that
$\D(\rT_{w_0}) \cong \rT_{w_0}$, we obtain \eqref{eq:1 on tilting w0}. In particular, we deduce that for all $x \in W$ we have
\begin{equation}
\label{eq:2 on tilting w0}
\dim_\bk \Hom_{\fP^{\RE}(\h,W)}\left(\rT_{w_0},\onabla_{x}\langle n\rangle\right)=\begin{cases}
                                                                                1 & \text{if $n=\ell(xw_0)$;}\\
                                                                                0 & \mbox{otherwise.}
                                                                               \end{cases}
\end{equation}

We now claim that any nonzero morphism $f:\rT_{w_0}\to\onabla_{x}\langle\ell(xw_0)\rangle$ is surjective. In fact, since the corresponding $\Hom$-space is $1$-dimensional, it suffices to prove that there exists a surjective morphism from $\rT_{w_0}$ to $\onabla_{x}\langle\ell(xw_0)\rangle$. Such a morphism is provided by the composition
\[
 \rT_{w_0} \twoheadrightarrow \onabla_{w_0} \twoheadrightarrow \onabla_{x}\langle\ell(xw_0)\rangle
\]
where the first morphism is given by~\eqref{eqn:surj-T-nabla} and the second one is provided by Proposition~\ref{prop:morphisms between costandard objects}.

This claim implies that $\rT_{w_0}$ has no quotient of the form $\rL_{y} \langle n \rangle$ with $y \neq e$, since otherwise we would obtain a nonzero and nonsurjective morphism $\rT_{w_0} \to \onabla_y \langle n \rangle$ as the composition
\[
 \rT_{w_0} \twoheadrightarrow \rL_{y} \langle n \rangle \hookrightarrow \onabla_{y} \langle n \rangle.
\]
In view of~\eqref{eq:2 on tilting w0}, we deduce that the head of $\rT_{w_0}$ is $\rL_e \langle \ell(w_0) \rangle$, and hence that there exists a surjective morphism
\[
 \rP_e \langle \ell(w_0) \rangle \twoheadrightarrow \rT_{w_0}.
\]
Since these objects have the same length (namely, the sum of the lengths of all objects $\oD_x$ with $x \in W$), this surjection must be an isomorphism, which proves the first isomorphism in~\eqref{eqn:isom-Pe-Tw0-Ie}. The second isomorphism follows by duality.
\end{proof}


\end{document}